\documentclass[10pt]{article}

\usepackage{enumerate,xspace}
\usepackage{amsmath,amssymb,wasysym}
\usepackage[all]{xy}
\usepackage{proof}
\usepackage[svgnames]{xcolor}
\usepackage{pict2e} 
\usepackage{tikz}
\usepackage{stmaryrd} 
\usepackage{mathtools}
\usepackage{latexsym}
\usepackage{dsfont}
\usepackage{multicol}
\usepackage{cmll}
\usepackage{multirow}
\usepackage{longtable}

\usepackage[sort,nocompress]{cite} 

\usepackage{lscape}
\usepackage{array}

\usepackage{hyperref} 
\hypersetup{
    colorlinks,
    citecolor=red,
    filecolor=red,
    linkcolor=blue,
    urlcolor=red
}

\newtheorem{observation}{Remark}[section]
\newtheorem{lemma}[observation]{Lemma}  
\newtheorem{theorem}[observation]{Theorem}
\newtheorem{definition}[observation]{Definition}

\newtheorem{proposition}[observation]{Proposition} 
\newtheorem{corollary}[observation]{Corollary}

\newcommand{\wand}{\ensuremath{
  \mathrel{\vbox{\offinterlineskip\ialign{
    \hfil##\hfil\cr
    $\star$\cr
    \noalign{\kern-1ex}
    $\vert$\cr
}}}}}

\makeatletter

\newdimen\w@dth

\def\setw@dth#1#2{\setbox\z@\hbox{\scriptsize $#1$}\w@dth=\wd\z@
\setbox\@ne\hbox{\scriptsize $#2$}\ifnum\w@dth<\wd\@ne \w@dth=\wd\@ne \fi
\advance\w@dth by 1.2em}

\def\t@^#1_#2{\allowbreak\def\n@one{#1}\def\n@two{#2}\mathrel
{\setw@dth{#1}{#2}
\mathop{\hbox to \w@dth{\rightarrowfill}}\limits
\ifx\n@one\empty\else ^{\box\z@}\fi
\ifx\n@two\empty\else _{\box\@ne}\fi}}
\def\t@@^#1{\@ifnextchar_ {\t@^{#1}}{\t@^{#1}_{}}}

\def\t@left^#1_#2{\def\n@one{#1}\def\n@two{#2}\mathrel{\setw@dth{#1}{#2}
\mathop{\hbox to \w@dth{\leftarrowfill}}\limits
\ifx\n@one\empty\else ^{\box\z@}\fi
\ifx\n@two\empty\else _{\box\@ne}\fi}}
\def\t@@left^#1{\@ifnextchar_ {\t@left^{#1}}{\t@left^{#1}_{}}}

\def\two@^#1_#2{\def\n@one{#1}\def\n@two{#2}\mathrel{\setw@dth{#1}{#2}
\mathop{\vcenter{\hbox to \w@dth{\rightarrowfill}\kern-1.7ex
                 \hbox to \w@dth{\rightarrowfill}}%
       }\limits
\ifx\n@one\empty\else ^{\box\z@}\fi
\ifx\n@two\empty\else _{\box\@ne}\fi}}
\def\tw@@^#1{\@ifnextchar_ {\two@^{#1}}{\two@^{#1}_{}}}

\def\tofr@^#1_#2{\def\n@one{#1}\def\n@two{#2}\mathrel{\setw@dth{#1}{#2}
\mathop{\vcenter{\hbox to \w@dth{\rightarrowfill}\kern-1.7ex
                 \hbox to \w@dth{\leftarrowfill}}%
       }\limits
\ifx\n@one\empty\else ^{\box\z@}\fi
\ifx\n@two\empty\else _{\box\@ne}\fi}}
\def\t@fr@^#1{\@ifnextchar_ {\tofr@^{#1}}{\tofr@^{#1}_{}}}

\newdimen\W@dth
\def\setW@dth#1#2{\setbox\z@\hbox{$#1$}\W@dth=\wd\z@
\setbox\@ne\hbox{$#2$}\ifnum\W@dth<\wd\@ne \W@dth=\wd\@ne \fi
\advance\W@dth by 1.2em}

\def\T@^#1_#2{\allowbreak\def\N@one{#1}\def\N@two{#2}\mathrel
{\setW@dth{#1}{#2}
\mathop{\hbox to \W@dth{\rightarrowfill}}\limits
\ifx\N@one\empty\else ^{\box\z@}\fi
\ifx\N@two\empty\else _{\box\@ne}\fi}}
\def\T@@^#1{\@ifnextchar_ {\T@^{#1}}{\T@^{#1}_{}}}

\def\T@left^#1_#2{\def\N@one{#1}\def\N@two{#2}\mathrel{\setW@dth{#1}{#2}
\mathop{\hbox to \W@dth{\leftarrowfill}}\limits
\ifx\N@one\empty\else ^{\box\z@}\fi
\ifx\N@two\empty\else _{\box\@ne}\fi}}
\def\T@@left^#1{\@ifnextchar_ {\T@left^{#1}}{\T@left^{#1}_{}}}

\def\Tofr@^#1_#2{\def\N@one{#1}\def\N@two{#2}\mathrel{\setW@dth{#1}{#2}
\mathop{\vcenter{\hbox to \W@dth{\rightarrowfill}\kern-1.7ex
                 \hbox to \W@dth{\leftarrowfill}}%
       }\limits
\ifx\N@one\empty\else ^{\box\z@}\fi
\ifx\N@two\empty\else _{\box\@ne}\fi}}
\def\T@fr@^#1{\@ifnextchar_ {\Tofr@^{#1}}{\Tofr@^{#1}_{}}}

\def\Two@^#1_#2{\def\N@one{#1}\def\N@two{#2}\mathrel{\setW@dth{#1}{#2}
\mathop{\vcenter{\hbox to \W@dth{\rightarrowfill}\kern-1.7ex
                 \hbox to \W@dth{\rightarrowfill}}%
       }\limits
\ifx\N@one\empty\else ^{\box\z@}\fi
\ifx\N@two\empty\else _{\box\@ne}\fi}}
\def\Tw@@^#1{\@ifnextchar_ {\Two@^{#1}}{\Two@^{#1}_{}}}

\def\to{\@ifnextchar^ {\t@@}{\t@@^{}}}
\def\from{\@ifnextchar^ {\t@@left}{\t@@left^{}}}
\def\tofro{\@ifnextchar^ {\t@fr@}{\t@fr@^{}}}
\def\To{\@ifnextchar^ {\T@@}{\T@@^{}}}
\def\From{\@ifnextchar^ {\T@@left}{\T@@left^{}}}
\def\Two{\@ifnextchar^ {\Tw@@}{\Tw@@^{}}}
\def\Tofro{\@ifnextchar^ {\T@fr@}{\T@fr@^{}}}

\makeatother

\title{Characterizing (Co)Free Dagger Categories}
\author{Jean-Simon Pacaud Lemay}

\begin{document}
\allowdisplaybreaks

\maketitle

\begin{abstract} For any category, there exists both a free dagger category and a cofree dagger category over it. A natural question to ask is: given a dagger category, how can we tell if it is free or cofree without specifying an external base category? In this paper, we provide characterizations of both free dagger categories and cofree dagger categories via internal dagger category structure. To characterize cofree dagger categories, we use rectangular bands and show that a dagger category is cofree if and only if it is enriched over rectangular bands. For free dagger categories, we define the notion of a zigzag dagger category, and then show that a dagger category is free if and only if it is a zigzag dagger category. We also show that free dagger categories can be characterized as the coalgebras of the induced comonad from the free dagger category adjunction, and similarly that cofree free dagger categories can be characterized as the algebras of the induced monad from the cofree dagger category adjunction. \end{abstract}



\section{Introduction}

The theory of dagger categories is now a firmly established branch of category theory with a deep and rich literature, and plays a central role in categorical quantum mechanics \cite{heunen2009categorical}. There are many interesting examples of dagger categories including the category of (finite dimensional) Hilbert spaces, the category of matrices over a (involutive) ring, the category of relations, the category of partial injections, groupoids, and one object (small) dagger categories are precisely involutive monoids. 

There has been a recent research program of characterizing certain important dagger categories using only dagger category theoretic axioms. Indeed, famously, Lawvere characterized the category of sets in showing that if a category satisfied a certain reasonable list of axioms, then it must be equivalent to the category of sets \cite{lawvere1964elementary}. Heunen and Kornell did the same for the category of Hilbert spaces from a dagger category theory perspective, in that they showed that if a dagger category satisfied a list of purely dagger categorical axioms, then it must equivalent (as a dagger category) to the category of Hilbert spaces \cite{heunen2022axioms}. This has led to characterizations of other key dagger categories such as the category of (finite dimensional) Hilbert spaces and linear contractions \cite{di2025dagger,https://doi.org/10.1112/blms.13010}, the category of sets and relations \cite{kornell2023axioms}, as well as alternative approaches which capture not only the category of complex Hilbert spaces, but also the category of real Hilbert spaces and the category of quaternionic Hilbert spaces \cite{paseka2025dagger,lack2025characterisation}. This paper follows along this line of research and aims to provide characterizations of two special kinds of dagger categories: \textit{free} dagger categories and \textit{cofree} dagger categories. 

Indeed, the forgetful functor from the (large) category of categories to the (large) category of dagger categories has both a left adjoint and a right adjoint \cite{heunen2009categorical}. As such, for any category, there exists both a free dagger category and cofree dagger category over it. The free dagger category is built by concatenating spans of maps of the base category giving us ``zigzags" (Sec \ref{sec:free-canonical}). Free dagger categories have been used to show completeness of finite dimensional Hilbert spaces \cite{selinger2012finite}, formulating types for quantum computing \cite{duncan2006types}, describing reversible information effects \cite{heunen2018reversible}, and have been implemented in Distributional Compositional Python (DisCoPy) \cite{toumi2023discopy}. Moreover, the category of relations and the category of partial injections are both quotients of the free dagger category over the category of sets \cite{heunen2009categorical}. On the other hand, the cofree dagger is simply built using pairs of maps going in opposite directions of the base category (Sec \ref{sec:cofree-canonical}). Cofree categories have been used to characterize ambilimits \cite{heunen2018limits,karvonen2019way}, build dagger adjoints from Frobenius functors \cite{heunen2016monads}, for categorical semantics of reversible computing \cite{jacobs2009categorical}, and generalizing Drazin inverses in dagger categories \cite{cockett2025dagger}. 

The objective of this paper is to answer the question: given a dagger category, how can we check whether it is (co)free without specifying an external base category? As such, we provide dagger categoric axioms on a dagger category which characterize precisely when it is (co)free. Using these axioms, we are able to extract a base category and show that our dagger category is in fact (co)free over it. 

Starting with cofree dagger categories (Sec \ref{sec:cofree}), we show that these can be characterized using rectangular bands\footnote{We thank Steve Lack for pointing us towards this terminology for rectangular bands.} \cite{clifford1954bands,kimura1958structure}, which are idempotent semigroups satisfying a special cancellation axiom. We define what is means for a dagger category to be enriched over rectangular bands (Def \ref{def:rect-band-dagger-cat}) and then show that this is equivalent to being cofree (Thm \ref{thm:cofree=rect}). Here, the base category is extracted by quotienting the rectangular band enriched dagger category by Green's $\mathcal{L}$-relation (Lemma \ref{lemma:DL}). 

To characterize free dagger categories (Sec \ref{sec:free}), we define a zigzag dagger category as a dagger category with a chosen special class of maps, called zigs and whose adjoints are called zag, such that every map decomposes into an alternating list of zigs and zags (Def \ref{def:zigzag-dag-cat}). We then show that zigzag dagger categories correspond precisely to free dagger categories (Thm \ref{thm:free=zigzag}). This time the base category is the subcategory of zig maps (Prop \ref{prop:zigzag-to-free}). 

Furthermore, we also characterize cofree dagger categories as monad algebras (Sec \ref{sec:alg}) and free dagger categories as comonad coalgebras (\ref{sec:coalg}). Indeed, the adjunction given by the free dagger category construction induces a comonad on the category of dagger categories, and we show that the coalgebras of this comonad are precisely zigzag dagger categories, hence free dagger categories (Thm \ref{thm:full-free}). Similarly, the adjunction given by the cofree dagger category construction this time induces a monad on the category of dagger categories, and we show that the algebras of this monad are precisely rectangular band enriched dagger categories, hence cofree dagger categories (Thm \ref{thm:full-cofree}). 

\section{Notation for Dagger Categories}

While we assume that the reader is familiar with the basics of category theory, including categories and functors, and has some familiarity with dagger categories, we have written this paper to be accessible to those with only basic knowledge of dagger categories. In this section, we briefly recall the definition of dagger categories and dagger functors to setup the notation and terminology that we will use in this paper. For an in-depth introduction to dagger categories, we refer the reader to \cite{heunen2009categorical,karvonen2019way}. 

Arbitrary categories will be denoted by capital letters using the font $\mathbb{C}$, $\mathbb{D}$, $\mathbb{B}$, etc. For an arbitrary category $\mathbb{C}$, we denote its class of objects as $\mathsf{Obj}(\mathbb{C})$, where we write objects using capital letters using the font $A,B,C$, etc., and denote its class of maps $\mathsf{Map}(\mathbb{X})$, where we write maps using miniscule letter using the font $f,g,h$, etc. Homsets will be denoted as $\mathbb{C}(A,B)$ and a map $f\in \mathbb{C}(A,B)$ will be denoted as an arrow ${f: A \to B}$. Identity maps will be denoted as $\mathsf{id}_A: A \to A$ (where we may sometimes drop the subscript and simply write $\mathsf{id}$ when there is no confusion) and we use classical applicative notation for composition $\circ$. Arbitrary functors will be denoted by capital letters using the font $\mathcal{F}$ and will also be denoted by arrows $\mathcal{F}: \mathbb{C} \to \mathbb{D}$. Application of a functor $\mathcal{F}$ will be denoted as $\mathcal{F}(-)$, where recall that a functor assigns every object $A$ to an object $\mathcal{F}(A)$, every map $f: A \to B$ to a map ${\mathcal{F}(f): \mathcal{F}(A) \to \mathcal{F}(B)}$, and is also required to preserve composition and identities: 
\begin{align}\label{def:functor}
\mathcal{F}(\mathsf{id}_A) = \mathsf{id}_{\mathcal{F}(A)} && \mathcal{F}(g \circ f) = \mathcal{F}(g) \circ \mathcal{F}(f)
\end{align}
We will also use $\circ$ for the composition of functors and write $\mathcal{G}\mathcal{F}(-)$ for application of the composite functor $\mathcal{G} \circ \mathcal{F}$. For an arbitrary category $\mathbb{C}$, we will denote its identity functor by $\mathsf{id}_\mathbb{C}: \mathbb{C} \to \mathbb{C}$. 

Now recall that we say that a functor ${\mathcal{F}: \mathbb{C} \to \mathbb{D}}$ is an \textbf{isomorphism of categories}, or simply an isomorphism, if there is a functor $\mathcal{F}^{-1}: \mathbb{D} \to \mathbb{C}$ such that $\mathcal{F} \circ \mathcal{F}^{-1} = \mathsf{id}_\mathbb{C}$ and $\mathcal{F}^{-1} \circ \mathcal{F} = \mathsf{id}_\mathbb{D}$. As a shorthand, we will write $\mathbb{C} \cong \mathbb{D}$ to say that there is an isomorphism of categories between these categories, and $\mathbb{C} \simeq \mathbb{D}$ if there is simply en equivalence between them. Also recall that we say that a functor is ${\mathcal{F}: \mathbb{C} \to \mathbb{D}}$ is full if it is surjective on maps, faithful if it is injective on maps, and essentially surjective if every object in $\mathbb{D}$ is isomorphic to an application of $\mathcal{F}$ to an object of $\mathbb{C}$. Then recall that $\mathcal{F}: \mathbb{C} \to \mathbb{D}$ induces an equivalence when it is full, faithful, and essentially surjective. 

Now a \textbf{dagger} on a category $\mathbb{D}$ is a contravariant functor $\dagger: \mathbb{D} \to \mathbb{D}$ which is the identity on objects and involutive. Equivalently, a dagger can be described as associating each map $f: A \to B$ to a map of dual type ${f^\dagger: B \to A}$, called the \textbf{adjoint} of $f$, and such that the following equalities hold: 
\begin{align}\label{eq:dagger}
\mathsf{id}_A^\dagger= \mathsf{id}_A && (g \circ f)^\dagger = f^\dagger \circ g^\dagger && f^{\dagger\dagger} = f 
\end{align}
Then a \textbf{dagger category} is a category equipped with a chosen dagger. We will abuse notation slightly and simply use $\dagger$ for the dagger of all our dagger categories. To distinguish between dagger categories and mere categories, we will denote dagger categories as pairs $(\mathbb{D}, \dagger)$, where $\mathbb{D}$ is the underlying category. For lists of many interesting examples of dagger categories, see \cite[Ex 3.1]{heunen2009categorical} and \cite[Ex 2.1.2]{karvonen2019way}. 

For dagger categories $(\mathbb{D}_1, \dagger)$ and $(\mathbb{D}_2, \dagger)$, a \textbf{dagger functor} ${\mathcal{F}: (\mathbb{D}_1, \dagger) \to (\mathbb{D}_2, \dagger)}$ is a functor between the underlying categories $\mathcal{F}: \mathbb{D}_1 \to \mathbb{D}_2$ which commutes with the dagger: 
\begin{align}\label{def:dagfun}
\mathcal{F}(f^\dagger) = \mathcal{F}(f)^\dagger
\end{align}
By a \textbf{dagger isomorphism} we mean a dagger functor ${\mathcal{F}: (\mathbb{D}_1, \dagger) \to (\mathbb{D}_2, \dagger)}$ such that $\mathcal{F}$ is also an isomorphism of the underlying categories. In this case, it comes for free that $\mathcal{F}^{-1}: (\mathbb{D}_2, \dagger) \to (\mathbb{D}_1, \dagger)$ is also a dagger functor. As shorthand, we will often simply write $(\mathbb{D}_1, \dagger) \cong (\mathbb{D}_2, \dagger)$ to say that there is a dagger isomorphism between these dagger categories. 

Lastly, we will denote $\mathsf{CAT}$ to be the (large) category of categories and functors between them, $\mathsf{DAG}$ to be the (large) category of dagger categories and dagger functors between them, and $\mathsf{U}: \mathsf{DAG} \to \mathsf{CAT}$ to be the forgetful functor, which is explicitly defined on objects (dagger categories) and maps (dagger functors) as follows: 
\begin{align}
\mathsf{U}(\mathbb{D}, \dagger) = \mathbb{D} && \mathsf{U}(\mathcal{F}) = \mathcal{F}
\end{align}

\section{Cofree Dagger Categories}\label{sec:cofree}

In this section, we give base independent characterizations of cofree daggers categories. First, let us review the definition of a cofree dagger category via its universal property over a base category. 

\begin{definition} A \textbf{cofree dagger category} over a category $\mathbb{B}$ is a triple $(\mathbb{D}, \dagger, \mathcal{E})$ consisting of a dagger category $(\mathbb{D}, \dagger)$ and a functor $\mathcal{E}: \mathbb{D} \to \mathbb{B}$, such that for every dagger category $(\mathbb{C}, \dagger)$ and functor $\mathcal{F}: \mathbb{C} \to \mathbb{B}$, there exists a unique dagger functor $\mathcal{F}^\flat: (\mathbb{C}, \dagger) \to (\mathbb{D}, \dagger)$ such that the following diagram commutes: 
\begin{equation}\begin{gathered}\label{diag:cofree}  \xymatrixcolsep{5pc}\xymatrix{\mathbb{C}\ar[dr]_-{\mathcal{F}}  \ar@{-->}[r]^-{\exists! ~ \mathcal{F}^\flat}  & \mathbb{D} \ar[d]^-{\mathcal{E}} \\
  & \mathbb{B} }
\end{gathered}\end{equation}
A dagger category $(\mathbb{D}, \dagger)$ is said to be \textbf{cofree} if there exists a category $\mathbb{B}$ and a functor $\mathcal{E}: \mathbb{D} \to \mathbb{B}$ such that $(\mathbb{D}, \dagger, \mathcal{E})$ is a cofree dagger category over $\mathbb{B}$. In this case, we call $\mathbb{B}$ a \textbf{base category} of the cofree dagger category $(\mathbb{D}, \dagger)$.
\end{definition}

As we will review shortly below in Sec \ref{sec:cofree-canonical}, for any category $\mathbb{B}$, there exists a cofree dagger category over it. Before doing so, we record three important facts about cofree dagger categories that follow immediately from the above universal property. The first is that any cofree dagger categories over the same base category are dagger isomorphic. The second is that dagger isomorphisms preserve being cofree. The third is that being a base of a cofree dagger category is also preserved by isomorphisms. 

\begin{lemma}\label{lemma:cofree-iso} Let $\mathbb{B}$ be a category. 
\begin{enumerate}[{\em (i)}]
\item If $(\mathbb{D}_1, \dagger, \mathcal{E}_1)$ and $(\mathbb{D}_2, \dagger, \mathcal{E}_2)$ are both cofree dagger categories over $\mathbb{B}$, then there exists a unique dagger isomorphism ${\mathcal{F}: (\mathbb{D}_1, \dagger) \to (\mathbb{D}_2, \dagger)}$ such that the following diagram commutes: 
\begin{equation}\begin{gathered}\label{diag:cofree-iso}
\xymatrixcolsep{5pc}\xymatrix{\mathbb{D}_1 \ar[dr]_-{\mathcal{E}_1}  \ar@{-->}[rr]^-{\exists! ~ \mathcal{F}}  && \mathbb{D}_2 \ar[dl]^-{\mathcal{E}_2} \\
  & \mathbb{B} }
  \end{gathered}\end{equation}
\item If $(\mathbb{D}_1, \dagger, \mathcal{E}_1)$ is a cofree dagger category over $\mathbb{B}$ and ${\mathcal{F}: (\mathbb{D}_1, \dagger) \to (\mathbb{D}_2, \dagger)}$ is a dagger isomorphism, then $(\mathbb{D}_2, \dagger, \mathcal{E}_1 \circ \mathcal{F}^{-1})$ is a cofree dagger category over $\mathbb{B}$. 
\item If $(\mathbb{D}, \dagger, \mathcal{E})$ is a cofree dagger category over $\mathbb{B}$ and $\mathcal{F}: \mathbb{B} \to \mathbb{B}_2$ is an isomorphism, then $(\mathbb{D}, \dagger, \mathcal{F} \circ \mathcal{E})$ is a cofree dagger category over $\mathbb{B}_2$. 
\end{enumerate}
\end{lemma}

However, it is important to stress that a dagger category can be cofree over different non-equivalent base categories. Indeed, a simple example comes from the fact that a cofree dagger category over a specified based is also cofree over the opposite category of its specified base category (and of course their are many examples of categories which are not equivalent to their opposite category). 

\begin{lemma}\label{lemma:cofree-opp-base} Let $(\mathbb{D}, \dagger, \mathcal{E})$ be a cofree dagger category over a category $\mathbb{B}$. Define the functor $\mathcal{E}^o: \mathbb{D} \to \mathbb{B}^{op}$ on objects as $\mathcal{E}^o(A) = \mathcal{E}(A)$ and on maps as $\mathcal{E}^o(f) = \mathcal{E}(f^\dagger)$. Then $(\mathbb{D}, \dagger, \mathcal{E}^o)$ is a cofree dagger category over $\mathbb{B}^{op}$.
\end{lemma}
\begin{proof} First note that for every dagger category $(\mathbb{D}, \dagger)$, its opposite category is also a dagger category $(\mathbb{D}^{op}, \dagger)$ with the same dagger, and (abusing notation slightly) the dagger becomes a dagger isomorphism $\dagger: (\mathbb{D}, \dagger) \to (\mathbb{D}^{op}, \dagger)$. Also recall that for every functor $\mathcal{E}: \mathbb{D} \to \mathbb{B}$, we get a functor on the opposite categories $\mathcal{E}^{op}: \mathbb{D}^{op} \to \mathbb{B}^{op}$. Observe that $\mathcal{E}^o$ is then precisely the composite $\mathcal{E}^o := \mathcal{E}^{op} \circ \dagger$, and thus a well-defined functor. Now let $(\mathbb{C}, \dagger)$ be a dagger category and suppose we have a functor $\mathcal{F}: \mathbb{C} \to \mathbb{B}^{op}$. Thus we have a functor $\mathcal{F}^{op}: \mathbb{C}^{op} \to \mathbb{B}$ (since recall that ${\mathbb{C}^{op}}^{op} = \mathcal{C}$). Therefore we have a dagger functor ${\mathcal{F}^{op}}^\flat: (\mathbb{C}^{op}, \dagger) \to (\mathbb{D}, \dagger)$. Precomposing it by the dagger gives us a dagger functor ${\mathcal{F}^{op}}^\flat \circ \dagger: (\mathbb{C}, \dagger) \to (\mathbb{D}, \dagger)$. From here, it is straightforward to check that $\mathcal{E}^o \circ {\mathcal{F}^{op}}^\flat \circ \dagger = \mathcal{F}$ and it is the unique such dagger functor. Thus we conclude that $(\mathbb{D}, \dagger, \mathcal{E}^o)$ is a cofree dagger category over $\mathbb{B}^{op}$. 
\end{proof}

Another example illustrating this, which will play an important role below, comes from the fact that a cofree dagger category over a specified base category is also cofree over the image category of the functor down to its specified base. So let $(\mathbb{D},\dagger, \mathcal{E})$ be a cofree dagger category over a category $\mathbb{B}$. Let $\mathsf{Im}[\mathcal{E}]$ be the image category of $\mathcal{E}$, that is, the category whose objects are of form $\mathcal{E}(A)$ for all $A \in \mathsf{Obj}(\mathbb{D})$ and whose maps are of the form ${\mathcal{E}(f): \mathcal{E}(A) \to \mathcal{E}(B)}$ for all $f: A \to B \in \mathsf{Map}(\mathbb{D})$. Let $\mathcal{I}_\mathcal{E}: \mathsf{Im}[\mathcal{E}] \to \mathbb{B}$ be the inclusion functor and $\overline{\mathcal{E}}: \mathbb{D} \to \mathsf{Im}[\mathcal{E}]$ be the obvious image functor, so we have that $\mathcal{E} = \mathcal{I}_\mathcal{E} \circ \overline{\mathcal{E}}$. 

\begin{lemma}\label{lemma:cofree-im-base} Let $(\mathbb{D},\dagger, \mathcal{E})$ be a cofree dagger category over a category $\mathbb{B}$. Then $(\mathbb{D},\dagger, \overline{\mathcal{E}})$ is a cofree dagger category over $\mathsf{Im}[\mathcal{E}]$. 
\end{lemma}
\begin{proof} Let  $(\mathbb{C}, \dagger)$ be a dagger category and suppose we have a functor $\mathcal{F}: \mathbb{C} \to \mathsf{Im}[\mathcal{E}]$. Then we have functor $\mathcal{I} \circ \mathcal{F}: \mathbb{C} \to \mathbb{B}$, and thus we obtain a dagger functor $(\mathcal{I} \circ \mathcal{F})^\flat: (\mathbb{C}, \dagger) \to (\mathbb{D}, \dagger)$. From here, it is straightforward to check that $\overline{\mathcal{E}} \circ (\mathcal{I} \circ \mathcal{F})^\flat = \mathcal{F}$ and and it is the unique such dagger functor. Thus we conclude that $(\mathbb{D},\dagger, \overline{\mathcal{E}})$ is a cofree dagger category over $\mathsf{Im}[\mathcal{E}]$.
\end{proof}

It is important to stress that in general, $\mathcal{E}$ is neither full nor faithful. Therefore, $\mathsf{Im}[\mathcal{E}]$ is not in general isomorphic (or even equivalent) to $\mathbb{B}$, and is only a proper wide subcategory of $\mathbb{B}$. 


However, observe that in both Lemmas \ref{lemma:cofree-opp-base} and \ref{lemma:cofree-im-base}, the objects of our constructed other base categories are the same as the starting base category. Indeed, while the possible base categories may be far from isomorphic, we will see below in Lemma \ref{lemma:cofree-objects} that they all have isomorphic classes of objects. In fact, we will also see below how the class of objects of a cofree dagger category is isomorphic to that of its base category as well. Thus while cofree dagger categories preserve all the information about the objects of its base category, it may loose information about the maps of its base category. 

\subsection{Canonical Construction}\label{sec:cofree-canonical}

We now review the canonical construction of a cofree dagger category over a specified category. So for a category $\mathbb{B}$, define the dagger category $(\mathsf{C}[\mathbb{B}], \dagger)$ \cite[Def 3.1.16]{heunen2009categorical} as follows: 
\begin{itemize}
\item The objects of $\mathsf{C}[\mathbb{B}]$ are the same as $\mathbb{B}$;
\item Maps in $\mathsf{C}[\mathbb{B}]$ are pairs $(f,g): A \to B$ consisting of a map $f: A \to B$ and a map $g: B \to A$ in $\mathbb{B}$, in other words, the homsets of $\mathsf{C}[\mathbb{B}]$ are $\mathsf{C}[\mathbb{B}](A,B) = \mathbb{B}(A,B) \times \mathbb{B}(B,A)$;
\item Composition is defined as follows: 
\begin{align}\label{def:comp-cofree}
(f,g) \circ (h,k) = (f \circ h, k \circ g)
\end{align}
\item Identity maps are $(\mathsf{id}_A, \mathsf{id}_A): A \to A$;
\item The dagger is defined as follows:
\begin{align}\label{def:dag-cofree}
(f,g)^\dagger= (g,f)
\end{align}
\end{itemize}
Define the functor $\mathcal{E}_\mathbb{B}: \mathsf{C}[\mathbb{B}] \to \mathbb{B}$ on objects as $\mathcal{E}_\mathbb{B}(A) = A$ and on maps as $\mathcal{E}_\mathbb{B}(f,g) = f$. 

\begin{proposition}\label{prop:canon-cofree} \cite[Thm 3.1.17]{heunen2009categorical} $(\mathsf{C}[\mathbb{B}], \dagger, \mathcal{E}_\mathbb{B})$ is a cofree dagger category over $\mathbb{B}$. Explicitly, for any dagger category $(\mathbb{C}, \dagger)$ and functor $\mathcal{F}: \mathbb{C} \to \mathbb{B}$, the unique dagger $\mathcal{F}^\flat: (\mathbb{C}, \dagger) \to (\mathsf{C}[\mathbb{B}], \dagger)$ such that the following diagram commutes: 
\begin{equation}\begin{gathered}\label{diag:cofree-canon} \xymatrixcolsep{5pc}\xymatrix{\mathbb{C}\ar[dr]_-{\mathcal{F}}  \ar@{-->}[r]^-{\exists! ~ \mathcal{F}^\flat}  & \mathsf{C}[\mathbb{B}] \ar[d]^-{\mathcal{E}_\mathbb{B}} \\
  & \mathbb{B} }
   \end{gathered}\end{equation}
is defined on objects as $\mathcal{F}^\flat(A) = \mathcal{F}(A)$ and on maps as $\mathcal{F}^\flat(f) = (f,f^\dagger)$. 
\end{proposition}

As such, applying Lemma \ref{lemma:cofree-iso}, we get that cofree dagger categories can be characterized as the dagger categories which are dagger isomorphic to canonical cofree dagger categories. 

\begin{corollary}\label{cor:cofree-iso.1} If $(\mathbb{D}, \dagger, \mathcal{E})$ is a cofree dagger category over a category $\mathbb{B}$, then $\mathcal{E}^\flat: (\mathbb{D}, \dagger) \to (\mathsf{C}[\mathbb{B}], \dagger)$ is a dagger isomorphism with inverse $\mathcal{E}^\flat_\mathbb{B}: (\mathsf{C}[\mathbb{B}], \dagger) \to (\mathbb{D}, \dagger)$. So $(\mathbb{D}, \dagger) \cong (\mathsf{C}[\mathbb{B}], \dagger)$. 
\end{corollary}

\begin{corollary}\label{cor:cofree-iso.2} A dagger category $(\mathbb{D}, \dagger)$ is cofree if and only if there exists a category $\mathbb{B}$ and a dagger isomorphism $(\mathbb{D}, \dagger) \cong (\mathsf{C}[\mathbb{B}], \dagger)$. 
\end{corollary}

The above characterization still requires one to specify a base category. The objective of the next two subsections is to provide base independent characterizations of cofree dagger categories. 

With this canonical construction of cofree dagger categories, we may now show that arbitrary cofree dagger categories essentially have the same objects as its base category. In turn this implies that all possible base categories have isomorphic classes of objects. 

\begin{lemma}\label{lemma:cofree-objects} Let $(\mathbb{D}, \dagger)$ be a dagger category. 
\begin{enumerate}[{\em (i)}]
\item If $(\mathbb{D}, \dagger, \mathcal{E})$ is a cofree dagger category over a category $\mathbb{B}$, then $\mathcal{E}: \mathsf{Obj}(\mathbb{D}) \to \mathsf{Obj}(\mathbb{B})$ is an isomorphism, so we have that $\mathsf{Obj}(\mathbb{B}) \cong \mathsf{Obj}(\mathbb{D})$.
\item If $(\mathbb{D}, \dagger, \mathcal{E}_1)$ is a cofree dagger category over a category $\mathbb{B}_1$ and $(\mathbb{D}, \dagger, \mathcal{E}_2)$ is also a cofree dagger category over a category $\mathbb{B}_2$, then we have an isomorphism $\mathsf{Obj}(\mathbb{B}_1) \cong \mathsf{Obj}(\mathbb{B}_2)$.
\end{enumerate}
\end{lemma}
\begin{proof} For $(i)$, since $(\mathbb{D}, \dagger, \mathcal{E})$ is cofree over $\mathbb{B}$, we have that $(\mathbb{D}, \dagger) \cong (\mathsf{C}[\mathbb{B}], \dagger)$. In particular, since isomorphisms of categories gives isomorphisms between the classes of objects, we have that $\mathsf{Obj}(\mathbb{D}) \cong \mathsf{Obj}\left( \mathsf{C}[\mathbb{B}] \right)$. However, by definition we have that $\mathsf{Obj}\left( \mathsf{C}[\mathbb{B}] \right) = \mathsf{Obj}(\mathbb{B})$, and thus $\mathsf{Obj}(\mathbb{B}) \cong \mathsf{Obj}(\mathbb{D})$ as desired. Then $(ii)$ follows from $(i)$. 
\end{proof}

We conclude this section by applying Lemma \ref{lemma:cofree-im-base} to the canonical cofree dagger category construction and give a more explicit description of the resulting image category. So let $\mathbb{B}$ be a category. Now clearly by definition, we have that $\mathsf{Obj}\left(\mathsf{Im}[\mathcal{E}_\mathbb{B}] \right) = \mathsf{Obj}(\mathbb{B})$. On the other hand, maps of $\mathsf{Im}[\mathcal{E}_\mathbb{B}]$ are of the form $\mathcal{E}_\mathbb{B}(f,g) =f$. As such we see that a map $f: A \to B$ in $\mathbb{B}$ will be a map of $\mathsf{Im}[\mathcal{E}_\mathbb{B}]$ if and only if there exists some map $g: B \to A$. Thus the homsets of $\mathsf{Im}[\mathcal{E}_\mathbb{B}]$ are given as follows: 
\begin{align}
\mathsf{Im}[\mathcal{E}_\mathbb{B}](A,B) = \begin{cases} \mathbb{B}(A,B) & \text{if } \mathbb{B}(B,A) \neq \emptyset \\
\emptyset & \text{if } \mathbb{B}(B,A) = \emptyset
\end{cases}
\end{align}
As mentioned before, cofree dagger categories may loose information about some maps of the base category. We see this clearly now as if $\mathbb{B}(A,B)$ is non-empty but $\mathbb{B}(B,A)$ is empty, we have that $\mathsf{C}[\mathbb{B}](A,B)$ is empty, and thus we cannot extract $\mathbb{B}(A,B)$ from $\mathsf{C}[\mathbb{B}](A,B)$. Furthermore, observe that $\mathsf{C}[\mathbb{B}](A,B) = \mathsf{Im}[\mathcal{E}_\mathbb{B}](A,B) \times \mathsf{Im}[\mathcal{E}_\mathbb{B}](B,A)$. Thus not only is $(\mathsf{C}[\mathbb{B}], \dagger)$ cofree over $\mathsf{Im}[\mathcal{E}_\mathbb{B}]$, it is in fact precisely equal to the canonical cofree dagger category over it.  

\begin{lemma}\label{lemma:cofree-imE=B} Let $\mathbb{B}$ be a category. Then we have have that $(\mathsf{C}[\mathbb{B}], \dagger)=(\mathsf{C}\left[\mathsf{Im}[\mathcal{E}_\mathbb{B}]\right], \dagger)$. 
\end{lemma}

\subsection{Cofree Dagger Categories Via Rectangular Bands}

In this section we will give a ``basis free'' description of cofree dagger categories in terms of enrichment over \textit{rectangular bands}. Rectangular bands were first introduced by Clifford in \cite{clifford1954bands} in terms of cartesian products of sets. Later, Kimura in \cite{kimura1958structure} gave an equivalent definition of a rectangular band as an idempotent semigroup satisfying a special cancellation axiom. We give this latter definition of a rectangular band. 

Recall that a \textbf{semigroup} is a pair $(S, \ast)$ consisting of a set $S$ equipped with an associative binary operation $\ast: S \times S \to S$. A \textbf{band} \cite{clifford1954bands}, also often called an \textbf{idempotent semigroup}, is a semigroup $(B, \ast)$ such that for every element is idempotent, that is, for all $x \in B$ the following equality holds:
\begin{align}\label{eq:ast-idem}
x \ast x = x
\end{align}
Then a \textbf{rectangular band} \cite[Lemma 2]{kimura1958structure} is a band $(B, \ast)$ which satisfies that the \textbf{middle cancellation law}, that is, for all $x,y,z \in B$, the following equality holds: 
\begin{align}\label{eq:ast-cancel}
x \ast y \ast z = x \ast z
\end{align}

The original example of a rectangular band is the cartesian product of sets. Indeed, for any pair of sets $X$ and $Y$, their cartesian product $X \times Y$ is a rectangular band with binary operation $\ast_\times$ defined as:
\begin{align}\label{band-cart}
(x,y) \ast_\times (z,w) = (x,w)
\end{align}
This was Clifford's original definition of a rectangular band. It turns out that every rectangular band is isomorphic to a rectangular band of this form \cite[Chap IV, Prop 3.2]{howie1976introduction}. Let us briefly review how this works, so let $(B,\ast)$ be a rectangular band. If $B$ is empty then this is clear. If $B$ is non-empty, take a $x \in B$ and define $x \ast B = \lbrace x \ast y \vert \forall y \in B \rbrace$ and $B \ast x = \lbrace y \ast x \vert \forall y \in B \rbrace$. Then we have an isomorphism of rectangular bands $B \cong (B \ast x) \times (x \ast B)$, where the isomorphism $B \to (B \ast x) \times (x \ast B)$ sends $y \mapsto \left( (y \ast x), (x \ast y) \right)$, and the inverse $(B \ast x) \times (x \ast B) \to B$ sends $\left( (y_1 \ast x), (x \ast y_2) \right) \mapsto y_1 \ast y_2$. 

Now recall that for a category $\mathbb{B}$, for its canonical cofree dagger category, the homsets of $\mathsf{C}[\mathbb{B}]$ were cartesian products of homsets of $\mathbb{B}$. Therefore, the homsets of $\mathsf{C}[\mathbb{B}]$ are all rectangular bands. In fact this makes $\mathsf{C}[\mathbb{B}](A,B)$ \textit{enriched} over the category of rectangular bands, meaning in particular that both composition and the dagger are compatible with the rectangular band structure. We use this fact as inspiration for our base independent characterization of cofree dagger categories. 

We now define what it means for an arbitrary dagger category to be dagger enriched over rectangular bands, which we call a \textit{rectangular banded dagger category}. We give our following definition in an unpacked concrete fashion, as it is simple enough and does not require background on enriched categories. That said, for those readers who are familiar with enriched category theory, it is easy enough to see that the following definition is indeed equivalent to a dagger category which is dagger enriched over the category of rectangular bands.

\begin{definition}\label{def:rect-band-dagger-cat} A \textbf{rectangular banded dagger category} is a triple $(\mathbb{D}, \dagger, \ast)$ consisting of a dagger category $(\mathbb{D},\dagger)$ equipped with a family of functions (indexed by pairs of objects $A,B \in \mathsf{Obj}(\mathbb{D})$):
\[\ast: \mathbb{D}(A,B) \times \mathbb{D}(A,B) \to \mathbb{D}(A,B)\] 
so every pair of parallel maps $f: A \to B$ and $g: A \to B$, we have a map $f \ast g: A \to B$, such that each homset $(\mathbb{D}(A,B), \ast)$ a rectangular band and also that:
\begin{enumerate}[{\em (i)}]
\item Composition is a band morphism, that is, for all maps $f: C \to D$, $g: B \to C$, $h: B \to C$, and ${k: A \to B}$, the following equality holds: 
\begin{align}\label{band-comp} 
f \circ (g \ast h) \circ k = (f \circ g \circ k) \ast (f \circ h \circ k)
\end{align}
\item Dagger is a band anti-morphism, that is, for all $f: A \to B$ and $g: A \to B$, the following equality holds: 
\begin{align}\label{band-dagger}
(f \ast g)^\dagger = g^\dagger \ast f^\dagger
\end{align}
\end{enumerate}
\end{definition}

It is important to stress that a dagger category can be rectangular banded in multiple ways. Indeed, to see this recall that for a semigroup $(S,\ast)$, its dual semigroup is the semigroup $(S,\ast^{op})$ with the same underlying set $S$ but where the binary operation is defined as $x \ast^{op} y = y \ast x$. In general, $\ast \neq \ast^{op}$, nor are $(S,\ast)$ and $(S,\ast^{op})$ isomorphic as semigroups. Moreover, it easy to see that if $(S,\ast)$ is a rectangular band then $(S,\ast^{op})$ is also a rectangular band. 

\begin{lemma} Let $(\mathbb{D},\dagger, \ast)$ be a rectangular banded dagger category. Then $(\mathbb{D},\dagger, \ast^{op})$ is a rectangular banded dagger category. 
\end{lemma}
\begin{proof} This is straightforward to check. 
\end{proof}

Thus being rectangular banded is additional structure on a dagger category rather than a property. We next observe a useful identity in a rectangular banded dagger category, which is that there is an interchange law between composition and the rectangular band operation. 

\begin{lemma} In a rectangular banded dagger category $(\mathbb{D},\dagger, \ast)$, the following equality holds: 
\begin{align}\label{ast-inter}
(f \ast g) \circ (h \ast k ) =  (f \circ h) \ast (g \circ k)
\end{align}
\end{lemma}
\begin{proof} We compute:
\begin{gather*}(f \ast g) \circ (h \ast k) \overset{\text{(\ref{band-comp})}}{=} \left( (f \ast g) \circ h \right) \ast \left( (f \ast g) \circ k \right) \overset{\text{(\ref{band-comp})}}{=} (f \circ h) \ast (g \circ h) \ast (f \circ k ) \ast ( g \circ k) \\
\overset{\text{(\ref{eq:ast-cancel})}}{=} (f \circ h) \ast (g \circ h)  \ast ( g \circ k) \overset{\text{(\ref{eq:ast-cancel})}}{=}  (f \circ h) \ast ( g \circ k)
\end{gather*}
So the desired identity holds. 
\end{proof}

Our goal is to now show that a rectangular banded dagger categories are in fact cofree dagger categories. To do so, we will need to extract a base category. To achieve this, we turn towards one of Green's relations, specifically the $\mathcal{L}_\ast$-relation \cite[Chap II]{howie1976introduction}. Recall that for a semigroup $(S,\ast)$, we say that $x, y \in S$ are \textbf{$\mathcal{L}_\ast$-related}, written $x ~\mathcal{L}_\ast~ y$, if $x$ and $y$ generate the same principal left ideal. Moreover, $\mathcal{L}_\ast$ is an equivalence relation (in fact a right congruence) \cite[Chap II, Lem 1.2]{howie1976introduction}. For (rectangular) bands, the expression of the $\mathcal{L}_\ast$-relation can be greatly simplified. In a band $(B,\ast)$, $x,y \in B$ are $\mathcal{L}_\ast$-related if and only if $x \ast y=x$ and $y \ast x=y$ \cite[Lem 4]{mclean}. Now in a rectangular band $(B,\ast)$, it turns out that $x \ast y=x$ if and only if $y \ast x=y$. Therefore, in a rectangular band $(B,\ast)$, we have that for all $x,y \in B$:
\begin{align}\label{def:sim}
x ~\mathcal{L}_\ast~ y \Leftrightarrow x \ast y = y \text{ or } y \ast x = x
\end{align}
Now let $B/~\mathcal{L}_\ast~$ be the quotient of $B$ by the $\mathcal{L}_\ast$-relation and let $[x]_{\mathcal{L}_\ast}$ be the equivalence class of $x \in B$. 

Then for a rectangular banded dagger category $(\mathbb{D},\dagger, \ast)$, define the category $\mathbb{D}_{\mathcal{L}_\ast}$ as follows: 
\begin{itemize}
\item The objects of $\mathbb{D}_{\mathcal{L}_\ast}$ are the same as $\mathbb{D}$;
\item The homsets of $\mathbb{D}_{\mathcal{L}_\ast}$ are the quotients of the homsets of $\mathbb{D}$ via $\mathcal{L}_\ast$, that is, $\mathbb{D}_{\mathcal{L}_\ast}(A,B) := \mathbb{D}(A,B)/~\mathcal{L}_\ast~$. So maps are equivalence classes $[f]_{\mathcal{L}_\ast}: A \to B$ of maps $f: A \to B$ in $\mathbb{D}$;
\item Composition is defined as $[g]_{\mathcal{L}_\ast} \circ [f]_{\mathcal{L}_\ast} = [g \circ f]_{\mathcal{L}_\ast}$;
\item Identity maps are $[\mathsf{id}_A]_{\mathcal{L}_\ast}: A \to A$. 
\end{itemize}
Define the functor $\mathcal{E}_{\mathcal{L}_\ast}: \mathbb{D} \to \mathbb{D}_{\mathcal{L}_\ast}$ which sends objects to themselves $\mathcal{E}_{\mathcal{L}_\ast}(A) = A$ and maps to their equivalence classes $\mathcal{E}_{\mathcal{L}_\ast}(f) = [f]_{\mathcal{L}_\ast}$. 

\begin{lemma}\label{lemma:DL} $\mathbb{D}_{\mathcal{L}_\ast}$ is a category and $\mathcal{E}_{\mathcal{L}_\ast}: \mathbb{D} \to \mathbb{D}_{\mathcal{L}_\ast}$ is a functor. 
\end{lemma}
\begin{proof} To show that $\mathbb{D}_{\mathcal{L}_\ast}$ is a category, we need only explain why composition is well-defined. So suppose that we have maps such that $f_1 ~\mathcal{L}_\ast~ f_2$ and $g_1 ~\mathcal{L}_\ast~ g_2$, where $f_i: A \to B$ and $g_i: B \to C$. Then we compute: 
\[ (g_1 \circ f_1) \ast (g_2 \circ f_2) \overset{\text{(\ref{ast-inter})}}{=} (g_1 \ast g_2) \circ (f_1 \ast f_2) \overset{\text{(\ref{def:sim})}}{=} g_1 \circ f_1  \]
So $(g_1 \circ f_1) \ast (g_2 \circ f_2) = g_1 \circ f_1$ and thus $(g_1 \circ f_1) ~\mathcal{L}_\ast~ (g_2 \circ f_2)$. From this, it follows that composition on equivalence classes is well-defined, and therefore we can conclude that $\mathbb{D}_{\mathcal{L}_\ast}$ is indeed a category. It also immediately follows that $\mathcal{E}_{\mathcal{L}_\ast}: \mathbb{D} \to \mathbb{D}_{\mathcal{L}_\ast}$ is indeed a functor as well. 
\end{proof}

\begin{proposition}\label{prop:rect-to-cofree} If $(\mathbb{D},\dagger, \ast)$ is a rectangular banded dagger category, then $(\mathbb{D},\dagger, \mathcal{E}_{\mathcal{L}_\ast})$ is a cofree dagger category over $\mathbb{D}_{\mathcal{L}_\ast}$. 
\end{proposition}
\begin{proof} We need to show that $(\mathbb{D},\dagger, \mathcal{E}_{\mathcal{L}_\ast})$ satisfies the necessary universal property. So let $(\mathbb{C}, \dagger)$ be another dagger category and let $\mathcal{F}: \mathbb{C} \to \mathbb{D}_{\mathcal{L}_\ast}$ be a functor. So for a map $f: X \to Y$ in $\mathbb{C}$, $\mathcal{F}(f)$ is an equivalence class in $\mathbb{D}\left(\mathcal{F}(X),\mathcal{F}(Y)\right)/~\mathcal{L}_\ast~$. So elements $g \in \mathcal{F}(f)$ are maps of type $g:  \mathcal{F}(X) \to \mathcal{F}(Y)$ in $\mathbb{D}$. Note that for the adjoint $f^\dagger: Y \to X$, elements $h \in \mathcal{F}(f)$ are maps of type $h: \mathcal{F}(Y) \to \mathcal{F}(X)$ in $\mathbb{D}$, and so ${h^\dagger: \mathcal{F}(X) \to \mathcal{F}(Y)}$. Now for any $g_1,g_2 \in \mathcal{F}(f)$ (so $g_1~\mathcal{L}_\ast~g_2$) and $h_1,h_2 \in \mathcal{F}(f^\dagger)$ (so $h_1~\mathcal{L}_\ast~h_2$), we compute that: 
\begin{align*}
 g_1 \ast h_1^\dagger \overset{\text{(\ref{def:sim})}}{=} g_2 \ast g_1 \ast  \left( h_2 \ast h_1 \right)^\dagger \overset{\text{(\ref{band-dagger})}}{=}  g_2 \ast g_1 \ast h_1^\dagger \ast h_2^\dagger  \overset{\text{(\ref{eq:ast-cancel})}}{=} g_2 \ast h_1^\dagger \ast h_2^\dagger \overset{\text{(\ref{eq:ast-cancel})}}{=} g_2 \ast h_2^\dagger 
\end{align*}
So for any $g_1,g_2 \in \mathcal{F}(f)$ and $h_1,h_2 \in \mathcal{F}(f^\dagger)$, the following equality holds: 
\begin{align}\label{eq:Fflat-wd}
 g_1 \ast h_1^\dagger =  g_2 \ast h_2^\dagger
\end{align}
Then define the functor $\mathcal{F}^\flat: \mathbb{C} \to \mathbb{D}$ on objects as $\mathcal{F}^\flat(X) = \mathcal{F}(X)$ and on maps as follows:
\begin{align}\label{def:Fflat}
\mathcal{F}^\flat(f) := g \ast h^\dagger
\end{align}
for any $g \in \mathcal{F}(f)$ and $h \in \mathcal{F}(f^\dagger)$. Now (\ref{eq:Fflat-wd}) tells us that (\ref{def:Fflat}) does not depend on the choice of elements we picked from $\mathcal{F}(f)$ and $\mathcal{F}(f^\dagger)$. Therefore, $\mathcal{F}^\flat$ is indeed well-defined. 

Now we check that $\mathcal{F}^\flat$ is a dagger functor: 
\begin{enumerate}[{\em (i)}]
\item Identity: Note that by functoriality of $\mathcal{F}$ we must have that $\mathcal{F}(\mathsf{id}_X) = [\mathsf{id}_{\mathcal{F}(X)}]_{\mathcal{L}_\ast}$. Now since the dagger preserves identities (\ref{eq:dagger}), we have $1^\dagger_X = \mathsf{id}_X$, so $\mathcal{F}(1^\dagger_X)=\mathcal{F}(\mathsf{id}_X) = [\mathsf{id}_{\mathcal{F}(X)}]_{\mathcal{L}_\ast}$. Then taking $\mathsf{id}_{\mathcal{F}(X)} \in \mathcal{F}(\mathsf{id}_X)$ and $\mathsf{id}_{\mathcal{F}(X)} \in \mathcal{F}(1^\dagger_X)$, by definition (\ref{def:Fflat}), we get that $\mathcal{F}^\flat(\mathsf{id}_X) = \mathsf{id}_{\mathcal{F}(X)} \ast \mathsf{id}_{\mathcal{F}(X)}$. However by idempotency (\ref{eq:ast-idem}), $\mathsf{id}_{\mathcal{F}(X)} \ast \mathsf{id}_{\mathcal{F}(X)} = \mathsf{id}_{\mathcal{F}(X)}$. Lastly, since $\mathcal{F}^\flat(X) = \mathcal{F}(X)$, $\mathsf{id}_{\mathcal{F}(X)} = \mathsf{id}_{\mathcal{F}^\flat(X)}$. So we conclude that $\mathcal{F}^\flat(\mathsf{id}_X)= \mathsf{id}_{\mathcal{F}^\flat(X)}$. 
\item Composition: Let $f_1: X \to Y$ and $f_2: Y \to Z$ be composable maps in $\mathbb{C}$. Then let $g_1 \in \mathcal{F}(f_1)$ and $h_1 \in \mathcal{F}(f_1^\dagger)$, and let $g_2 \in \mathcal{F}(f_2)$ and $h_2 \in \mathcal{F}(f_2^\dagger)$. Then we compute that: 
\begin{gather*}
\mathcal{F}^\flat(f_2) \circ  \mathcal{F}^\flat(f_1) \overset{\text{(\ref{def:Fflat})}}{=} (g_2 \ast h^\dagger_2) \circ (g_1 \ast h^\dagger_1) \overset{\text{(\ref{ast-inter})}}{=} (g_2 \circ g_1)  \ast (h^\dagger_2 \circ h^\dagger_1) \overset{\text{(\ref{band-dagger})}}{=} (g_2 \circ g_1) \ast (h_1 \circ h_2)^\dagger 
\end{gather*}
On the other hand, by functoriality of $\mathcal{F}$ we must have that $\mathcal{F}(f_2 \circ f_1) = \mathcal{F}(f_2) \circ \mathcal{F}(f_1)$. Therefore, it follows that $g_2 \circ g_1 \in \mathcal{F}(f_2 \circ f_1)$. Since the dagger is contravariant (\ref{eq:dagger}), we have that $(f_2 \circ f_1)^\dagger= f_1^\dagger \circ f_2^\dagger$, and therefore $\mathcal{F}((f_2 \circ f_1)^\dagger) = \mathcal{F}(f_1^\dagger) \circ \mathcal{F}(f_2^\dagger)$. As such, we also get that $h_1 \circ h_2 \in \mathcal{F}((f_2 \circ f_1)^\dagger)$. Then by definition (\ref{def:Fflat}), we get that $\mathcal{F}^\flat(f_2 \circ f_1) = (g_2 \circ g_1) \ast (h^\dagger_2 \circ h^\dagger_1)$. So by the above calculation, we get that $\mathcal{F}^\flat(f_2 \circ f_1)= \mathcal{F}^\flat(f_2) \circ  \mathcal{F}^\flat(f_1)$ as desired. 

\item Dagger: Let $g \in \mathcal{F}(f)$ and $h \in \mathcal{F}(f^\dagger)$. Then we first compute that: 
\[ \mathcal{F}^\flat(f)^\dagger \overset{\text{(\ref{def:Fflat})}}{=} (g \ast h^\dagger)^\dagger \overset{\text{(\ref{band-dagger})}}{=} h^{\dagger\dagger}  \ast g^\dagger \overset{\text{(\ref{eq:dagger})}}{=} h \ast g^\dagger  \]
Now since dagger is involutive (\ref{eq:dagger}), we have $f^{\dagger\dagger}=f$. Thus, $h \in \mathcal{F}(f^\dagger)$ and $g \in \mathcal{F}(f^{\dagger\dagger})$. Therefore by definition (\ref{def:Fflat}), we get that $\mathcal{F}^\flat(f^\dagger) = g^\dagger \ast h$. So by the above calculation, we get that $\mathcal{F}^\flat(f^\dagger) = \mathcal{F}^\flat(f)^\dagger$ as desired. 
\end{enumerate}
Therefore, we have that $\mathcal{F}^\flat: (\mathbb{C}, \dagger) \to (\mathbb{D}, \dagger)$ is a dagger functor. 

Next we compose $\mathcal{F}^\flat$ with $\mathcal{E}_{\mathcal{L}_\ast}$. On objects $X$, we clearly have that $\mathcal{E}_{\mathcal{L}_\ast}(\mathcal{F}^\flat(X)) = \mathcal{E}_{\mathcal{L}_\ast}(\mathcal{F}(X))=\mathcal{F}(X)$. On maps $f$, let $g \in \mathcal{F}(f)$ and $h \in \mathcal{F}(f^\dagger)$. Then we compute: 
\begin{gather*}
g \ast \mathcal{F}^\flat(f) \overset{\text{(\ref{def:Fflat})}}{=} g \ast g \ast h^\dagger \overset{\text{(\ref{eq:ast-cancel})}}{=} g \ast h^\dagger \overset{\text{(\ref{def:Fflat})}}{=} \mathcal{F}^\flat(f) 
\end{gather*}
Thus $\mathcal{F}^\flat(f)~\mathcal{L}_\ast~g$, implying that $\mathcal{F}^\flat(f) \in \mathcal{F}(f)$, which in turn gives us $\mathcal{F}(f) = \left[ \mathcal{F}^\flat(f) \right]_{\mathcal{L}_\ast} = \mathcal{E}_{\mathcal{L}_\ast}(\mathcal{F}^\flat(f))$. So have that  $\mathcal{E}_{\mathcal{L}_\ast} \circ \mathcal{F}^\flat = \mathcal{F}$. 

Lastly, we need to show uniqueness of $\mathcal{F}^\flat$. So suppose we have another dagger functor $\mathcal{G}: (\mathbb{C}, \dagger) \to (\mathbb{D}, \dagger)$ such that $\mathcal{E}_{\mathcal{L}_\ast} \circ \mathcal{G} = \mathcal{F}$. First observe that on objects, we have that $\mathcal{G}(X) = \mathcal{E}_{\mathcal{L}_\ast}\left( \mathcal{G}(X) \right) = \mathcal{E}_{\mathcal{L}_\ast}\left( \mathcal{F}^\flat(X) \right) =  \mathcal{F}(X) =  \mathcal{F}^\flat(X)$. So $\mathcal{G}$ and $\mathcal{F}^\flat$ are equal on objects. Now for a map $f$, $\mathcal{E}_{\mathcal{L}_\ast}\left( \mathcal{G}(f) \right) = \mathcal{F}(f)$ says that $\left[ \mathcal{G}(f) \right]_{\mathcal{L}_\ast} = \mathcal{F}(f)$, so $\mathcal{G}(f) \in \mathcal{F}(f)$. Similarly we also get that $\mathcal{G}(f^\dagger) \in \mathcal{F}(f^\dagger)$. Then we compute: 
\begin{gather*}
\mathcal{F}^\flat(f) \overset{\text{(\ref{def:Fflat})}}{=} \mathcal{G}(f) \ast \mathcal{G}(f^\dagger)^\dagger  \overset{\text{(\ref{def:dagfun})}}{=} \mathcal{G}(f) \ast \mathcal{G}(f^{\dagger\dagger})  \overset{\text{(\ref{eq:dagger})}}{=} \mathcal{G}(f) \ast \mathcal{G}(f) \overset{\text{(\ref{eq:ast-idem})}}{=} \mathcal{G}(f) 
\end{gather*}
So $\mathcal{G}(f) =\mathcal{F}^\flat(f)$ on maps as well. So we get that $\mathcal{G} = \mathcal{F}^\flat$, which proves the desired uniqueness of $\mathcal{F}^\flat$. Therefore, we conclude that $(\mathbb{D},\dagger, \mathcal{E}_{\mathcal{L}_\ast})$ is a cofree dagger category over $\mathbb{D}_{\mathcal{L}_\ast}$. 
\end{proof}

Our goal is now to prove the converse: that a cofree dagger category is rectangular banded. As mentioned above, we first show that canonical cofree dagger categories are clearly rectangular banded via the canonical rectangular band operation on the cartesian product of sets (\ref{band-cart}).

\begin{proposition}\label{prop:cofree-canon-band} For a category $\mathbb{B}$, $(\mathsf{C}[\mathbb{B}], \dagger, \ast_\mathbb{B})$ is a rectangular banded dagger category where $\ast_\mathbb{B}: \mathsf{C}[\mathbb{B}](A,B) \times \mathsf{C}[\mathbb{B}](A,B) \to \mathsf{C}[\mathbb{B}](A,B)$ is defined as follows: 
\begin{align}\label{def:ast-cofree}
(f,g) \ast_\mathbb{B} (h,k) = (f,k)
\end{align}
\end{proposition}
\begin{proof} Clearly, viewing $\mathsf{C}[\mathbb{B}](A,B) = \mathbb{B}(A,B) \times \mathbb{B}(B,A)$, we have that $\ast_\mathbb{B} = \ast_\times$, and so each homset $(\mathsf{C}[\mathbb{B}](A,B), \ast_\mathbb{B})$ is indeed a rectangular band. Next we show the necessary identities to be a rectangular banded dagger category. 
\begin{enumerate}[{\em (i)}]
\item Composition is a band morphism: 
\begin{gather*}
(f_1,f_2) \circ \left( (g_1,g_2) \ast_\mathbb{B} (h_1,h_2) \right) \circ (k_1, k_2) \overset{\text{(\ref{def:ast-cofree})}}{=} (f_1,f_2) \circ (g_1, h_2) \circ (k_1, k_2) \\
\overset{\text{(\ref{def:comp-cofree})}}{=} (f_1 \circ g_1 \circ k_1, k_2 \circ h_2 \circ f_2) \overset{\text{(\ref{def:ast-cofree})}}{=} (f_1 \circ g_1 \circ k_1, k_2 \circ g_2 \circ f_2) \ast_\mathbb{B} (f_1 \circ h_1 \circ k_1, k_2 \circ h_2 \circ f_2) \\
\overset{\text{(\ref{def:comp-cofree})}}{=} \left((f_1,f_2) \circ (g_1,g_2) \circ (k_1, k_2) \right) \ast_\mathbb{B} \left((f_1,f_2) \circ (h_1,h_2) \circ (k_1, k_2) \right)
\end{gather*}
So (\ref{band-comp}) holds as desired. 
\item Dagger is a band antimorphism: 
\begin{gather*}
\left( (f,g) \ast_\mathbb{B} (h,k) \right)^\dagger \overset{\text{(\ref{def:ast-cofree})}}{=} (f,k)^\dagger \overset{\text{(\ref{def:dag-cofree})}}{=} (k,f) \overset{\text{(\ref{def:ast-cofree})}}{=} (k,h) \ast_\mathbb{B} (g,f) \overset{\text{(\ref{def:dag-cofree})}}{=} (h,k)^\dagger \ast_\mathbb{B} (f,g)^\dagger
\end{gather*}
So (\ref{band-dagger}) holds as desired. 
\end{enumerate}
So we conclude that $(\mathsf{C}[\mathbb{B}], \dagger, \ast_\mathbb{B})$ is a rectangular banded dagger category. 
\end{proof}

To extend this to arbitrary cofree dagger categories, we first show that dagger isomorphisms transfer rectangular band structure. To see how this works, recall that if $(S,\ast)$ is a semigroup and $f: S \to X$ is a bijection, then $(X, \ast_f)$ is a semigroup where $\ast_f: X \times X \to X$ is defined as $x \ast_f y= f(f^{-1}(x) \ast f^{-1}(y))$. Moreover, it is easy to see that semigroup properties of $(S,\ast)$ transfer to $(X, \ast_f)$. So in particular, if $(S,\ast)$ is a rectangular band, then $(X, \ast_f)$ is also a rectangular band. 

\begin{lemma}\label{lemma:band-dag-iso} Let $(\mathbb{D}_1,\dagger, \ast)$ be a rectangular banded dagger category and $\mathcal{F}: (\mathbb{D}_1,\dagger) \to (\mathbb{D}_2, \dagger)$ a dagger isomorphism. Then $(\mathbb{D}_2,\dagger, \ast_\mathcal{F})$ is a rectangular banded dagger category where $\ast_\mathcal{F}: \mathbb{D}_2(A,B) \times \mathbb{D}_2(A,B) \to \mathbb{D}_2(A,B)$ is defined as follows: 
\begin{align}\label{def:ast-F}
f \ast_\mathcal{F} g = \mathcal{F}\left( \mathcal{F}^{-1}(f) \ast \mathcal{F}^{-1}(g) \right)
\end{align}
\end{lemma}
\begin{proof} Abusing notation slightly, recall that since $\mathcal{F}: \mathbb{D}_1 \to \mathbb{D}_2$ is an isomorphism of categories, we obtain bijections on the homsets $\mathcal{F}: \mathbb{D}_1(X,Y) \to \mathbb{D}_2(\mathcal{F}(X), \mathcal{F}(Y))$ for every pair of objects $X,Y \in \mathsf{Obj}(\mathbb{D}_1)$. In particular, for every pair of objects $A,B \in \mathsf{Obj}(\mathbb{D}_2)$ we get a bijection $\mathcal{F}: \mathbb{D}_1(\mathcal{F}^{-1}(A),\mathcal{F}^{-1}(B)) \to \mathbb{D}_2(A,B)$. Therefore, as explained above, since $(\mathbb{D}_1(\mathcal{F}^{-1}(A),\mathcal{F}^{-1}(B)),\ast)$ is a rectangular band, we get that $(\mathbb{D}_2(A,B),\ast_\mathcal{F})$ is also a rectangular band. It remains to check the compatibility with composition and the dagger. 

\begin{enumerate}[{\em (i)}]
\item Composition is a band morphism: 

\begin{gather*}
f \circ (g \ast_\mathcal{F} h) \circ k \overset{\text{(\ref{def:ast-F})}}{=} f \circ \mathcal{F}\left( \mathcal{F}^{-1}(g) \ast \mathcal{F}^{-1}(h) \right) \circ k \overset{\text{$\mathcal{F}$ iso.}}{=} \mathcal{F}(\mathcal{F}^{-1}(f)) \circ \mathcal{F}\left( \mathcal{F}^{-1}(g) \ast \mathcal{F}^{-1}(h) \right) \circ \mathcal{F}(\mathcal{F}^{-1}(k)) \\\overset{\text{(\ref{def:functor})}}{=} \mathcal{F}\left( \mathcal{F}^{-1}(f) \circ \left( \mathcal{F}^{-1}(g) \ast \mathcal{F}^{-1}(h) \right) \circ \mathcal{F}^{-1}(k) \right) \\\overset{\text{(\ref{band-comp})}}{=} \mathcal{F}\left( \left(\mathcal{F}^{-1}(f) \circ \mathcal{F}^{-1}(g) \circ \mathcal{F}^{-1}(k) \right) \ast \left(\mathcal{F}^{-1}(f) \circ \mathcal{F}^{-1}(h) \circ \mathcal{F}^{-1}(k) \right) \right) \\\overset{\text{(\ref{def:functor})}}{=} \mathcal{F}\left( \mathcal{F}^{-1}(f \circ g \circ k) \ast \mathcal{F}^{-1}(f \circ h \circ k) \right)\overset{\text{(\ref{def:ast-F})}}{=}  (f \circ g \circ k) \ast_\mathcal{F} (f \circ h \circ k)
\end{gather*}
\item Dagger is a band antimorphism: 
\begin{gather*}
(f \ast_\mathcal{F} g)^\dagger \overset{\text{(\ref{def:ast-F})}}{=} \mathcal{F}\left( \mathcal{F}^{-1}(f) \ast \mathcal{F}^{-1}(g) \right)^\dagger \overset{\text{(\ref{def:dagfun})}}{=} \mathcal{F}\left( \left( \mathcal{F}^{-1}(f) \ast \mathcal{F}^{-1}(g) \right)^\dagger \right) \overset{\text{(\ref{band-dagger})}}{=}  \mathcal{F}\left( \left( \mathcal{F}^{-1}(g)^\dagger \ast \mathcal{F}^{-1}(f)^\dagger \right)^\dagger \right)\\
\overset{\text{(\ref{def:dagfun})}}{=}  \mathcal{F}\left( \mathcal{F}^{-1}(g^\dagger) \ast \mathcal{F}^{-1}(f^\dagger) \right) \overset{\text{(\ref{def:ast-F})}}{=} g^\dagger \ast_\mathcal{F} f^\dagger
\end{gather*}
\end{enumerate}
Thus we conclude that $(\mathbb{D}_2,\dagger, \ast_\mathcal{F})$ is a rectangular banded dagger category. 
\end{proof}

\begin{proposition}\label{prop:cofree-to-band} If $(\mathbb{D},\dagger, \mathcal{E})$ is a cofree dagger category over a category $\mathbb{B}$, then $(\mathbb{D},\dagger,\ast_\mathcal{E})$ is a rectangular banded dagger category where the binary operation $\ast_\mathcal{E}: \mathbb{D}(A,B) \times \mathbb{D}(A,B) \to \mathbb{D}(A,B)$ is defined as follows:
\begin{align}\label{def:ast-E}
f \ast_\mathcal{E} g = \mathcal{E}^\flat_\mathbb{B}\left(  \mathcal{E}(f), \mathcal{E}(g^\dagger) \right)
\end{align}
\end{proposition}
\begin{proof} We showed in Prop \ref{prop:cofree-canon-band} that $(\mathsf{C}[\mathbb{B}], \dagger, \ast_\mathbb{B})$ is a rectangular banded dagger category. Also recall from Cor \ref{cor:cofree-iso.1} that $\mathcal{E}^\flat_\mathbb{B}: (\mathsf{C}[\mathbb{B}], \dagger) \to (\mathbb{D}, \dagger)$ is a dagger isomorphism with inverse $\mathcal{E}^\flat: (\mathbb{D}, \dagger) \to (\mathsf{C}[\mathbb{B}], \dagger)$. Thus applying Lemma \ref{lemma:band-dag-iso} give us that $(\mathbb{D},\dagger,\ast_{\mathcal{E}^\flat_\mathbb{B}})$ is a rectangular banded dagger category. Let us work out $\ast_{\mathcal{E}^\flat_\mathbb{B}}$ explicitly:
\begin{gather*}
f \ast_{\mathcal{E}^\flat_\mathbb{B}} g \overset{\text{(\ref{def:ast-F})}}{=} \mathcal{E}^\flat_\mathbb{B}\left( \mathcal{E}^\flat(f) \ast_\mathbb{B} \mathcal{E}^\flat(g) \right) \overset{\overset{\text{Prop}}{\ref{prop:canon-cofree}}}{=} \mathcal{E}^\flat_\mathbb{B}\left( (\mathcal{E}(f),\mathcal{E}(f^\dagger)) \ast_\mathbb{B} (\mathcal{E}(g),\mathcal{E}(g^\dagger)) \right) \overset{\text{(\ref{def:ast-cofree})}}{=} \mathcal{E}^\flat_\mathbb{B}\left( (\mathcal{E}(f),\mathcal{E}(g^\dagger)) \right) \overset{\text{(\ref{def:ast-E})}}{=}f \ast_\mathcal{E} g
\end{gather*}
Thus $\ast_{\mathcal{E}^\flat_\mathbb{B}}=\ast_\mathcal{E}$. So $(\mathbb{D},\dagger,\ast_\mathcal{E})$ is a rectangular banded dagger category as desired. 
\end{proof}

Applying the above proposition to the canonical cofree dagger category construction gives us back the canonical rectangular banded structure. 

\begin{corollary} Let $\mathbb{B}$ be a category. Then for the cofree dagger category $(\mathsf{C}[\mathbb{B}], \dagger, \mathcal{E}_\mathbb{B})$, we have that $\ast_{\mathcal{E}_\mathbb{B}} = \ast_\mathbb{B}$.  
\end{corollary}
\begin{proof} Observe that in this case, $\mathcal{E}^\flat_\mathbb{B} = \mathsf{id}_{\mathsf{C}[\mathbb{B}]}$. Thus we compute: 
\begin{gather*}
(f,g) \ast_{\mathcal{E}_\mathbb{B}} (h,k) \overset{\text{(\ref{def:ast-E})}}{=} \mathcal{E}^\flat_\mathbb{B}\left(  \mathcal{E}_\mathbb{B}(f,g), \mathcal{E}_\mathbb{B}((h,k)^\dagger) \right) \overset{\text{$\mathcal{E}^\flat_\mathbb{B} = \mathsf{id}_{\mathsf{C}[\mathbb{B}]}$}}{=} \left(  \mathcal{E}_\mathbb{B}(f,g), \mathcal{E}_\mathbb{B}((h,k)^\dagger) \right) \overset{\text{(\ref{def:dag-cofree})}}{=} \left(  \mathcal{E}_\mathbb{B}(f,g), \mathcal{E}_\mathbb{B}(k,h) \right) \\\overset{\overset{\text{Prop}}{\text{\ref{prop:canon-cofree}}}}{=} (f,k) \overset{(\ref{def:ast-cofree})}{=} (f,g) \ast_\mathbb{B} (h,k)
\end{gather*}
So indeed $\ast_{\mathcal{E}_\mathbb{B}} = \ast_\mathbb{B}$.  
\end{proof}

Thus, Prop \ref{prop:rect-to-cofree} and Prop \ref{prop:cofree-to-band} together give us our first main result:  

\begin{theorem}\label{thm:cofree=rect} A dagger category is cofree if and only if it has a rectangular banded structure. 
\end{theorem}

Therefore, rectangular banded structure is a base independent way of describing cofreeness of a dagger category, where we've replaced the outer base category with an internal dagger category structure. Of course, a natural question to ask is if the constructions of Prop \ref{prop:rect-to-cofree} and Prop \ref{prop:cofree-to-band}. Clearly applying one than the other will preserve the dagger category, but what about the base category or the rectangular band structure. We show that if we start with a rectangular banded dagger category, then apply Prop \ref{prop:cofree-to-band} to the constructed base category from Prop \ref{prop:rect-to-cofree}, we obtain the starting rectangular band structure. 

\begin{proposition}\label{prop:asE=ast}Let $(\mathbb{D}, \dagger, \ast)$ be a rectangular banded dagger category. Then $\ast_{\mathcal{E}_{\mathcal{L}_\ast}} = \ast$. 
\end{proposition}
\begin{proof} Observe that $\mathcal{E}_{\mathbb{D}_{\mathcal{L}_\ast}} \left( [f]_{\mathcal{L}_\ast}, [g]_{\mathcal{L}_\ast}\right) = [f]_{\mathcal{L}_\ast}$ and $\mathcal{E}_{\mathbb{D}_{\mathcal{L}_\ast}} \left( \left([f]_{\mathcal{L}_\ast}, [g]_{\mathcal{L}_\ast} \right)^\dagger \right) = [g]_{\mathcal{L}_\ast}$. So $f \in \mathcal{E}_{\mathbb{D}_{\mathcal{L}_\ast}} \left( [f]_{\mathcal{L}_\ast}, [g]_{\mathcal{L}_\ast}\right)$ and $g \in \mathcal{E}_{\mathbb{D}_{\mathcal{L}_\ast}} \left( \left([f]_{\mathcal{L}_\ast}, [g]_{\mathcal{L}_\ast} \right)^\dagger \right)$. Then we compute that: 
\begin{gather*} 
f \ast_{\mathcal{E}_{\mathcal{L}_\ast}} g \overset{\text{(\ref{def:ast-E})}}{=} \mathcal{E}^\flat_{\mathbb{D}_{\mathcal{L}_\ast}} \left(  \mathcal{E}_{\mathcal{L}_\ast}(f), \mathcal{E}_{\mathcal{L}_\ast}(g^\dagger) \right) = \mathcal{E}^\flat_{\mathbb{D}_{\mathcal{L}_\ast}} \left( [f]_{\mathcal{L}_\ast}, [g^\dagger]_{\mathcal{L}_\ast}\right) \overset{\text{(\ref{def:Fflat})}}{=} f \ast g^{\dagger \dagger} \overset{\text{(\ref{eq:dagger})}}{=} f \ast g
\end{gather*}
Thus $\ast_{\mathcal{E}_{\mathcal{L}_\ast}} = \ast$ as desired. 
\end{proof}

On the other hand, if we start with a cofree dagger category over a specified base category, apply Prop \ref{prop:cofree-to-band} to get a rectangular band structure, and then apply Prop \ref{prop:rect-to-cofree}, the resulting quotient base category may not necessarily be the same as the starting base category. Instead, it will capture another base that we have already discussed: the image category of the functor down to the starting base category (Lemma \ref{lemma:cofree-im-base}). 

\begin{proposition}\label{prop:imE=ast} Let $(\mathbb{D},\dagger, \mathcal{E})$ be a cofree dagger category over a category $\mathbb{B}$. Then the functor $\mathcal{E}^\natural: \mathbb{D}_{\mathcal{L}_{\ast_\mathcal{E}}} \to \mathsf{Im}[\mathcal{E}]$ defined on objects as $\mathcal{E}^\natural(A) = \mathcal{E}(A)$ and on maps as $\mathcal{E}^\natural([f]_{\mathcal{L}_{\ast_\mathcal{E}}}) = \mathcal{E}(f)$ is an isomorphism. Moreover, the composite functor $\mathcal{I}_\mathcal{E} \circ \mathcal{E}^\natural: \mathbb{D}_{\mathcal{L}_{\ast_\mathcal{E}}} \to \mathbb{B}$ is a faithful functor and it is the unique functor such that the following diagram commutes: 
\begin{equation}\begin{gathered}\label{diag:cofree-imE-L} 
\xymatrixcolsep{5pc}\xymatrix{ & \mathbb{D} \ar[dr]^-{\mathcal{E}} \ar[dl]_-{\mathcal{E}_{\mathcal{L}_{\ast_\mathcal{E}}}} \\
\mathbb{D}_{\mathcal{L}_{\ast_\mathcal{E}}}  \ar@{-->}[rr]_-{\exists! ~\mathcal{I}_\mathcal{E} \circ \mathcal{E}^\natural} &&  \mathbb{B}}
\end{gathered}\end{equation}
\end{proposition}
\begin{proof} The key to this proof is to show the following, that for parallel maps $f,g \in \mathbb{D}(A,B)$ we have that: 
\begin{align}\label{lemma:E-iff-ast}
f~\mathcal{L}_{\ast_\mathcal{E}}~g ~\Leftrightarrow~ \mathcal{E}(f) = \mathcal{E}(g)
\end{align}
So suppose that $f~\mathcal{L}_{\ast_\mathcal{E}}~g$, that is, $f \ast_\mathcal{E} g = g$. Then we compute: 
\begin{gather*}
\mathcal{E}(g) \overset{\text{(\ref{def:sim})}}{=} \mathcal{E}\left( f \ast_\mathcal{E} g \right) \overset{\text{(\ref{def:ast-E})}}{=} \mathcal{E}\left( \mathcal{E}^\flat_\mathbb{B}\left(  \mathcal{E}(f), \mathcal{E}(g^\dagger) \right) \right) \overset{\text{(\ref{diag:cofree})}}{=} \mathcal{E}_\mathbb{B}\left(  \mathcal{E}(f), \mathcal{E}(g^\dagger) \right) \overset{\overset{\text{Prop}}{\text{\ref{prop:canon-cofree}}}}{=} \mathcal{E}(f)
\end{gather*}
So $\mathcal{E}(f) = \mathcal{E}(g)$ as desired. Conversely, suppose we have that $\mathcal{E}(f) = \mathcal{E}(g)$. Then we compute:
\begin{gather*}
f \ast_\mathcal{E} g \overset{\text{(\ref{def:ast-E})}}{=} \mathcal{E}^\flat_\mathbb{B}\left(  \mathcal{E}(f), \mathcal{E}(g^\dagger) \right) \overset{\text{$\mathcal{E}(f) = \mathcal{E}(g)$}}{=} \mathcal{E}^\flat_\mathbb{B}\left(  \mathcal{E}(g), \mathcal{E}(g^\dagger) \right) \overset{\text{(\ref{def:ast-E})}}{=} g \ast_\mathcal{E} g \overset{\text{(\ref{eq:ast-idem})}}{=} g
\end{gather*}
So $f \ast_\mathcal{E} g = g$, which means $f~\mathcal{L}_{\ast_\mathcal{E}}~g$. Thus (\ref{lemma:E-iff-ast}) holds as desired. 

Now with this, it follows that $\mathcal{E}^\natural$ is well-defined. Functoriality of $\mathcal{E}^\natural$ is immediate from functoriality of $\mathcal{E}$. So $\mathcal{E}^\natural$ is indeed a functor. Now it's inverse ${\mathcal{E}^\natural}^{-1}: \mathsf{Im}[\mathcal{E}] \to \mathbb{D}_{\mathcal{L}_{\ast_\mathcal{E}}}$ is defined as follows, on object ${\mathcal{E}^\natural}^{-1}\left( \mathcal{E}(A) \right) = A$ and on maps ${\mathcal{E}^\natural}^{-1}\left( \mathcal{E}(f) \right) = [f]_{\mathcal{L}_{\ast_\mathcal{E}}}$. Now by Lemma \ref{lemma:cofree-objects}, recall that $\mathcal{E}$ is bijective on objects, therefore ${\mathcal{E}^\natural}^{-1}$ is well-defined on objects. By (\ref{lemma:E-iff-ast}), ${\mathcal{E}^\natural}^{-1}$ is also well-defined on maps. It is clear that ${\mathcal{E}^\natural}^{-1}$ is functorial, so ${\mathcal{E}^\natural}^{-1}$ is indeed a functor. Moreover, since $\mathcal{E}$ is bijective on objects and by (\ref{lemma:E-iff-ast}), we get that $\mathcal{E}^\natural$ and ${\mathcal{E}^\natural}^{-1}$ are indeed inverses of each other. So we conclude that $\mathcal{E}^\natural$ is an isomorphism. From this it clearly follows that $\mathcal{I}_\mathcal{E} \circ \mathcal{E}^\natural$ is faithful and that it is the unique functor such that (\ref{diag:cofree-imE-L}) commutes. 
\end{proof}

\begin{corollary}\label{cor:astE=3} Let $(\mathbb{D},\dagger, \mathcal{E})$ be a cofree dagger category over a category $\mathbb{B}$. Then $\ast_\mathcal{E} = \ast_{\overline{\mathcal{E}}} = \ast_{\mathcal{E}_{\mathcal{L}_{\ast_\mathcal{E}}}}$. 
\end{corollary}
\begin{proof} This follows from Prop \ref{prop:imE=ast} and then applying Prop \ref{prop:asE=ast}. 
\end{proof}

\begin{corollary}\label{cor:astE=3-canon} Let $\mathbb{B}$ be a category. Then for $(\mathsf{C}[\mathbb{B}], \dagger, \ast_\mathbb{B})$, we have that $\ast_\mathbb{B} = \ast_{\mathsf{Im}[\mathcal{E}_\mathbb{B}}$
\end{corollary}
\begin{proof} This follows from Lemma \ref{lemma:cofree-imE=B} and then applying Cor \ref{cor:astE=3}. 
\end{proof}

\subsection{Cofree Dagger Categories as Algebras}\label{sec:alg}

In this section we will provide yet another characterization of cofree dagger categories, this time as algebras of a monad. For a review on adjunctions, monads, and their algebras, we invite the reader to see \cite{mac1971categories}. 

Indeed, the canonical cofree dagger construction from Sec \ref{sec:cofree-canonical} gives a right adjoint to the forgetful functor $\mathsf{U}: \mathsf{DAG} \to \mathsf{CAT}$ \cite[Thm 3.1.17]{heunen2009categorical}. As such, this induces a monad $(\mathsf{C}, \mathcal{M}, \mathcal{N})$ on $\mathsf{DAG}$. The functor $\mathsf{C}: \mathsf{DAG} \to \mathsf{DAG}$ maps a dagger category $(\mathbb{D},\dagger)$ to the canonical cofree dagger category of its underlying category, $\mathsf{C}(\mathbb{D},\dagger)= (\mathsf{C}[\mathbb{D}], \dagger)$, and sends a dagger functor $\mathcal{F}: (\mathbb{D}_1, \dagger) \to (\mathbb{D}_2, \dagger)$ to the dagger functor $\mathsf{C}(\mathcal{F}): (\mathsf{C}[\mathbb{D}_1], \dagger) \to (\mathsf{C}[\mathbb{D}_2], \dagger)$ defined as the unique dagger functor which makes the following diagram commute: 
\begin{equation}\begin{gathered}\label{def:CF}  \xymatrixcolsep{5pc}\xymatrix{\mathsf{C}[\mathbb{D}_1]  \ar[d]_-{\mathcal{E}_{\mathbb{D}_1}}  \ar@{-->}[r]^-{\exists! ~ \mathsf{C}[\mathcal{F}]}  & \mathsf{C}[\mathbb{D}_2]  \ar[d]^-{\mathcal{E}_{\mathbb{D}_2}} \\
 \mathbb{D}_1 \ar[r]_-{\mathcal{F}} & \mathbb{D}_2 }
\end{gathered}\end{equation}
Explicitly, $\mathsf{C}[\mathcal{F}]$ is defined as follows on objects and maps respectively: 
\begin{align}
\mathsf{C}[\mathcal{F}](A) = \mathcal{F}(A) && \mathsf{C}[\mathcal{F}](f,g)= (\mathcal{F}(f), \mathcal{F}(g))
\end{align}
Now for a dagger category $(\mathbb{D}, \dagger)$, the monad multiplication $\mathcal{M}_{(\mathbb{D},\dagger)}: (\mathsf{C}\left[\mathsf{C}[\mathbb{D}]\right], \dagger) \to (\mathsf{C}[\mathbb{D}], \dagger)$ and the monad unit  $\mathcal{N}_{(\mathbb{D},\dagger)}: (\mathbb{D},\dagger) \to (\mathsf{C}[\mathbb{D}], \dagger)$ are of course the unique dagger functors which make the following diagrams commute respectively: 
\begin{equation}\begin{gathered}\label{diag:MN}  \xymatrixcolsep{5pc}\xymatrix{ \mathbb{D} \ar@{=}[dr] \ar@{-->}[r]^-{\exists! ~ \mathcal{N}_{(\mathbb{D},\dagger)}}  & \mathsf{C}[\mathbb{D}]  \ar[d]^-{\mathcal{E}_{\mathbb{D}}} & \mathsf{C}\left[\mathsf{C}[\mathbb{D}] \right]  \ar[d]_-{\mathcal{E}_{\mathsf{C}[\mathbb{D}]}}  \ar@{-->}[r]^-{\exists! ~ \mathcal{M}_{(\mathbb{D},\dagger)}}  & \mathsf{C}[\mathbb{D}]  \ar[d]^-{\mathcal{E}_{\mathbb{D}}} \\
& \mathbb{D} & \mathsf{C}[\mathbb{D}] \ar[r]_-{\mathcal{E}_\mathbb{D}} & \mathbb{D} }
\end{gathered}\end{equation}
Explicitly, $\mathcal{M}_{(\mathbb{D},\dagger)}$ is defined on objects and maps respectively as follows: 
\begin{align}
\mathcal{M}_{(\mathbb{D},\dagger)}(A) = A && \mathcal{M}_{(\mathbb{D},\dagger)}\left( (f,g), (h,k) \right) = (f,h)
\end{align}
while $\mathcal{N}_{(\mathbb{D},\dagger)}$ is defined on objects and maps respectively as follows: 
\begin{align}
\mathcal{N}_{(\mathbb{D},\dagger)}(A) = A && \mathcal{N}_{(\mathbb{D},\dagger)}\left( f \right) = (f,f^\dagger)
\end{align}
Since we have a monad, we can consider its algebras.  

\begin{definition} A \textbf{$\mathsf{C}$-algebra} is a triple $(\mathbb{D},\dagger, \mathcal{D})$ consisting of a dagger category $(\mathbb{D},\dagger)$ and a dagger functor ${\mathcal{D}: (\mathsf{C}[\mathbb{D}], \dagger) \to (\mathbb{D},\dagger)}$ such that the following diagrams commute:
\begin{equation}\begin{gathered}\label{diag:C-alg}   \xymatrixcolsep{5pc}\xymatrix{ \mathbb{D} \ar@{=}[dr]^-{} \ar[r]^-{\mathcal{N}_{(\mathbb{D},\dagger)}} & \mathsf{C}[\mathbb{D}] \ar[d]^-{\mathcal{D}} &  \mathsf{C}\left[\mathsf{C}[\mathbb{D}]\right] \ar[r]^-{\mathcal{M}_{(\mathbb{D},\dagger)}}  \ar[d]_-{\mathsf{C}[\mathcal{D}]} &   \mathbb{D}  \ar[d]^-{\mathcal{D}}    \\
 & \mathbb{D} & \mathsf{C}[\mathbb{D}]  \ar[r]_-{\mathcal{D}} & \mathbb{D}
 }
\end{gathered}\end{equation}
We call $\mathcal{D}$ a \textbf{$\mathsf{C}$-algebra structure} of $(\mathbb{D},\dagger)$. 
\end{definition}

It is again important to stress that a dagger category can have multiple $\mathsf{C}$-algebra structures on it. Let us now unpack what the diagrams (\ref{diag:C-alg}) are saying more explicitly. 

\begin{lemma}\label{lemma:cofree-Calg} Let $(\mathbb{D},\dagger)$ be a dagger category and $\mathcal{D}: \mathsf{C}[\mathbb{D}] \to \mathbb{D}$ a functor. Then $(\mathbb{D},\dagger, \mathcal{D})$ is a $\mathsf{C}$-algebra if and only if: 
\begin{enumerate}[{\em (i)}] 
\item For every map $f: A \to B$, the following equality holds: 
\begin{align}\label{eq:D-idem}
\mathcal{D}(f,f^\dagger) = f
\end{align}
\item For all maps $f: A \to B$, $g: B \to A$, $h: B \to A$, and $k: A \to B$, the following equality holds: 
\begin{align}\label{eq:D-cancel}
\mathcal{D}\left( \mathcal{D}(f,g), \mathcal{D}(h,k) \right) = \mathcal{D}(f,h)
\end{align}
\item For all maps $f: A \to B$ and $g: B \to A$, the following equality holds: 
\begin{align}\label{eq:D-dagger}
\mathcal{D}(f,g)^\dagger = \mathcal{D}(g,f)
\end{align}
\end{enumerate}
\end{lemma}
\begin{proof} It is straightforward to see that (\ref{eq:D-dagger}) is equivalent to saying $\mathcal{D}$ is a dagger functor, that (\ref{eq:D-idem}) is equivalent to saying the left diagram of (\ref{diag:C-alg}) commutes, while (\ref{eq:D-cancel}) is equivalent to saying that the right diagram of (\ref{diag:C-alg}) commutes. 
\end{proof}

By well-known facts about adjunctions, every cofree dagger category is a $\mathsf{C}$-algebra. 

\begin{lemma} Let $(\mathbb{D}, \dagger, \mathcal{E})$ be a  cofree dagger category over a category $\mathbb{B}$. Define the dagger functor $\mathcal{D}_\mathcal{E}: (\mathsf{C}[\mathbb{D}], \dagger) \to  (\mathbb{D}, \dagger)$ as the unique dagger functor which makes the following diagram commute: 
\begin{align*} \xymatrixcolsep{5pc}\xymatrix{\mathsf{C}[\mathbb{D}]  \ar[d]_-{\mathcal{E}_\mathbb{D}}  \ar@{-->}[r]^-{\exists! ~ \mathcal{D}_\mathcal{E}}  & \mathbb{D}\ar[d]^-{\mathcal{E}} \\
 \mathbb{D} \ar[r]_-{\mathcal{E}} & \mathbb{B} }
\end{align*}
Then $(\mathbb{D}, \dagger, \mathcal{D}_\mathcal{E})$ is a $\mathsf{C}$-algebra. 
\end{lemma}

\begin{corollary} Let $\mathbb{B}$ be a category and let $\mathcal{D}_\mathbb{B}: (\mathsf{C}\left[\mathsf{C}[\mathbb{B}]\right], \dagger) \to  (\mathsf{C}[\mathbb{B}], \dagger)$ be the dagger functor defined as $\mathcal{D}_\mathbb{B} := \mathcal{D}_{\mathcal{E}_\mathbb{B}}$. Explicitly, on objects $\mathcal{D}_\mathbb{B}(A)=A$ while on maps $\mathcal{D}_\mathbb{B}\left( (f,g), (h,k) \right) = (f,k)$. Then $(\mathsf{C}[\mathbb{B}], \dagger, \mathcal{D}_\mathbb{B})$ is a $\mathsf{C}$-algebra. 
\end{corollary}

Therefore, every rectangular banded dagger category is a $\mathsf{C}$-algebra, whose $\mathsf{C}$-algebra structure can be given in terms of the rectangular band operation. 

\begin{proposition}\label{prop:rect-calg} Let $(\mathbb{D}, \dagger,\ast)$ be a rectangular banded dagger category. Then define the functor $\mathcal{D}_\ast: \mathsf{C}[\mathbb{D}] \to \mathbb{D}$ on objects as $\mathcal{D}_\ast(A) = A$ and on maps as $\mathcal{D}_\ast(f,g) = f \ast g^\dagger$. Then $(\mathbb{D}, \dagger, \mathcal{D}_\ast)$ is a $\mathsf{C}$-algebra. 
\end{proposition}
\begin{proof} By Prop \ref{prop:rect-to-cofree}, we know that $(\mathbb{D},\dagger, \mathcal{E}_{\mathcal{L}_\ast})$ is a cofree dagger category over $\mathbb{D}_{\mathcal{L}_\ast}$. By Lemma \ref{lemma:cofree-Calg}, we then get that $(\mathbb{D}, \dagger, \mathcal{D}_{\mathcal{E}_{\mathcal{L}_\ast}})$ is a $\mathsf{C}$-algebra. Now by definition, $\mathcal{D}_{\mathcal{E}_{\mathcal{L}_\ast}}= (\mathcal{E}_{\mathcal{L}_\ast} \circ \mathcal{E}_\mathbb{D})^\flat$. Observe that $\mathcal{E}_{\mathcal{L}_\ast}\left( \mathcal{E}_\mathbb{D} (f,g) \right) = [f]_{\mathcal{L}_\ast}$ and similarly, $\mathcal{E}_{\mathcal{L}_\ast}\left( \mathcal{E}_\mathbb{D} \left( (f,g)^\dagger \right) \right) = [g]_{\mathcal{L}_\ast}$. Thus, $f \in \mathcal{E}_{\mathcal{L}_\ast}\left( \mathcal{E}_\mathbb{D} (f,g) \right)$ and $g \in \mathcal{E}_{\mathcal{L}_\ast}\left( \mathcal{E}_\mathbb{D} \left( (f,g)^\dagger \right) \right)$. So by (\ref{def:Fflat}), we have that $\mathcal{D}_{\mathcal{E}_{\mathcal{L}_\ast}}(f,g) = f \ast g^\dagger$. As such, $\mathcal{D}_\ast=\mathcal{D}_{\mathcal{E}_{\mathcal{L}_\ast}}$. So we conclude that $(\mathbb{D}, \dagger, \mathcal{D}_\ast)$ is a $\mathsf{C}$-algebra as desired. 
\end{proof}

We will now show that, in fact, every $\mathsf{C}$-algebra is a cofree dagger category. To do so, we will show that every $\mathsf{C}$-algebra is rectangular banded.  

\begin{proposition}\label{prop:calg-rest} If $(\mathbb{D},\dagger, \mathcal{D})$ is a $\mathsf{C}$-algebra, then $(\mathbb{D},\dagger,\ast_\mathcal{D})$ is a rectangular banded dagger category where the binary operation $\ast_\mathcal{D}: \mathbb{D}(A,B) \times \mathbb{D}(A,B) \to \mathbb{D}(A,B)$ is defined as follows: 
\begin{align}\label{def:ast-D}
f \ast_\mathcal{D} g = \mathcal{D}(f, g^\dagger) 
\end{align}
Therefore, $(\mathbb{D},\dagger)$ is a cofree dagger category. 
\end{proposition}
\begin{proof} We first show that $(\mathbb{D}(A,B),\ast_\mathcal{D})$ is a rectangular band :
\begin{enumerate}[{\em (i)}]
\item Associativity:
\begin{gather*}
    (f \ast_\mathcal{D} g) \ast_\mathcal{D} h \overset{\text{(\ref{def:ast-D})}}{=} \mathcal{D}( f, g^\dagger ) \ast_\mathcal{D} h \overset{\text{(\ref{def:ast-D})}}{=} \mathcal{D}\left( \mathcal{D}( f, g^\dagger), h^\dagger \right) \overset{\text{(\ref{eq:D-idem})}}{=} \mathcal{D}\left( \mathcal{D}( f, g^\dagger), \mathcal{D}(h^\dagger, h^{\dagger\dagger}) \right) \\ \overset{\text{(\ref{eq:dagger})}}{=} \mathcal{D}\left( \mathcal{D}( f, g^\dagger), \mathcal{D}(h^\dagger, h) \right)
    \overset{\text{(\ref{eq:D-cancel})}}{=} \mathcal{D}(f,h^\dagger) \overset{\text{(\ref{eq:D-cancel})}}{=} \mathcal{D}\left( \mathcal{D}( f, f^\dagger), \mathcal{D}(h^\dagger, g) \right) \overset{\text{(\ref{eq:D-idem})}}{=}  \mathcal{D}\left( f, \mathcal{D}(h^\dagger, g) \right) \\ \overset{\text{(\ref{eq:D-dagger})}}{=}  \mathcal{D}\left(f, \mathcal{D}(g,h^\dagger)^\dagger \right) \overset{\text{(\ref{def:ast-D})}}{=} f \ast_\mathcal{D} \mathcal{D}(g,h^\dagger)  \overset{\text{(\ref{def:ast-D})}}{=} f \ast_\mathcal{D} (g \ast_\mathcal{D} h)  
\end{gather*}
\item Idempotent: 
\begin{align*}
    f \ast_\mathcal{D}  f \overset{\text{(\ref{def:ast-D})}}{=} \mathcal{D}( f, f^\dagger ) \overset{\text{(\ref{eq:D-idem})}}{=} f
\end{align*}
\item Middle Cancellation: 
\begin{gather*}
f \ast_\mathcal{D} g \ast_\mathcal{D} h \overset{\text{(\ref{def:ast-D})}}{=} \mathcal{D}( f, g^\dagger ) \ast h \overset{\text{(\ref{def:ast-D})}}{=} \mathcal{D}\left( \mathcal{D}( f, g^\dagger), h^\dagger \right) \overset{\text{(\ref{eq:D-idem})}}{=} \mathcal{D}\left( \mathcal{D}( f, g^\dagger), \mathcal{D}(h^\dagger, h^{\dagger\dagger}) \right) \\\overset{\text{(\ref{eq:dagger})}}{=} \mathcal{D}\left( \mathcal{D}( f, g^\dagger), \mathcal{D}(h^\dagger, h) \right) \overset{\text{(\ref{eq:D-cancel})}}{=} \mathcal{D}(f,h^\dagger) \overset{\text{(\ref{def:ast-D})}}{=} f \ast_\mathcal{D} h
\end{gather*}
\end{enumerate}
Next we show the rectangular band enrichment: 
\begin{enumerate}[{\em (i)}]
\setcounter{enumi}{3}
\item Composition:
\begin{gather*}
f \circ (g \ast_\mathcal{D} h) \circ k \overset{\text{(\ref{def:ast-D})}}{=} f \circ \mathcal{D}(g,h^\dagger) \circ k \overset{\text{(\ref{eq:D-idem})}}{=} \mathcal{D}( f, f^\dagger ) \circ \mathcal{D}( g, h^\dagger ) \circ \mathcal{D}( k, k^\dagger)   \overset{\text{(\ref{def:functor})}}{=} \mathcal{D}\left( (f,f^\dagger) \circ (g, h^\dagger) \circ ( k, k^\dagger ) \right) \\
\overset{\text{(\ref{def:comp-cofree})}}{=} \mathcal{D}\left( f \circ g \circ k, k^\dagger \circ h^\dagger \circ f^\dagger \right) \overset{\text{(\ref{eq:dagger})}}{=}   \mathcal{D}\left( f \circ g \circ k, (f \circ h \circ k)^\dagger \right) \overset{\text{(\ref{def:ast-D})}}{=} (f \circ g \circ k) \ast_\mathcal{D} (f \circ h \circ k) 
\end{gather*}
\item Dagger: 
\begin{gather*}
    (f \ast_\mathcal{D}  g)^\dagger \overset{\text{(\ref{def:ast-D})}}{=} \mathcal{D}( f, g^\dagger )^\dagger \overset{\text{(\ref{eq:D-dagger})}}{=}  \mathcal{D}( g^\dagger, f ) \overset{\text{(\ref{eq:dagger})}}{=} \mathcal{D}( g^\dagger, f^{\dagger\dagger} )  \overset{\text{(\ref{def:ast-D})}}{=} g^\dagger \ast f^\dagger 
\end{gather*}
\end{enumerate}
So we conclude that $(\mathbb{D},\dagger,\ast_\mathcal{D})$ is a rectangular banded dagger category and, therefore by Thm \ref{thm:cofree=rect}, that $(\mathbb{D},\dagger)$ is also a cofree dagger category. 
\end{proof}

It turns out that the constructions of Prop \ref{prop:rect-calg} and Prop \ref{prop:calg-rest} are inverses of each other.

\begin{theorem}\label{thm:coag=rect} For a dagger category $(\mathbb{D}, \dagger)$ there is a bijective correspondence between rectangular banded structure and $\mathsf{C}$-algebra structure. Explicitly,
\begin{enumerate}[{\em (i)}]
\item If $\ast$ is a rectangular banded structure on $(\mathbb{D}, \dagger)$, then $\ast_{\mathcal{D}_\ast} = \ast$.
\item If $\mathcal{D}$ is a $\mathsf{C}$-algebra structure, then $\mathcal{D}_{\ast_\mathcal{D}} = \mathcal{D}$. 
\end{enumerate}
\end{theorem}
\begin{proof} Starting with a rectangular banded structure $\ast$, we compute: 
\begin{gather*}
f \ast_{\mathcal{D}_\ast} g \overset{\text{(\ref{def:ast-D})}}{=} \mathcal{D}_\ast(f,g^\dagger) \overset{\overset{\text{Prop}}{\text{\ref{prop:rect-calg}}}}{=} f \ast g^{\dagger\dagger} \overset{\text{(\ref{eq:dagger})}}{=} f \ast g
\end{gather*}
Thus $\ast_{\mathcal{D}_\ast} = \ast$. On the other hand, starting with a $\mathsf{C}$-algebra structure $\mathcal{D}$, we compute: 
\begin{gather*}
\mathcal{D}_{\ast_\mathcal{D}}(f,g) \overset{\overset{\text{Prop}}{\text{\ref{prop:rect-calg}}}}{=} f \ast_\mathcal{D} g^\dagger \overset{\text{(\ref{def:ast-D})}}{=} \mathcal{D}(f,g^{\dagger\dagger}) \overset{\text{(\ref{eq:dagger})}}{=} \mathcal{D}(f,g)
\end{gather*}
Thus $\mathcal{D}_{\ast_\mathcal{D}} = \mathcal{D}$. 
\end{proof}

Applying the above constructions to a cofree dagger category with a specified base category recaptures the induced rectangular banded structure and $\mathsf{C}$-algebra structure, as expected. 

\begin{corollary}\label{cor:Dast=} Let $(\mathbb{D}, \dagger, \mathcal{E})$ be a  cofree dagger category over a category $\mathbb{B}$. Then $\mathcal{D}_\mathcal{E} = \mathcal{D}_{\ast_\mathcal{E}}$ and $\ast_{\mathcal{D}_\mathcal{E}} = \ast_{\mathcal{E}}$.
\end{corollary}
\begin{proof} We compute: 
\begin{gather*}
\mathcal{E}\left( \mathcal{D}_{\ast_\mathcal{E}}(f,g) \right) \overset{\overset{\text{Prop}}{\text{\ref{prop:rect-calg}}}}{=} \mathcal{E}(f \ast_\mathcal{E} g^\dagger) \overset{\text{(\ref{def:ast-E})}}{=} \mathcal{E}\left( \mathcal{E}^\flat_\mathbb{B}\left(  \mathcal{E}(f), \mathcal{E}(g^{\dagger\dagger}) \right) \right) \overset{\text{(\ref{diag:cofree})}}{=} \mathcal{E}_\mathbb{B}\left(  \mathcal{E}(f), \mathcal{E}(g^{\dagger\dagger}) \right) \overset{\overset{\text{Prop}}{\text{\ref{prop:canon-cofree}}}}{=} \mathcal{E}(f) \overset{\overset{\text{Prop}}{\text{\ref{prop:canon-cofree}}}}{=} \mathcal{E}\left( \mathcal{E}_\mathbb{D}(f,g) \right)
\end{gather*}
Thus $\mathcal{E} \circ \mathcal{D}_{\ast_\mathcal{E}} = \mathcal{E} \circ \mathcal{E}_\mathbb{D}$. However by uniqueness, we then get that $\mathcal{D}_\mathcal{E} = \mathcal{D}_{\ast_\mathcal{E}}$ as desired. Then by Thm \ref{thm:coag=rect}, it follows that $\ast_{\mathcal{D}_\mathcal{E}} = \ast_{\mathcal{E}}$ as well. 
\end{proof}

\begin{corollary} Let $\mathbb{B}$ be a category. Then for $(\mathsf{C}[\mathbb{B}], \dagger)$, we have that $\mathcal{D}_\mathbb{B} = \mathcal{D}_{\ast_\mathbb{B}}$.
\end{corollary}
\begin{proof} This follows from Cor \ref{cor:astE=3-canon} and applying Cor \ref{cor:Dast=} to $(\mathsf{C}[\mathbb{B}], \dagger, \mathcal{E}_\mathbb{B})$. 
\end{proof}

Therefore, bringing all of this together, we obtain our main result about cofree dagger categories. 

\begin{theorem}\label{thm:full-cofree} For a dagger category, the following are equivalent: 
\begin{enumerate}[{\em (i)}]
\item It is cofree;
\item It is rectangular banded; 
\item It is a $\mathsf{C}$-algebra. 
\end{enumerate}
\end{theorem}

We conclude this section with the observation that Thm \ref{thm:coag=rect} can easily be extended to an isomorphism between the (large) category of rectangular banded dagger categories and the category of $\mathsf{C}$-algebras, where the latter is often called the Eilenberg-Moore category. 

\begin{definition} For rectangular banded dagger categories $(\mathbb{D}_1, \dagger, \ast)$ and $(\mathbb{D}_2, \dagger, \ast)$, a \textbf{rectangular banded dagger functor} ${\mathcal{F}: (\mathbb{D}_1, \dagger,\ast) \to (\mathbb{D}_2, \dagger,\ast)}$ is a dagger functor ${\mathcal{F}: (\mathbb{D}_1, \dagger) \to (\mathbb{D}_2, \dagger)}$ which preserves the rectangular band operation, that is, for every map $f:A \to B$ and $g: A\to B$ in $\mathbb{D}_1$, the following equality holds: 
\begin{align}\label{eq:band-fun}
\mathcal{F}(f\ast g)= \mathcal{F}(f) \ast \mathcal{F}(g)
\end{align}
Let $\mathsf{RBDAG}$ be the (large) category of rectangular banded dagger categories and rectangular banded dagger functors between them. 
\end{definition}

\begin{definition} For $\mathsf{C}$-algebras $(\mathbb{D}_1, \dagger, \mathcal{D}_1)$ and $(\mathbb{D}_2, \dagger, \mathcal{D}_2)$, a \textbf{$\mathsf{C}$-algebra functor} ${\mathcal{F}: (\mathbb{D}_1, \dagger,\mathcal{D}_1) \to (\mathbb{D}_2, \dagger,\mathcal{D}_2)}$ is a dagger functor ${\mathcal{F}: (\mathbb{D}_1, \dagger) \to (\mathbb{D}_2, \dagger)}$ such that the following diagram commutes: 
\begin{equation}\begin{gathered}\label{diag:C-alg-fun}   \xymatrixcolsep{5pc}\xymatrix{   \mathsf{C}[\mathbb{D}_1] \ar[r]^-{\mathsf{C}[\mathcal{F}]}  \ar[d]_-{\mathcal{D}_1} &   \mathsf{C}[\mathbb{D}_2] \ar[d]^-{\mathcal{D}_2}    \\
\mathbb{D}_1  \ar[r]_-{\mathcal{F}} & \mathbb{D}_2
 }
\end{gathered}\end{equation}
Let $\mathsf{C}\text{-}\mathsf{ALG}$ be the (large) category of $\mathsf{C}$-algebras and $\mathsf{C}$-algebra functors between them. 
\end{definition}

\begin{lemma} Let ${\mathcal{F}: (\mathbb{D}_1, \dagger,\ast) \to (\mathbb{D}_2, \dagger,\ast)}$ be a rectangular banded functor. Then ${\mathcal{F}: (\mathbb{D}_1, \dagger,\mathcal{D}_\ast) \to (\mathbb{D}_2, \dagger,\mathcal{D}_\ast)}$ is a $\mathsf{C}$-algebra functor. 
\end{lemma}
\begin{proof} We compute: 
\begin{gather*}
\mathcal{D}_\ast\left( \mathsf{C}[\mathcal{F}](f,g) \right) \overset{\text{(\ref{def:CF}}}{=} \mathcal{D}_\ast\left( \mathcal{F}(f), \mathcal{F}(g) \right) \overset{\overset{\text{Prop}}{\text{\ref{prop:rect-calg}}}}{=} \mathcal{F}(f) \ast \mathcal{F}(g)^\dagger \overset{\text{(\ref{def:dagfun})}}{=} \mathcal{F}(f) \ast \mathcal{F}(g^\dagger) \overset{\text{(\ref{eq:band-fun})}}{=}
 \mathcal{F}(f \ast g^\dagger) \overset{\overset{\text{Prop}}{\text{\ref{prop:rect-calg}}}}{=} \mathcal{F}\left( \mathcal{D}_\ast(f,g) \right)
\end{gather*}
Thus it follows that $\mathcal{D}_\ast \circ \mathsf{C}[\mathcal{F}] = \mathcal{F} \circ \mathcal{D}_\ast$. So ${\mathcal{F}: (\mathbb{D}_1, \dagger,\mathcal{D}_\ast) \to (\mathbb{D}_2, \dagger,\mathcal{D}_\ast)}$ is a $\mathsf{C}$-algebra functor. 
\end{proof}

\begin{lemma} Let ${\mathcal{F}: (\mathbb{D}_1, \dagger,\mathcal{D}_1) \to (\mathbb{D}_2, \dagger,\mathcal{D}_2)}$ be a $\mathsf{C}$-algebra functor. Then ${\mathcal{F}: (\mathbb{D}_1, \dagger,\ast_{\mathcal{D}_1}) \to (\mathbb{D}_2, \dagger,\ast_{\mathcal{D}_2})}$ is a rectangular banded dagger functor.
\end{lemma}
\begin{proof} Observe that (\ref{diag:C-alg-fun}) says that on maps, we have that the following equality holds: 
\begin{align}\label{eq:C-alg-fun}
\mathcal{D}_2\left( \mathcal{F}(f), \mathcal{F}(g) \right) = \mathcal{F}\left( \mathcal{D}_1(f,g) \right)
\end{align}
Then we compute: 
\begin{gather*} 
\mathcal{F}(f \ast_{\mathcal{D}_1} g) \overset{\text{(\ref{def:ast-D}}}{=} \mathcal{F}\left(\mathcal{D}_1(f,g^\dagger) \right) \overset{\text{(\ref{eq:C-alg-fun}}}{=} \mathcal{D}_2\left( \mathcal{F}(f), \mathcal{F}(g^\dagger) \right) \overset{\text{(\ref{def:dagfun})}}{=} \mathcal{D}_2\left( \mathcal{F}(f), \mathcal{F}(g)^\dagger \right) \overset{\text{(\ref{def:ast-D}}}{=} \mathcal{F}(f) \ast_{\mathcal{D}_2} \mathcal{F}(g)
\end{gather*}
Thus we have that ${\mathcal{F}: (\mathbb{D}_1, \dagger,\ast_{\mathcal{D}_1}) \to (\mathbb{D}_2, \dagger,\ast_{\mathcal{D}_2})}$ be a rectangular banded dagger functor.
\end{proof}

By combining the above lemmas with Thm \ref{thm:coag=rect}, we conclude that: 

\begin{theorem} $\mathsf{RBDAG} \simeq \mathsf{C}\text{-}\mathsf{ALG}$
\end{theorem}

It is worth mentioning that the forgetful functor $\mathsf{U}: \mathsf{DAG} \to \mathsf{CAT}$ is in fact comonadic \cite[Thm 2.1.11]{karvonen2019way},  which means that the coalgebras of the induced comonad on $\mathsf{CAT}$ are precisely dagger categories, or in other words, the coEilenberg-Moore category of this induced comonad is equivalent to $\mathsf{DAG}$. However the forgetful functor's right adjoint $\mathsf{C}: \mathsf{CAT} \to \mathsf{DAG}$ is not comonadic since the comparison functor between $\mathsf{CAT}$ and $\mathsf{C}\text{-}\mathsf{ALG}$ is not an equivalence (while it is essentially surjective and full, it is not faithful). 

\section{Free Dagger Categories}\label{sec:free}


In this section, we turn our attention to giving base independent characterizations of free dagger categories. As before, we begin by first reviewing the definition of a free dagger category via its universal property over a base category. 

\begin{definition} A \textbf{free dagger category} over a category $\mathbb{B}$ is a triple $(\mathbb{D}, \dagger, \mathcal{N})$ consisting of a dagger category $(\mathbb{D}, \dagger)$ and a functor $\mathcal{N}: \mathbb{B} \to \mathbb{D}$, such that for every dagger category $(\mathbb{C}, \dagger)$ and functor $\mathcal{F}: \mathbb{B} \to \mathbb{C}$, there exists a unique dagger functor $\mathcal{F}^\sharp: (\mathbb{D}, \dagger) \to (\mathbb{C}, \dagger)$ such that the following diagram commutes: 
\begin{equation}\begin{gathered}\label{diag:free}  \xymatrixcolsep{5pc}\xymatrix{\mathbb{D}  \ar@{-->}[r]^-{\exists! ~ \mathcal{F}^\sharp}  & \mathbb{C}  \\
   \mathbb{B} \ar[u]^-{\mathcal{N}} \ar[ur]_-{\mathcal{F}} }
\end{gathered}\end{equation}
A dagger category $(\mathbb{D}, \dagger)$ is said to be \textbf{free} if there exists a category $\mathbb{B}$ and a functor $\mathcal{N}: \mathbb{B} \to \mathbb{D}$ such that $(\mathbb{D}, \dagger, \mathcal{N})$ is a free dagger category over $\mathbb{B}$. In this case, we call $\mathbb{B}$ a \textbf{base category} of the free dagger category $(\mathbb{D}, \dagger)$.
\end{definition}

As we will review shortly below in Sec \ref{sec:free-canonical}, for any category $\mathbb{B}$, there exists a free dagger category over it. Before doing so, we record important facts about free dagger categories that follow immediately from the above universal property, which are dual to the ones we saw in Lemma \ref{lemma:cofree-iso} for cofree dagger categories. 

\begin{lemma}\label{lemma:free-iso} Let $\mathbb{B}$ be a category. 
\begin{enumerate}[{\em (i)}]
\item If $(\mathbb{D}_1, \dagger, \mathcal{N}_1)$ and $(\mathbb{D}_2, \dagger, \mathcal{N}_2)$ are both free dagger categories over $\mathbb{B}$, then there exists a unique dagger isomorphism ${\mathcal{F}: (\mathbb{D}_1, \dagger) \to (\mathbb{D}_2, \dagger)}$ such that the following diagram commutes: 
\begin{equation}\begin{gathered}\label{diag:free-iso}
\xymatrixcolsep{5pc}\xymatrix{\mathbb{D}_1   \ar@{-->}[rr]^-{\exists! ~ \mathcal{F}}  && \mathbb{D}_2  \\
  & \ar[ul]^-{\mathcal{N}_1} \mathbb{B} \ar[ur]_-{\mathcal{N}_2} }
  \end{gathered}\end{equation}
\item If $(\mathbb{D}_1, \dagger, \mathcal{N}_1)$ is a free dagger category over $\mathbb{B}$ and ${\mathcal{F}: (\mathbb{D}_1, \dagger) \to (\mathbb{D}_2, \dagger)}$ is a dagger isomorphism, then $(\mathbb{D}_2, \dagger, \mathcal{F}\circ \mathcal{N})$ is a free dagger category over $\mathbb{B}$. 
\item If $(\mathbb{D}, \dagger, \mathcal{N})$ is a free dagger category over $\mathbb{B}$ and $\mathcal{F}: \mathbb{B} \to \mathbb{B}_2$ is an isomorphism, then $(\mathbb{D}, \dagger, \mathcal{N} \circ \mathcal{F}^{-1})$ is a free dagger category over $\mathbb{B}_2$. 
\end{enumerate}
\end{lemma}

Once again, it is important to stress that a dagger category can be free over different non-equivalent base categories. Dual to cofree dagger categories, we also get that a free dagger category over a specified base is also free over the opposite category of its specified base category.

\begin{lemma}\label{lemma:free-opp-base} Let $(\mathbb{D}, \dagger, \mathcal{N})$ be a free dagger category over a category $\mathbb{B}$. Define the functor $\mathcal{N}^o: \mathbb{B}^{op} \to \mathbb{D}$ on objects as $\mathcal{N}^o(A) = \mathcal{N}(A)$ and on maps as $\mathcal{N}^o(f) = \mathcal{N}(f)^\dagger$. Then $(\mathbb{D}, \dagger, \mathcal{N}^o)$ is a free dagger category over $\mathbb{B}^{op}$.
\end{lemma}
\begin{proof} This follows essentially from the dual argument of the proof of Lemma \ref{lemma:cofree-opp-base}. 
\end{proof}

Recall that for cofree dagger categories, we saw that the image of the functor down to the starting base category provides another base category for our cofree dagger category. The same is true for free dagger categories. However, while for cofree dagger categories, this image category construction might result in a different base category than our starting base category, for free dagger categories, it in fact recaptures the starting base category (up to isomorphism). We will look at this in more detail after we review canonical free dagger categories in the next section. 

\subsection{ZigZag Construction}\label{sec:free-canonical}

We now review the canonical construction of a free dagger category over a specified category. Doing so will require a bit of setup. So let $\mathbb{B}$ be a category. For every pair of objects $A, B \in \mathsf{Obj}(\mathbb{B})$, let $\mathsf{pre}\text{-}\mathsf{Z}[\mathbb{B}](A,B)$ be the set of non-empty ordered even length lists of maps:
\[ ({}_n f, f_n,  \hdots, {}_i f, f_i, \hdots, {}_1f, f_1) \]
where for $1 \leq i \leq n$, setting ${}_0 A = A$ and $A_n =B$, the maps $f_{i}$ and ${}_{i}f$ are of type: 
\begin{align}
f_{i}: A_{i-1} \to {}_iA && {}_{i}f: A_{i+1} \to {}_iA
\end{align}
Now we can define a composition operation $\circ: \mathsf{pre}\text{-}\mathsf{Z}[\mathbb{B}](B,C) \times \mathsf{pre}\text{-}\mathsf{Z}[\mathbb{B}](A,B) \to \mathsf{pre}\text{-}\mathsf{Z}[\mathbb{B}](A,C)$ given simply by concatenation of lists: 
\begin{align}
({}_m g, g_m, \hdots, {}_1 g, g_1) \circ ({}_n f, f_n, \hdots, {}_1f, f_1) = ({}_m g, g_m, \hdots, {}_1 g, g_1, {}_n f, f_n, \hdots, {}_1f, f_1) 
\end{align}
We can also define a dagger operation $\dagger: \mathsf{pre}\text{-}\mathsf{Z}[\mathbb{B}](A,B) \to \mathsf{pre}\text{-}\mathsf{Z}[\mathbb{B}](B,A)$ given by simply reversing the order of the list: 
\begin{align}
({}_n f, f_n,  \hdots, {}_i f, f_i, \hdots, {}_1f, f_1)^\dagger = ({}_m g, g_m, \hdots, {}_1 g, g_1, {}_n f, f_n, \hdots, {}_1f, f_1) 
\end{align}
For identities, one would like the list $(\mathsf{id}_A,\mathsf{id}_A)$ to be the identity maps. However, these pairs will not be units for the concatenation of lists. We will need to impose certain identifications to make this happen. So for every pair of objects $A$ and $B$, let $\sim$ be the smallest equivalence relation on $\mathsf{pre}\text{-}\mathsf{Z}[\mathbb{B}](A,B)$ which satisfies the following (for all suitable lists of maps and $1 < i \leq n$): 
\begin{align}
({}_n f, f_n,  \hdots, {}_i f, f_i, \mathsf{id}_{A_{i-1}}, f_{i-1}\hdots, {}_1f, f_1) \sim ({}_n f, f_n,  \hdots, {}_i f, f_i \circ f_{i-1}, \hdots, {}_1f, f_1) \label{reduce-1} \\
({}_n f, f_n,  \hdots, {}_i f, \mathsf{id}_{A_{i-1}}, {}_{i-1}f, f_{i-1}\hdots, {}_1f, f_1) \sim ({}_n f, f_n,  \hdots, {}_{i-1}f \circ {}_i f, f_{i-1}\hdots, {}_1f, f_1) \label{reduce-2}
\end{align}
Note that the above is indeed well-typed, since if ${}_{i-1}f = \mathsf{id}_{A_{i-1}}$, this means $A_{i-1} = {}_{i-1}A$, so we can compose $f_{i-1}: A_{i-2} \to {}_{i-1}A=A_{i-1}$ and $f_i: A_{i-1} \to {}_i A$, while if $f_i = \mathsf{id}_{A_{i-1}}$, this means $A_{i-1}={}_i A$, so we can compose ${}_i f: A_{i+1} \to {}_iA=A_{i-1}$ and ${}_{i-1} f: A_{i-1} \to {}_{i-1} A$. Essentially, (\ref{reduce-1}) and (\ref{reduce-2}) tell us that we can remove identity maps in the middle of lists.  As a shorthand, we shall denote the equivalence class of a list $({}_n f, f_n, \hdots, {}_1f, f_1)$ using square brackets: 
\[ [{}_n f, f_n,  \hdots, {}_1f, f_1] = \left[ ({}_n f, f_n,  \hdots, {}_1f, f_1) \right]_\sim \]

Then define the dagger category $(\mathsf{Z}[\mathbb{B}], \dagger)$ \cite[Def 3.1.18]{heunen2009categorical} as follows: 
\begin{itemize}
\item The objects of $\mathsf{Z}[\mathbb{B}]$ are the same as $\mathbb{B}$;
\item Maps in $\mathsf{Z}[\mathbb{B}]$ are equivalence classes $[{}_n f, f_n,  \hdots, {}_1f, f_1]: A \to B$ of lists $({}_n f, f_n,  \hdots, {}_1f, f_1) \in \mathsf{pre}\text{-}\mathsf{Z}[\mathbb{B}](A,B)$, in other words, the homsets are the quotient sets $\mathsf{Z}[\mathbb{B}](A,B) := \mathsf{pre}\text{-}\mathsf{Z}[\mathbb{B}](A,B)/ \sim$;
\item Identity maps are $[\mathsf{id}_A, \mathsf{id}_A]: A \to A$;
\item Composition is defined via concatenation as follows: 
\begin{align}\label{def:comp-free}
[{}_m g, g_m, \hdots, {}_1 g, g_1] \circ [{}_n f, f_n, \hdots, {}_1f, f_1] = [{}_m g, g_m, \hdots, {}_1 g, g_1, {}_n f, f_n, \hdots, {}_1f, f_1]
\end{align}
\item The dagger is defined as follows:
\begin{align}\label{def:dag-free}
[{}_n f, f_n, \hdots, {}_i f, f_i, \hdots, {}_1f, f_1]^\dagger = [f_1, {}_1f, \hdots, f_i, {}_i f, \hdots, f_n, {}_n f]
\end{align}
\end{itemize}
Define the functor $\mathcal{N}_\mathbb{B}: \mathbb{B} \to \mathsf{Z}[\mathbb{B}]$ on objects as $\mathcal{N}_\mathbb{B}(A) = A$ and on maps as $\mathcal{N}_\mathbb{B}(f) = [\mathsf{id}_B, f]$.  

\begin{proposition}\label{prop:canon-free} \cite[Thm 3.1.19]{heunen2009categorical} $(\mathsf{Z}[\mathbb{B}], \dagger, \mathcal{N}_\mathbb{B})$ is a free dagger category over $\mathbb{B}$. Explicitly, for any dagger category $(\mathbb{C}, \dagger)$ and functor $\mathcal{F}: \mathbb{B} \to \mathbb{C}$, the unique dagger $\mathcal{F}^\sharp: (\mathsf{Z}[\mathbb{B}], \dagger) \to (\mathbb{C}, \dagger)$ such that the following diagram commutes: 
\begin{equation}\begin{gathered}\label{diag:free-canon}  \xymatrixcolsep{5pc}\xymatrix{\mathsf{Z}[\mathbb{B}]  \ar@{-->}[r]^-{\exists! ~ \mathcal{F}^\sharp}  & \mathbb{C}  \\
   \mathbb{B} \ar[u]^-{\mathcal{N}_\mathbb{B}} \ar[ur]_-{\mathcal{F}} }
\end{gathered}\end{equation}
is defined on objects as $\mathcal{F}^\sharp(A) = \mathcal{F}(A)$ and on maps as follows: 
\begin{align}
\mathcal{F}^\sharp\left([{}_n f, f_n, \hdots, {}_i f, f_i, \hdots, {}_1f, f_1] \right) = \mathcal{F}({}_nf)^\dagger  \circ \mathcal{F} (f_n) \circ \hdots \circ \mathcal{F}({}_i f)^\dagger  \circ \mathcal{F} (f_i) \circ \hdots \circ \mathcal{F}({}_1f)^\dagger  \circ \mathcal{F} (f_1)
\end{align}
\end{proposition}

By applying Lemma \ref{lemma:free-iso}, we get that free dagger categories can be characterized as the dagger categories which are dagger isomorphic to canonical free dagger categories. 

\begin{corollary}\label{cor:free-iso.1} If $(\mathbb{D}, \dagger, \mathcal{N})$ is a free dagger category over a category $\mathbb{B}$, then $\mathcal{N}_\mathbb{B}^\sharp: (\mathbb{D}, \dagger) \to (\mathsf{Z}[\mathbb{B}], \dagger)$ is a dagger isomorphism with inverse $\mathcal{N}^\sharp: (\mathsf{Z}[\mathbb{B}], \dagger) \to (\mathbb{D}, \dagger)$. So $(\mathbb{D}, \dagger) \cong (\mathsf{Z}[\mathbb{B}], \dagger)$. 
\end{corollary}

\begin{corollary}\label{cor:free-iso.2} A dagger category $(\mathbb{D}, \dagger)$ is free if and only if there exists a category $\mathbb{B}$ and a dagger isomorphism $(\mathbb{D}, \dagger) \cong (\mathsf{Z}[\mathbb{B}], \dagger)$. 
\end{corollary}

The above characterization still requires one to specify an external base category. The objective of the next two subsections is provide base independent characterizations of free dagger categories. 

As with cofree dagger categories, we can also use the canonical construction of free dagger categories to show that free dagger categories essentially have the same objects as their base category, and also that all possible base categories have isomorphic classes of objects. 

\begin{lemma}\label{lemma:free-objects} Let $(\mathbb{D}, \dagger)$ be a dagger category. 
\begin{enumerate}[{\em (i)}]
\item If $(\mathbb{D}, \dagger, \mathcal{N})$ is a free dagger category over a category $\mathbb{B}$, then $\mathcal{N}: \mathsf{Obj}(\mathbb{D}) \to \mathsf{Obj}(\mathbb{B})$ is an isomorphism, so we have that $\mathsf{Obj}(\mathbb{B}) \cong \mathsf{Obj}(\mathbb{D})$.
\item If $(\mathbb{D}, \dagger, \mathcal{N}_1)$ is a cofree dagger category over a category $\mathbb{B}_1$ and $(\mathbb{D}, \dagger, \mathcal{N}_2)$ is also a cofree dagger category over a category $\mathbb{B}_2$, then we have an isomorphism $\mathsf{Obj}(\mathbb{B}_1) \cong \mathsf{Obj}(\mathbb{B}_2)$.
\end{enumerate}
\end{lemma}
\begin{proof} Given by dual argument to the proof of Lemma \ref{lemma:cofree-objects}. 
\end{proof}

We turn our attention to arguably the most important fact about $(\mathsf{Z}[\mathbb{B}], \dagger)$. Morphisms in $(\mathsf{Z}[\mathbb{B}], \dagger)$ have a canonical normal form where by using (\ref{reduce-1}) and (\ref{reduce-2}), we can always find a representative of the equivalence class which is a list with no identity maps in it, except perhaps the first and last term. So we say that a list $({}_n f, f_n, \hdots, {}_1f, f_1) \in \mathsf{pre}\text{-}\mathsf{Z}[\mathbb{B}](A,B)$ is \textbf{normal} if for all $1 \leq i < n$ and $1 < j \leq n$, ${}_i f$ and $f_j$ are not identity maps. 

\begin{proposition}\label{prop:normal}\cite[Page 55]{heunen2009categorical} Let $F \in \mathsf{Z}[\mathbb{B}](A,B)$. Then there exists a unique normal list $({}_n f, f_n, \hdots, {}_1f, f_1) \in \mathsf{pre}\text{-}\mathsf{Z}[\mathbb{B}](A,B)$ such that $F = [{}_n f, f_n, \hdots, {}_1f, f_1]$. Moreover,
\begin{enumerate}[{\em (i)}]
\item If $[{}_m g, g_m, \hdots, {}_1 g, g_1] = [{}_n f, f_n, \hdots, {}_1f, f_1]$, then $n \leq m$. 
\item If $[{}_n g, g_n, \hdots, {}_1 g, g_1] = [{}_n f, f_n, \hdots, {}_1f, f_1]$, then ${}_ig = {}_if$ and $g_i=f_i$ for all $1 \leq i \leq n$.
\end{enumerate}
\end{proposition}

For a morphism $F \in \mathsf{Z}[\mathbb{B}](A,B)$, if $({}_n f, f_n, \hdots, {}_1f, f_1)$ is the (necessarily unique) normal list such that $F = [{}_n f, f_n, \hdots, {}_1f, f_1]$, we call $[{}_n f, f_n, \hdots, {}_1f, f_1]$ the \textbf{normal form} of $F$. This normal form presentation for maps in our canonical free dagger category will be the inspiration for how we characterize free dagger categories in the next section. 

A consequence of this normal form is that we get that the functor from the base category into our canonical free dagger category is faithful. In fact, this is true for any free dagger category. 

\begin{lemma}\label{lemma:N-faithful} $\mathcal{N}_\mathbb{B}: \mathbb{B} \to \mathsf{Z}[\mathbb{B}]$ is faithful. 
\end{lemma}
\begin{proof} On maps $f: A \to B$ and $g: A \to B$ in $\mathbb{B}$, suppose that $\mathcal{N}_\mathbb{B}(f) =\mathcal{N}_\mathbb{B}(g)$, that is, $[\mathsf{id}_B,f] = [\mathsf{id}_B,g]$. Note that the lists $(\mathsf{id}_B,f)$ and $(\mathsf{id}_B,g)$ are both normal (as the condition of the middle terms not being identity maps is vacuously true). Therefore by uniqueness of normal forms (Prop \ref{prop:normal}), we must have that $(\mathsf{id}_B,f) = (\mathsf{id}_B,g)$, which implies that $f=g$. So we conclude that $\mathcal{N}_\mathbb{B}$ is indeed faithful. 
\end{proof}

\begin{corollary}\label{cor:N-faithful} Let $(\mathbb{D}, \dagger, \mathcal{N})$ be a free dagger category over a category $\mathbb{B}$. Then $\mathcal{N}$ is faithful.
\end{corollary}
\begin{proof} By universal property of canonical free dagger categories (\ref{diag:free-canon}), we have that $\mathbb{N} = \mathbb{N}^\sharp \circ \mathbb{N}_\mathbb{B}$. But we know form Cor \ref{cor:cofree-iso.1} that $\mathbb{N}^\sharp$ is an isomorphism and hence also faithful, while by Lemma \ref{lemma:N-faithful}, we have that $\mathcal{N}_\mathbb{B}$ is faithful. Since the composite of faithful functors is again faithful, it follows that $\mathbb{N}$ is faithful. 
\end{proof}

Now faithful functors are an equivalent presentation of subcategories. Therefore, the above corollary tells us that the base category is isomorphic to a subcategory of its free dagger category. Specifically, our base category will be isomorphic to the image category of the functor into its free dagger category. Then by applying Lemma \ref{lemma:free-iso}, we get that free dagger categories are also free over this subcategory. So let $(\mathbb{D},\dagger, \mathcal{N})$ be a free dagger category over a category $\mathbb{B}$. Let $\mathsf{Im}[\mathcal{N}]$ be the image category of $\mathcal{N}$, that is, the category whose objects are of form $\mathcal{N}(A)$ for all $A \in \mathsf{Obj}(\mathbb{B})$ and whose maps are of the form $\mathcal{N}(f): \mathcal{N}(A) \to \mathcal{N}(B)$ for all $f: A \to B \in \mathsf{Map}(\mathbb{B})$. Let $\mathcal{I}_\mathcal{N}: \mathsf{Im}[\mathcal{N}] \to \mathbb{D}$ be the inclusion functor and $\overline{\mathcal{N}}: \mathbb{B} \to \mathsf{Im}[\mathcal{N}]$ be the obvious image functor, so we have that $\mathcal{N} = \mathcal{I}_\mathcal{N} \circ \overline{\mathcal{N}}$. 

\begin{lemma}Let $(\mathbb{D}, \dagger, \mathcal{N})$ be a free dagger category over a category $\mathbb{B}$. Then $\overline{\mathcal{N}}: \mathbb{B} \to \mathsf{Im}[\mathcal{N}]$ is an isomorphism, so $\mathbb{B} \cong \mathsf{Im}[\mathcal{N}]$. Therefore, $(\mathbb{D},\dagger, \overline{\mathcal{N}})$ is a free dagger category over $\mathsf{Im}[\mathcal{N}]$. 
\end{lemma}
\begin{proof} This follows directly from Cor \ref{cor:N-faithful} and then Lemma \ref{lemma:free-iso}.
\end{proof}

In particular for a canonical free dagger category, we have that $\mathsf{Im}[\mathcal{N}_\mathbb{B}]$ is the category with the same objects as $\mathbb{B}$ and whose morphisms of the form $[\mathsf{id}_B, f]$ for every map $f \in \mathsf{Map}(\mathbb{B})$. Thus 

\begin{corollary} For any category $\mathbb{B}$, we have that $\mathbb{B} \cong \mathsf{Im}[\mathcal{N}_\mathbb{B}]$ and $(\mathsf{Z}[\mathbb{B}], \dagger) \cong (\mathsf{Z}\left[ \mathsf{Im}[\mathcal{N}_\mathbb{B}] \right], \dagger)$.  
\end{corollary}

Thus, while in the cofree scenario, cofree dagger categories were also cofree over a subcategory of their specified base category, here in the free scenario, free dagger categories are also free over a subcategory of the dagger category itself. 

\subsection{ZigZag Decomposition}\label{sec:zigzag-decomp}

In this section, we present a direct characterization of free dagger categories by introducing the notion of a \textit{zigzag dagger category}. Briefly, a zigzag dagger category is one with a chosen subclass of maps, which we call zigs and whose adjoints are called zags, such that every map decomposes into an alternating composition of zigs and zags. This decomposition corresponds to the normal form of maps in canonical free dagger categories. We will then show that a zigzag dagger category is free over its subcategory of zig maps, and conversely that every free dagger category is a zigzag dagger category whose zig maps are essentially maps of the chosen base category. 

\begin{definition}\label{def:zigzag-dag-cat} A \textbf{zigzag dagger category} is a triple $(\mathbb{D}, \dagger, \mathsf{Zig})$ consisting of a dagger category $(\mathbb{D}, \dagger)$ with a subclass of maps $\mathsf{Zig} \subseteq \mathsf{Map}(\mathbb{D})$ where we say that a map $f \in \mathsf{Map}(\mathbb{D})$ is:  
\begin{enumerate}[{\em (i)}]
\item A \textbf{zig} if $f \in \mathsf{Zig}$;
\item A \textbf{zag} if $f^\dagger$ is a zig;
\item A \textbf{true zig} (resp. \textbf{true zag}) if $f$ is a zig (resp. zag) and not also a zag (resp. zig), 
\end{enumerate}
and such that this subclass of maps $\mathsf{Zig}$ satisfies the following: 
\begin{enumerate}[{\em (i)}]
\item Composition: if $f: A \to B$ and $g: B \to C$ are zigs then their composition $g \circ f: A \to C$ is a zig;
\item Identities: for every object $A$, $\mathsf{id}_A: A \to A$ is a zig;
\item Decomposition: For every map $f: A \to B$, there exists a unique ordered even length non-empty list of maps:
\begin{align}
\mathcal{Z}[f]=({}_n\mathcal{Z}[f], \mathcal{Z}[f]_n, \hdots, {}_{i}\mathcal{Z}[f], \mathcal{Z}[f]_i, \hdots, {}_1\mathcal{Z}[f],\mathcal{Z}[f]_1)
\end{align}
where for $1 \leq i \leq n$, setting ${}_0 A = A$ and $A_n =B$, the maps $\mathcal{Z}[f]_{i}$ and ${}_{i}\mathcal{Z}[f]$ are of type: 
\begin{align}
\mathcal{Z}[f]_{i}: A_{i-1} \to {}_iA && {}_{i}\mathcal{Z}[f]: {}_iA \to A_{i+1}
\end{align}
and such that: 
\begin{enumerate}
\item $\mathcal{Z}[f]_1$ is a zig;
\item ${}_n\mathcal{Z}[f]$ is a zag;
\item For every $1<i\leq n$, $\mathcal{Z}[f]_i$ is a true zig;
\item For every $1\leq i < n$, ${}_i\mathcal{Z}[f]$ is a true zag,
\end{enumerate}
and the following equality holds: 
\begin{align}\label{eq:zigzag-decomp}
f = {}_n\mathcal{Z}[f] \circ \mathcal{Z}[f]_{n} \circ \hdots \circ {}_i\mathcal{Z}[f] \circ \mathcal{Z}[f]_{i}  \circ \hdots \circ {}_1\mathcal{Z}[f] \circ \mathcal{Z}[f]_1
\end{align}
\end{enumerate}
We call this list $\mathcal{Z}[f]$ the \textbf{zigzag decomposition} of $f$. 
\end{definition}

It will be first be useful to record that a map is a zig (resp. zag) if and only if its adjoints is a zag (resp. zig). 

\begin{lemma}\label{lemma:zigzagiff} For a map $f: A \to B$ in a zigzag dagger category $(\mathbb{D}, \dagger, \mathsf{Zig})$,
\begin{enumerate}[{\em (i)}]
\item \label{lemma:zigzagiff.1} $f$ is a zig if and only if $f^\dagger$ is a zag;
\item \label{lemma:zigzagiff.2} $f$ is a zag if and only if $f^\dagger$ is a zig;
\item \label{lemma:zigzagiff.3} $f$ is a true zig if and only if $f$ is a zig and $f^\dagger$ is not a zig;
\item \label{lemma:zigzagiff.4} $f$ is a true zag if and only if $f$ is a zag and $f^\dagger$ is not a zag. 
\end{enumerate}
\end{lemma}
\begin{proof} This follows from the involution property of the dagger. Indeed, for (\ref{lemma:zigzagiff.1}), for the $\Leftarrow$ direction, suppose that $f^\dagger$ is a zag. By definition, this means $f^{\dagger \dagger}$ is a zig. Of course, since $f^{\dagger \dagger} = f$, this means $f$ is a zig. Conversely, for the $\Rightarrow$ direction, suppose $f$ is a zig. Then considering $f^\dagger$, since $f = f^{\dagger \dagger}$, which means $f^{\dagger \dagger}$ is a zig, by definition this means $f^\dagger$ is a zag. Then (\ref{lemma:zigzagiff.2}) follows by applying (\ref{lemma:zigzagiff.1}) to $f^\dagger$. For (\ref{lemma:zigzagiff.3}) (resp. (\ref{lemma:zigzagiff.4})), by definition, $f$ is a true zig (resp. zag) if and only if $f$ is a zig (resp. zag) and $f$ is not a zag (resp. zig), which by (\ref{lemma:zigzagiff.2}) (resp (\ref{lemma:zigzagiff.1})), is equivalent to saying $f$ is a true zig (resp. zag) if and only if $f$ is a zig (resp. zag) and $f^\dagger$ is not a zig (resp. zag). 
\end{proof}

From this it follows that identity maps are both zigs and zags, so they are not true zigs or true zags, and also that their zigzag decomposition is just the identity paired with its self. 

\begin{lemma}\label{lemma:zigzagid} In a zigzag dagger category $(\mathbb{D}, \dagger, \mathsf{Zig})$, for every object $A$:
\begin{enumerate}[{\em (i)}]
\item \label{lemma:zigzagid.1} $\mathsf{id}_A: A \to A$ is both a zig and zag;
\item \label{lemma:zigzagid.2} $\mathsf{id}_A$ is neither a true zig nor a true zag;
\item \label{lemma:zigzagid.3} $\mathcal{Z}[\mathsf{id}_A]=(\mathsf{id}_A, \mathsf{id}_A)$.
\end{enumerate}
\end{lemma}
\begin{proof} For (\ref{lemma:zigzagid.1}), by definition of a zigzag dagger category, we already know that $\mathsf{id}_A$ is a zig. So we need to explain why it is also a zag. This follows from the fact that daggers preserve identities. Indeed, since $\mathsf{id}_A$ is a zig, then by Lem \ref{lemma:zigzagiff}.((\ref{lemma:zigzagiff.1})), we have that $\mathsf{id}_A^\dagger$ is a zag. However since $\mathsf{id}_A^\dagger = \mathsf{id}_A$, we have that $\mathsf{id}_A$ is also a zag. For (\ref{lemma:zigzagid.2}), since $\mathsf{id}_A$ is both a zig and zag, by definition it cannot also be a true zig nor a true zag. Now for (\ref{lemma:zigzagid.3}), since $\mathsf{id}_A$ is both a zig and zag, $(\mathsf{id}_A, \mathsf{id}_A)$ is an even length list (with $n=1$) where the first entry is a zig and the last entry is a zag, with the requirement of true zigs and true zags being vacuously true, and trivially we have that $\mathsf{id}_A \circ \mathsf{id}_A = \mathsf{id}_A$. Therefore, by uniqueness of zigzag decompositions, we have that $\mathcal{Z}[\mathsf{id}_A]=(\mathsf{id}_A, \mathsf{id}_A)$.  
\end{proof}

Then the zigzag decomposition of a zig (resp. zag) is the pairing of itself with the identity map. 

\begin{lemma}\label{lemma:zigzag-zigzag-decomp} For a map $f: A \to B$ in a zigzag dagger category $(\mathbb{D}, \dagger, \mathsf{Zig})$,
\begin{enumerate}[{\em (i)}]
\item \label{lemma:zigzag-zigzag-decomp.1} $f$ is a zig if and only if $\mathcal{Z}[f]=(\mathsf{id}_B,f)$;
\item \label{lemma:zigzag-zigzag-decomp.2} $f$ is a zag if and only if $\mathcal{Z}[f]=(f,\mathsf{id}_A)$. 
\end{enumerate}
\end{lemma}
\begin{proof} For (\ref{lemma:zigzag-zigzag-decomp.1}), for the $\Rightarrow$ direction, suppose that $f$ is a zig. Then $(\mathsf{id}_B,f)$ is an even length list (with $n=1$) where the first entry is a zig by assumption and the last entry is a zag by Lem \ref{lemma:zigzagid}.(\ref{lemma:zigzagid.1}), with the requirement of true zigs and true zags being vacuously true, and trivially we have that $\mathsf{id}_B \circ f = f$. Therefore, by uniqueness of zigzag decompositions, we have that $\mathcal{Z}[f]=(f,\mathsf{id}_B)$. Conversely, for the $\Leftarrow$ direction, suppose that $\mathcal{Z}[f]=(\mathsf{id}_B,f)$. Then by definition of a zigzag decomposition, we have that $f$ is a zig. One can show (\ref{lemma:zigzag-zigzag-decomp.2}) using similar arguments. 
\end{proof}

From this it follows that only identity maps are both zigs and zags. 

\begin{lemma}\label{lemma:bothzigzag} In a zigzag dagger category $(\mathbb{D}, \dagger, \mathsf{Zig})$,
\begin{enumerate}[{\em (i)}]
\item \label{lemma:bothzigzag.1} A map $f: A \to B$ is both a zig and zag if and only if $A=B$ and $f=\mathsf{id}_A$; 
\item \label{lemma:bothzigzag.2} A map $f: A \to B$ is a true zig (resp. true zag) if and only if $f$ is a zig (resp. zag) such that $f \neq \mathsf{id}_A$. 
\end{enumerate}
\end{lemma}
\begin{proof} For (\ref{lemma:bothzigzag.1}), the $\Leftarrow$ direction is Lem \ref{lemma:zigzagid}.(\ref{lemma:zigzagid.1}). For the $\Rightarrow$ direction, suppose that $f$ is both a zig and zag. By Lem \ref{lemma:zigzag-zigzag-decomp}.(\ref{lemma:zigzag-zigzag-decomp.1}), since $f$ is a zig, we have that $\mathcal{Z}[f]=(\mathsf{id}_B,f)$, but by Lem \ref{lemma:zigzag-zigzag-decomp}.(\ref{lemma:zigzag-zigzag-decomp.2}), since $f$ is also a zag, we have that $\mathcal{Z}[f]=(f,\mathsf{id}_A)$. So the two lists must be the same, $(\mathsf{id}_B,f)=(f,\mathsf{id}_A)$, which is only the case if $A=B$ and $f=\mathsf{id}_A$. Then for (\ref{lemma:bothzigzag.2}), by definition, $f$ is a true zig (resp. true zag) if and only if $f$ is a zig (resp. zag) which is not a zag (resp. zig), which by (\ref{lemma:bothzigzag.1}), is the case if and only if $f$ is a zig (resp. zag) such that $f \neq \mathsf{id}_A$. 
\end{proof}

A natural question to ask is what are the zigzag decompositions of composite of maps and of adjoints of maps. Since we will not require formulas for these explicitly in what follows, we just give them informally and leave it to the reader to check the technical details. In particular, as we will see below, giving a composition formula for a composition is a tricky and a bit of a combinatorial mess. So let $f: A \to B$ be a map in a zigzag dagger category $(\mathbb{D}, \dagger, \mathsf{Zig})$. It is straightforward to see that the zigzag decomposition of its adjoints $f^\dagger: B \to A$ will be the reverse list of the adjoints of $\mathcal{Z}[f]$: 
\[ \mathcal{Z}[f^\dagger] = ({}_n\mathcal{Z}[f]^\dagger, \mathcal{Z}[f]_n^\dagger, \hdots, \mathcal{Z}[f]_i^\dagger, {}_{i}\mathcal{Z}[f]^\dagger, \hdots, {}_1\mathcal{Z}[f]^\dagger, \mathcal{Z}[f]_1^\dagger) \]
Now let $g: B \to C$ be another map. So what is the zigzag decomposition of the composite $g \circ f$? Well unfortunately it will not in general be the concatenation of $\mathcal{Z}[f]$ followed by $\mathcal{Z}[g]$. Indeed while the concatenation will be a list of alternating zigs then zags, the issue is that ${}_n\mathcal{Z}[f]$ may not be a true zag or $\mathcal{Z}[g]_1$ may not be a true zig. If ${}_n\mathcal{Z}[f]$ is a true zag and $\mathcal{Z}[g]_1$ is a true zig, then yes, the zigzag decomposition of $g \circ f$ is the simply the concatenation of their zigzag decompositions: 
\[ \mathcal{Z}[g \circ f]=({}_m\mathcal{Z}[g], \mathcal{Z}[g]_m, \hdots, {}_1\mathcal{Z}[g],\mathcal{Z}[g]_1, {}_n\mathcal{Z}[f], \mathcal{Z}[f]_n, \hdots, {}_1\mathcal{Z}[f],\mathcal{Z}[f]_1) \]
If ${}_n\mathcal{Z}[f]$ is not a true zag and $\mathcal{Z}[g]_1$ is not a true zig, hence by Lemma \ref{lemma:bothzigzag}.(\ref{lemma:bothzigzag.2}) both of them are identities, then we remove them are concatenate the remaining lists: 
\[ \mathcal{Z}[g \circ f]=({}_m\mathcal{Z}[g], \mathcal{Z}[g]_m, \hdots, {}_1\mathcal{Z}[g], \mathcal{Z}[f]_n, \hdots, {}_1\mathcal{Z}[f],\mathcal{Z}[f]_1) \]
Now suppose that ${}_n\mathcal{Z}[f]$ is a true zag but $\mathcal{Z}[g]_1$ is not a true zig. Then we remove $\mathcal{Z}[g]_1$ from our list, but then we are left two consecutive zags ${}_1\mathcal{Z}[g]$ followed by ${}_n\mathcal{Z}[f]$. To fix this, simply compose them, ${}_1\mathcal{Z}[g] \circ {}_n\mathcal{Z}[f]$ which still gives a zag. However the composite of true zags is not necessarily a true zag. If ${}_1\mathcal{Z}[g] \circ {}_n\mathcal{Z}[f]$ is a true zag, then our zigzag decomposition in this case is: 
\[ \mathcal{Z}[g \circ f]=({}_m\mathcal{Z}[g], \mathcal{Z}[g]_m, \hdots, {}_1\mathcal{Z}[g] \circ {}_n\mathcal{Z}[f], \hdots, {}_1\mathcal{Z}[f],\mathcal{Z}[f]_1) \]
However if ${}_1\mathcal{Z}[g] \circ {}_n\mathcal{Z}[f]$ is a not true zag, then we remove. Then we get two consecutive true zigs, which we compose and then check if it a true zig. If it is, we get our zigzag decomposition, if it is not, then we remove it and start process again. Since our lists are finite, this process always terminate where we either get true zigs or true zags, or we end with a list of length $2$ where we no longer need to check for true zigs or true zags. Similarly if we had started with ${}_n\mathcal{Z}[f]$ not a true zag but $\mathcal{Z}[g]_1$ now a true zig, we would apply the same kind of procedure. While it might be technically possible to write down an explicit formula for the zigzag decomposition of a composition, it is not necessarily more enlightening. So we leave this as an exercise for the motivated reader. 

Now before we turn our attention to freeness of zigzag dagger categories, we first show that a dagger category can be zigzag with respect to the different subclasses of zig maps. In particular, every zigzag dagger category is also a zigzag dagger category with respect to its zag maps. 

\begin{lemma} Let $(\mathbb{D}, \dagger, \mathsf{Zig})$ be a zigzag dagger category and let $\mathsf{Zag} = \lbrace f \in \mathsf{Maps}(\mathbb{D}) \vert~ f^\dagger \in \mathsf{Zig} \rbrace$. Then $(\mathbb{D}, \dagger, \mathsf{Zag})$ is also a zigzag dagger category. 
\end{lemma}
\begin{proof} We first need to show that $\mathsf{Zag}$ is closed under composition and identities. For identities, by Lemma \ref{lemma:zigzagid}.(\ref{lemma:zigzagid.1}), we know that identity maps are zags, hence $\mathsf{id}_A \in \mathsf{Zag}$. On the other hand, suppose that $f: A \to B$ and $g: B \to C$ are both in $\mathsf{Zag}$. So $g^\dagger: C \to B$ and $f^\dagger: B \to A$ are both zigs. Since zigs are closed under composition, $f^\dagger \circ g^\dagger$ is a zig. Since the dagger is contravariant (\ref{eq:dagger}), we get that $(g\circ f)^\dagger = f^\dagger \circ g^\dagger$ is a zig, hence $g \circ f$ is a zag, or in other words, $g \circ f \in \mathsf{Zag}$. Now we need to check that every map $f: A \to B$ has a zigzag decomposition with respect to $\mathsf{Zag}$, which means we need a list of alternating between zags and zigs. If $\mathcal{Z}[f]_1$ is a true zig and ${}_n\mathcal{Z}[f]$ is a true zag, then by simply adding identities we get an appropriate list for $\mathsf{Zag}$: 
\[ (\mathsf{id}_B,{}_n\mathcal{Z}[f], \mathcal{Z}[f]_n, \hdots, {}_{i}\mathcal{Z}[f], \mathcal{Z}[f]_i, \hdots, {}_1\mathcal{Z}[f],\mathcal{Z}[f]_1, \mathsf{id}_A) \]
If $\mathcal{Z}[f]_1$ is not a true zig and ${}_n\mathcal{Z}[f]$ is not a true zag, then by simply removing then we also get an appropriate list: 
\[ (\mathcal{Z}[f]_n, \hdots, {}_{i}\mathcal{Z}[f], \mathcal{Z}[f]_i, \hdots, {}_1\mathcal{Z}[f]) \]
If $\mathcal{Z}[f]_1$ is a true zig but ${}_n\mathcal{Z}[f]$ is not a true zag, we remove the latter and add an identity at the end: 
\[ (\mathcal{Z}[f]_n, \hdots, {}_{i}\mathcal{Z}[f], \mathcal{Z}[f]_i, \hdots, {}_1\mathcal{Z}[f], \mathcal{Z}[f]_1, \mathsf{id}_A) \]
If $\mathcal{Z}[f]_1$ is not a true zig but ${}_n\mathcal{Z}[f]$ is a true zag, we remove the former and add an identity at the start: 
\[ (\mathsf{id}_B,{}_n\mathcal{Z}[f], \mathcal{Z}[f]_n, \hdots, {}_{i}\mathcal{Z}[f], \mathcal{Z}[f]_i, \hdots, {}_1\mathcal{Z}[f]) \]
From here, it is straightforward to check in each case that these lists are indeed alternating between zags then zigs, that the inner maps are true zags or true zigs and of the right type, and that they compose to $f$. Uniqueness of $\mathsf{Zag}$ zigzag decomposition will follow directly from uniqueness of zigzag decompositions for $\mathsf{Zig}$. So we conclude that  $(\mathbb{D}, \dagger, \mathsf{Zag})$ is indeed a zigzag dagger category. 
\end{proof}

We now turn our attention to proving that zigzag dagger categories are free dagger categories. Unsurprisingly, the base category will be the subcategory of zig maps. So let $(\mathbb{D}, \dagger, \mathsf{Zig})$ be a zigzag dagger category. Define $\mathbb{D}_{\mathsf{Zig}}$ to be the wide subcategory of zig maps of $\mathbb{D}$, which is a well-defined subcategory since zig maps are closed under composition and identity. Let $\mathcal{N}_{\mathsf{Zig}}: \mathbb{D}_{\mathsf{Zig}} \to \mathbb{D}$ be the inclusion functor. 

\begin{proposition}\label{prop:zigzag-to-free} Let $(\mathbb{D}, \dagger, \mathsf{Zig})$ be a zigzag dagger category. Then $(\mathbb{D}, \dagger, \mathcal{N}_{\mathsf{Zig}})$ is a free dagger category over $\mathbb{D}_{\mathsf{Zig}}$. 
\end{proposition}
\begin{proof} Rather than show the universal property directly (as we did in the cofree scenario) we will instead show that $(\mathbb{D}, \dagger)$ and $(\mathsf{Z}\left[ \mathbb{D}_{\mathsf{Zig}} \right], \dagger)$ are dagger isomorphic. So consider the induced dagger functor $\mathcal{N}_{\mathsf{Zig}}^\sharp: (\mathsf{Z}\left[ \mathbb{D}_{\mathsf{Zig}} \right], \dagger) \to (\mathbb{D}, \dagger)$ as given in Prop \ref{prop:canon-free}. So in this case, on objects we have that $\mathcal{N}_{\mathsf{Zig}}^\sharp(A) = A$, while on maps we have that: 
\begin{align}\label{eq:Nsharp}
\mathcal{N}_{\mathsf{Zig}}^\sharp\left([{}_n f, f_n, \hdots, {}_i f, f_i, \hdots, {}_1f, f_1] \right) = {}_nf^\dagger  \circ f_n \circ \hdots \circ {}_i f^\dagger  \circ f_i \circ \hdots \circ {}_1f^\dagger  \circ f_1
\end{align}
Instead of providing an explicit inverse, since $\mathcal{N}_{\mathsf{Zig}}^\sharp$ is the identity on objects, to show that it is also an isomorphism, it suffices to show that it is full and faithful as well. 

So let $f: A \to B$ be a map in $\mathbb{D}$ and consider its zigzag decomposition. Now by definition, each $\mathcal{Z}[f]_i$ is a zig, while each ${}_i\mathcal{Z}[f]$ is a zag, hence ${}_i\mathcal{Z}[f]^\dagger$ is a zig. Therefore we get that $[{}_n\mathcal{Z}[f]^\dagger, \mathcal{Z}[f]_n, \hdots, {}_1\mathcal{Z}[f]^\dagger,\mathcal{Z}[f]_1]: A \to B$ is a well defined map in $\mathsf{Z}\left[ \mathbb{D}_{\mathsf{Zig}} \right]$. Then we compute: 
\begin{gather*}
\mathcal{N}_{\mathsf{Zig}}^\sharp\left([{}_n\mathcal{Z}[f]^\dagger, \mathcal{Z}[f]_n, \hdots, {}_1\mathcal{Z}[f]^\dagger,\mathcal{Z}[f]_1] \right) \overset{\text{(\ref{eq:Nsharp})}}{=} {}_n\mathcal{Z}[f]^{\dagger\dagger} \circ \mathcal{Z}[f]_n \circ \hdots \circ {}_1\mathcal{Z}[f]^{\dagger\dagger} \circ \mathcal{Z}[f]_1 \\
\overset{\text{(\ref{eq:dagger})}}{=} {}_n\mathcal{Z}[f] \circ \mathcal{Z}[f]_n \circ \hdots \circ {}_1\mathcal{Z}[f] \circ \mathcal{Z}[f]_1 \overset{\text{(\ref{eq:zigzag-decomp})}}{=}f
\end{gather*}
Thus $\mathcal{N}_{\mathsf{Zig}}^\sharp\left([{}_n\mathcal{Z}[f]^\dagger, \mathcal{Z}[f]_n, \hdots, {}_1\mathcal{Z}[f]^\dagger,\mathcal{Z}[f]_1] \right) = f$, and so we get that $\mathcal{N}_{\mathsf{Zig}}^\sharp$ is full. 

On the other hand, let $[{}_n f, f_n, \hdots, {}_i f, f_i, \hdots, {}_1f, f_1]: A \to B$ and $[{}_m g, g_m, \hdots, {}_j g, g_j, \hdots, {}_1g, g_1]: A \to B$ be parallel maps in $\mathsf{Z}\left[ \mathbb{D}_{\mathsf{Zig}} \right]$. Without loss of generality, we can assume they are both in normal form. Also suppose that $\mathcal{N}_{\mathsf{Zig}}^\sharp\left([{}_n f, f_n, \hdots, {}_1f, f_1] \right) = \mathcal{N}_{\mathsf{Zig}}^\sharp\left([{}_m g, g_m, \hdots, {}_1g, g_1] \right)$, that is, the following equality holds: 
\begin{align}\label{eq:Nsharp-inj}
{}_nf^\dagger  \circ f_n \circ \hdots \circ {}_1f^\dagger  \circ f_1 = {}_mg^\dagger  \circ g_m \circ \hdots \circ {}_1g^\dagger  \circ g_1
\end{align}
Let us denote the above composition as the map $h: A \to B$ in $\mathbb{D}$. Now since $[{}_n f, f_n, \hdots, {}_1f, f_1]$ is a map in $\mathsf{Z}\left[ \mathbb{D}_{\mathsf{Zig}} \right]$, each $f_i$ is a zig in $\mathbb{D}$ and each ${}_if$ is also a zig in $\mathbb{D}$, which means each ${}_if^\dagger$ is a zag in $\mathbb{D}$. Moreover, since $[{}_n f, f_n, \hdots, {}_1f, f_1]$ is in normal form, none of the inner maps are identities. Thus for all $1 \leq i < n$ and $1 < j \leq n$, ${}_i f$ are true zags and $f_j$ are true zigs. Thus we get that $({}_n f^\dagger, f_n, \hdots, {}_1f^\dagger, f_1)$ is the zigzag decomposition of $h$. However by a similar argument, we also get that $({}_m g^\dagger, g_m, \hdots, {}_1g^\dagger, g_1)$ is a zigzag decomposition of $h$. However by uniqueness of zigzag decompositions, we get that $({}_n f^\dagger, f_n, \hdots, {}_1f^\dagger, f_1)=({}_m g^\dagger, g_m, \hdots, {}_1g^\dagger, g_1)$. This means that $m=n$, for all $1 \leq i \leq n$, $f_i = g_i$ and ${}_if^\dagger={}_ig^\dagger$, which by applying dagger to the second equality gives us ${}_if={}_ig$. So we conclude that $[{}_n f, f_n, \hdots, {}_1f, f_1]= [{}_m g, g_m, \hdots, {}_1g, g_1]$, and thus $\mathcal{N}_{\mathsf{Zig}}^\sharp$ is faithful. 

Therefore, we can conclude that $\mathcal{N}_{\mathsf{Zig}}^\sharp: (\mathsf{Z}\left[ \mathbb{D}_{\mathsf{Zig}} \right], \dagger) \to (\mathbb{D}, \dagger)$ is a dagger isomorphism. Then by applying Lemma \ref{lemma:free-iso}, we get that $(\mathbb{D}, \dagger,  \mathcal{N}_{\mathsf{Zig}}^\sharp \circ \mathcal{N}_{\mathbb{D}_{\mathsf{Zig}}})$ is also a free dagger category over $\mathbb{D}_{\mathsf{Zig}}$. However by definition, we have that $\mathcal{N}_{\mathsf{Zig}}^\sharp \circ \mathcal{N}_{\mathbb{D}_{\mathsf{Zig}}}= \mathcal{N}_{\mathsf{Zig}}$. Thus we conclude that $(\mathbb{D}, \dagger, \mathcal{N}_{\mathsf{Zig}})$ is a free dagger category over $\mathbb{D}_{\mathsf{Zig}}$. 
\end{proof}

In the above proof we did not work out explicitly the universal property of a zigzag dagger category as a free dagger over its subcategory of zig maps. We do so now, where the induced unique dagger functor is constructed using the zigzag decomposition. It is worth noting that had we proved the universal property of a free dagger category directly, this is where we would have needed an explicit formula for the zigzag decomposition of composition, to show functoriality of our constructed dagger functor. 

\begin{corollary}\label{cor:free-explicit} Let $(\mathbb{D}, \dagger, \mathsf{Zig})$ be a zigzag dagger category. Then for any dagger category $(\mathbb{C}, \dagger)$ and functor $\mathcal{F}: \mathbb{D}_{\mathsf{Zig}} \to \mathbb{C}$, the unique dagger $\mathcal{F}^\sharp: (\mathbb{D}, \dagger) \to (\mathbb{C}, \dagger)$ such that the following diagram commutes: 
\begin{equation}\begin{gathered}\label{diag:free-zigzag}  \xymatrixcolsep{5pc}\xymatrix{\mathbb{D}  \ar@{-->}[r]^-{\exists! ~ \mathcal{F}^\sharp}  & \mathbb{C}  \\
   \mathbb{D}_{\mathsf{Zig}} \ar[u]^-{\mathcal{N}_\mathsf{Zig}} \ar[ur]_-{\mathcal{F}} }
\end{gathered}\end{equation}
is defined on objects as $\mathcal{F}^\sharp(A) = \mathcal{F}(A)$ and on maps as follows: 
\begin{align}
\mathcal{F}^\sharp\left(f \right) = \mathcal{F}\left({}_n\mathcal{Z}[f] \right)^\dagger  \circ \mathcal{F} \left(\mathcal{Z}[f]_n \right) \circ \hdots \circ \mathcal{F}\left({}_i\mathcal{Z}[f] \right)^\dagger  \circ \mathcal{F} \left(\mathcal{Z}[f]_i \right) \circ \hdots \circ \mathcal{F}\left({}_1\mathcal{Z}[f] \right)^\dagger  \circ \mathcal{F} \left(\mathcal{Z}[f]_1 \right)
\end{align}
\end{corollary}
\begin{proof} By Lemma \ref{lemma:free-iso}, the unique dagger functor $\mathcal{F}^\sharp: (\mathbb{D}, \dagger) \to (\mathbb{C}, \dagger)$ is given by the composition of the unique dagger functor $(\mathsf{Z}\left[ \mathbb{D}_{\mathsf{Zig}} \right], \dagger) \to (\mathbb{C}, \dagger)$ as defined in Prop \ref{prop:canon-free}, precomposed with ${\mathcal{N}_{\mathsf{Zig}}^\sharp}^{-1}: (\mathbb{D}, \dagger) \to (\mathsf{Z}\left[ \mathbb{D}_{\mathsf{Zig}} \right], \dagger)$. By the arguments in the proofs of Prop \ref{prop:zigzag-to-free}, it follows that ${\mathcal{N}_{\mathsf{Zig}}^\sharp}^{-1}$ is defined on objects as ${\mathcal{N}_{\mathsf{Zig}}^\sharp}^{-1}(A) = A$ and on maps as follows: 
\begin{align}
{\mathcal{N}_{\mathsf{Zig}}^\sharp}^{-1}(f) = \left[ {}_n\mathcal{Z}[f]^\dagger, \mathcal{Z}[f]_n, \hdots, {}_{i}\mathcal{Z}[f]^\dagger, \mathcal{Z}[f]_i, \hdots, {}_1\mathcal{Z}[f]^\dagger,\mathcal{Z}[f]_1 \right]
\end{align}
From here, it follows that $\mathcal{F}^\sharp$ is indeed given as above. 
\end{proof}

We now prove the converse, that a free dagger category is a zigzag category, where the zig maps are essentially the maps of the chosen base category. 

\begin{proposition}\label{prop:free-to-zigzag} Let $(\mathbb{D}, \dagger, \mathcal{N})$ be a free dagger category over a category $\mathbb{B}$. Then $(\mathbb{D}, \dagger, \mathsf{Zig}_\mathcal{N})$ is a zigzag dagger category where $\mathsf{Zig}_\mathcal{N} = \lbrace \mathcal{N}(f) \vert~ \forall f \in \mathsf{Maps}(\mathbb{B}) \rbrace$. 
\end{proposition}
\begin{proof} Since $\mathcal{N}$ is a functor, it follows that $\mathsf{Zig}_\mathcal{N}$ is closed under composition and identity maps. So it remains to show that maps have a zigzag decomposition with respect to $\mathsf{Zig}_\mathcal{N}$. To produce this zigzag decomposition, consider the dagger isomorphism $\mathcal{N}_\mathbb{B}^\sharp: (\mathbb{D}, \dagger) \to (\mathsf{Z}[\mathbb{B}], \dagger)$ and its inverse $\mathcal{N}^\sharp: (\mathsf{Z}[\mathbb{B}], \dagger) \to (\mathbb{D}, \dagger)$ (Cor \ref{cor:free-iso.1}). Now recall by Lemma \ref{lemma:free-objects} that $\mathcal{N}$ gives us an isomorphism $\mathsf{Obj}(\mathbb{B})\cong \mathsf{Obj}(\mathbb{D})$. So every map $f$ in $\mathbb{D}$ is of type $f: \mathcal{N}(A) \to \mathcal{N}(B)$ for some objects $A,B \in \mathsf{Obj}(\mathbb{B})$. Applying $\mathcal{N}_\mathbb{B}^\sharp(f)$ then gives us a map of type $\mathcal{N}_\mathbb{B}^\sharp(f): A \to B$ in $\mathsf{Z}[\mathbb{B}]$. Let us write the normal form of $\mathcal{N}_\mathbb{B}^\sharp(f)$ as follows: 
\begin{align}\label{def:NBsharp-normal}
\mathcal{N}_\mathbb{B}^\sharp(f) = \left[ {}_nf, f_n, \hdots, {}_{i}f, f_i, \hdots, {}_1f, f_1 \right] 
\end{align}
Then consider the following list: 
\begin{align}\label{def:ZNF}
 \mathcal{Z}_{\mathcal{N}}[f] = \left( \mathcal{N}({}_nf)^\dagger, \mathcal{N}(f_n), \hdots, \mathcal{N}({}_{i}f)^\dagger, \mathcal{N}(f_i), \hdots, \mathcal{N}({}_1f)^\dagger, \mathcal{N}(f_1) \right)
\end{align}
By definition, $\mathcal{Z}[f]$ is clearly an alternating list of zigs followed by zags of the right types. Moreover since $\left[ {}_nf, f_n, \hdots, {}_1f, f_1 \right]$ was in normal form and that $\mathcal{N}$ is faithful, it follows that all the inner maps of $ \mathcal{Z}[f]$ are either true zigs or true zags as well. Then we compute: 
\begin{gather*}
\mathcal{N}({}_nf)^\dagger \circ \mathcal{N}(f_n) \circ \hdots \circ \mathcal{N}({}_1f)^\dagger \circ \mathcal{N}(f_1) \overset{\overset{\text{Prop}}{\text{\ref{prop:canon-free}}}}{=} \mathcal{N}^\sharp\left(\left[ {}_nf, f_n, \hdots {}_1f, f_1 \right]  \right) \overset{\text{(\ref{def:NBsharp-normal})}}{=} \mathcal{N}^\sharp\left(\mathcal{N}_\mathbb{B}^\sharp(f) \right) \overset{\overset{\text{Cor}}{\text{\ref{cor:cofree-iso.1}}}}{=} f
\end{gather*}
Thus the composition of $\mathcal{Z}_{\mathcal{N}}[f]$ is $f$. Lastly, it remains to show uniqueness. So suppose we had another zigzag decomposition of $f$, which would have to be of the form:
\[ \left( \mathcal{N}({}_mg)^\dagger, \mathcal{N}(g_m), \hdots, \mathcal{N}({}_1g)^\dagger, \mathcal{N}(g_1) \right) \]
Note that since the inner maps are either true zigs or true zags, hence by Lemma \ref{lemma:bothzigzag}.(\ref{lemma:bothzigzag.2}) not identities, and since $\mathcal{N}$ is faithful, it follows that the map $[{}_mg, g_m, \hdots, {}_1g, g_1]$ is a well-defined map in $\mathsf{Z}[\mathbb{B}]$ and also in normal form. Now we compute: 
\begin{gather*}
[{}_mg, g_m, \hdots, {}_1g, g_1] \overset{\overset{\text{Cor}}{\text{\ref{cor:cofree-iso.1}}}}{=} \mathcal{N}_\mathbb{B}^\sharp \left( \mathcal{N}^\sharp\left(\left[ {}_mg, g_m, \hdots {}_1g, g_1 \right]  \right) \right) \overset{\overset{\text{Prop}}{\text{\ref{prop:canon-free}}}}{=} \mathcal{N}_\mathbb{B}^\sharp \left( \mathcal{N}({}_mg)^\dagger \circ \mathcal{N}(g_m) \circ \hdots \circ \mathcal{N}({}_1g)^\dagger \circ \mathcal{N}(g_1) \right) \\\overset{\text{(\ref{eq:zigzag-decomp})}}{=} \mathcal{N}_\mathbb{B}^\sharp(f) \overset{\text{(\ref{def:NBsharp-normal})}}{=} \left[ {}_nf, f_n, \hdots, {}_1f, f_1 \right] 
\end{gather*}
So $[{}_mg, g_m, \hdots, {}_1g, g_1] = \left[ {}_nf, f_n, \hdots, {}_1f, f_1 \right]$. However since both are in normal form, by Prop \ref{prop:normal}, we must have that $n=m$ and $g_i=f_i$ and ${}_ig={}_if$. From this, it follows that $\mathcal{Z}_{\mathcal{N}}[f]$ is the unique zigzag decomposition of $f$ as desired. So we conclude that $(\mathbb{D}, \dagger, \mathsf{Zig}_\mathcal{N})$ is indeed a zigzag dagger category. 
\end{proof}

\begin{corollary} For a category $\mathbb{B}$, $(\mathsf{Z}[\mathbb{B}], \dagger, \mathsf{Zig}_\mathbb{B})$ is a zigzag dagger category where $\mathsf{Zig}_\mathbb{B} = \lbrace [\mathsf{id},f] \vert~ \forall f \in \mathsf{Maps}(\mathbb{B}) \rbrace$.
\end{corollary}
\begin{proof} By Prop \ref{prop:free-to-zigzag}, we get that $(\mathsf{Z}[\mathbb{B}], \dagger, \mathsf{Zig}_{\mathcal{N}_\mathbb{B}})$ is a zigzag dagger category. However since by definition $\mathcal{N}_\mathbb{B}(f) = [\mathsf{id},f]$, it follows that $\mathsf{Zig}_{\mathcal{N}_\mathbb{B}}=\mathsf{Zig}_\mathbb{B}$. So $(\mathsf{Z}[\mathbb{B}], \dagger, \mathsf{Zig}_\mathbb{B})$ is a zigzag dagger category as desired. 
\end{proof}

Thus, Prop \ref{prop:zigzag-to-free} and Prop \ref{prop:free-to-zigzag} together give us our second main result:  

\begin{theorem}\label{thm:free=zigzag} A dagger category is free if and only if it has a zigzag structure. 
\end{theorem}

Therefore, zigzag structure is an internal characterization of freeness of a dagger category, where we've replaced our outer base category with an internal subcategory. Once again, a natural question to ask is if the constructions of Prop \ref{prop:zigzag-to-free} and Prop \ref{prop:free-to-zigzag} are inverses of each other. Unlike the cofree situation, in this case they are (though in one direction is it only up to isomorphism). 

\begin{proposition} Let $(\mathbb{D}, \dagger, \mathsf{Zig})$ be a zigzag dagger category. Then $\mathsf{Zig}_{\mathcal{N}_{\mathsf{Zig}}}= \mathsf{Zig}$.
\end{proposition}
\begin{proof} By definition we have that $\mathsf{Maps}(\mathbb{D}_\mathsf{Zig}) = \mathsf{Zig}$ and $\mathcal{N}_{\mathsf{Zig}}(f)=f$. Thus we trivially get that $\mathsf{Zig}_{\mathcal{N}_{\mathsf{Zig}}}= \mathsf{Zig}$. 
\end{proof}

\begin{proposition} Let $(\mathbb{D},\dagger, \mathcal{N})$ be a free dagger category over a category $\mathbb{B}$. Then $\mathbb{D}_{\mathsf{Zig}_\mathcal{N}} = \mathsf{Im}[\mathcal{N}]$ and hence $\mathbb{D}_{\mathsf{Zig}_\mathcal{N}} \cong \mathbb{B}$. 
\end{proposition}
\begin{proof} It is clear that $\mathbb{D}_{\mathsf{Zig}_\mathcal{N}} = \mathsf{Im}[\mathcal{N}]$. Since $\mathcal{N}$ is faithful (Lemma \ref{lemma:N-faithful}), we essentially get back our starting external base category $\mathbb{D}_{\mathsf{Zig}_\mathcal{N}} \cong \mathbb{B}$ as well. 
\end{proof}

\subsection{Free Dagger Categories as Coalgebras}\label{sec:coalg}

In this section we will provide another characterization of free dagger categories, this time as coalgebras of a comonad. As before, for a review on adjunctions, comonads, and their coalgebras, we invite the reader to see \cite{mac1971categories}. 

Now the canonical free dagger category construction from Sec \ref{sec:free-canonical} gives a left adjoint $\mathsf{Z}: \mathsf{CAT} \to \mathsf{U}$ to the forgetful functor $\mathsf{U}: \mathsf{DAG} \to \mathsf{CAT}$ \cite[Thm 3.1.17]{heunen2009categorical}. This induces a comonad $(\mathsf{Z}, \mathcal{V}, \mathcal{E})$ on $\mathsf{DAG}$. The functor $\mathsf{Z}: \mathsf{DAG} \to \mathsf{DAG}$ maps a dagger category $(\mathbb{D},\dagger)$ to the canonical free dagger category of its underlying category, $\mathsf{Z}(\mathbb{D},\dagger)= (\mathsf{Z}[\mathbb{D}], \dagger)$, and sends a dagger functor $\mathcal{F}: (\mathbb{D}_1, \dagger) \to (\mathbb{D}_2, \dagger)$ to the dagger functor $\mathsf{Z}(\mathcal{F}): (\mathsf{Z}[\mathbb{D}_1], \dagger) \to (\mathsf{Z}[\mathbb{D}_2], \dagger)$ defined as the unique dagger functor which makes the following diagram commute: 
\begin{equation}\begin{gathered}\label{diag:ZF}  \xymatrixcolsep{5pc}\xymatrix{\mathsf{Z}[\mathbb{D}_1]    \ar@{-->}[r]^-{\exists! ~ \mathsf{Z}[\mathcal{F}]}  & \mathsf{Z}[\mathbb{D}_2]   \\
 \mathbb{D}_1 \ar[u]^-{\mathcal{N}_{\mathbb{D}_1}} \ar[r]_-{\mathcal{F}} & \mathbb{D}_2 \ar[u]_-{\mathcal{N}_{\mathbb{D}_2}} }
\end{gathered}\end{equation}
Explicitly, $\mathsf{Z}[\mathcal{F}]$ is defined as on objects as $\mathsf{Z}[\mathcal{F}](A) = \mathcal{F}(A)$ and on maps as: 
\begin{align}\label{def:ZF}
\mathsf{Z}[\mathcal{F}]\left([{}_n f, f_n, \hdots, {}_i f, f_i, \hdots, {}_1f, f_1] \right) = \left[ \mathcal{F}({}_nf), \mathcal{F} (f_n), \hdots, \mathcal{F}({}_i f), \mathcal{F} (f_i), \hdots, \mathcal{F}({}_1f), \mathcal{F} (f_1)\right]
\end{align}
Now for a dagger category $(\mathbb{D}, \dagger)$, the comonad comultiplication $\mathcal{V}_{(\mathbb{D},\dagger)}: (\mathsf{Z}[\mathbb{D}], \dagger) \to (\mathsf{Z}\left[\mathsf{Z}[\mathbb{D}]\right], \dagger)$ and the comonad counit  $\mathcal{E}_{(\mathbb{D},\dagger)}: (\mathsf{Z}[\mathbb{D}], \dagger) \to (\mathbb{D},\dagger)$ are of course the unique dagger functors which make the following diagrams commute respectively: 
\begin{equation}\begin{gathered}\label{diag:VE}  \xymatrixcolsep{5pc}\xymatrix{ \mathsf{Z}[\mathbb{D}]   \ar@{-->}[r]^-{\exists! ~ \mathcal{N}_{(\mathbb{D},\dagger)}}  &  \mathbb{D} &  \mathsf{Z}[\mathbb{D}]   \ar@{-->}[r]^-{\exists! ~ \mathcal{V}_{(\mathbb{D},\dagger)}}  & \mathsf{Z}\left[\mathsf{Z}[\mathbb{D}] \right]    \\
 \mathbb{D} \ar[u]^-{\mathcal{E}_{\mathbb{D}}} \ar@{=}[ur]  & & \mathbb{D}  \ar[r]_-{\mathcal{N}_\mathbb{D}} \ar[u]^-{\mathcal{N}_{\mathbb{D}}} &  \mathsf{Z}[\mathbb{D}] \ar[u]_-{\mathcal{N}_{\mathsf{C}[\mathbb{D}]}} }
\end{gathered}\end{equation}
Explicitly, $\mathcal{V}_{(\mathbb{D},\dagger)}$ is defined on objects as $\mathcal{V}_{(\mathbb{D},\dagger)}(A) = A$ and on maps as follows: 
\begin{align}
\mathcal{V}_{(\mathbb{D},\dagger)}\left([{}_n f, f_n, \hdots, {}_i f, f_i, \hdots, {}_1f, f_1] \right) = \left[ [\mathsf{id},{}_nf], [\mathsf{id},f_n], \hdots, [\mathsf{id},{}_i f], [\mathsf{id},f_i], \hdots, [\mathsf{id},{}_1f], [\mathsf{id},f_1]\right]
\end{align}
while $\mathcal{E}_{(\mathbb{D},\dagger)}$ is defined on objects as $\mathcal{E}_{(\mathbb{D},\dagger)}(A) = A $ and on maps as follows: 
\begin{align}
\mathcal{E}_{(\mathbb{D},\dagger)}\left([{}_n f, f_n, \hdots, {}_i f, f_i, \hdots, {}_1f, f_1] \right) = {}_n f^\dagger \circ f_n \circ \hdots \circ {}_i f^\dagger \circ f_i \circ \hdots \circ {}_1f^\dagger \circ f_1 
\end{align}
Since we have a comonad, we can consider its coalgebras.  

\begin{definition} A \textbf{$\mathsf{Z}$-coalgebra} is a triple $(\mathbb{D},\dagger, \mathcal{W})$ consisting of a dagger category $(\mathbb{D},\dagger)$ and a dagger functor ${\mathcal{W}: (\mathbb{D},\dagger) \to (\mathsf{Z}[\mathbb{D}], \dagger)}$ such that the following diagrams commute:
\begin{equation}\begin{gathered}\label{diag:Z-coalg}  \xymatrixcolsep{5pc}\xymatrix{ (\mathbb{D},\dagger) \ar@{=}[dr]^-{} \ar[r]^-{\mathcal{W}} & (\mathsf{Z}[\mathbb{D}], \dagger) \ar[d]^-{\mathcal{E}_{(\mathbb{D},\dagger)}} &  ({\mathsf{Z}[\mathbb{D}]}, \dagger)  \ar[r]^-{\mathcal{W}}  \ar[d]_-{\mathcal{W}} &   (\mathsf{Z}[\mathbb{D}], \dagger) \ar[d]^-{\mathsf{Z}[\mathcal{W}]}    \\
 & (\mathbb{D},\dagger) & (\mathsf{Z}[\mathbb{D}], \dagger)  \ar[r]_-{\mathcal{V}_{(\mathbb{D},\dagger)}} & (\mathsf{Z}\left[ \mathsf{Z}[\mathbb{D}] \right], \dagger)
 }
\end{gathered}\end{equation}
We call $\mathcal{W}$ a \textbf{$\mathsf{Z}$-coalgebra structure} of $(\mathbb{D},\dagger)$. 
\end{definition}

Let us unpack a bit what a $\mathsf{Z}$-coalgebra structure ${\mathcal{W}: (\mathbb{D},\dagger) \to (\mathsf{Z}[\mathbb{D}], \dagger)}$ tells us. First, observe that since $\mathcal{E}_{(\mathbb{D},\dagger)} \circ \mathcal{W} = \mathsf{id}_{\mathbb{D}}$, it follows that $\mathcal{W}$ is the identity on objects, so $\mathcal{W}(A) = A$. Therefore for any map $f: A \to B$ in $\mathbb{D}$, its image is a map of type $\mathcal{W}(f): A \to B$ in $\mathsf{Z}[\mathbb{D}]$. Let us write the normal form of $\mathcal{W}(f)$ as follows: 
\begin{align}\label{eq:W-normal}
\mathcal{W}(f) = \left[ {}_n\mathcal{W}(f), \mathcal{W}(f)_n ,\hdots ,{}_i\mathcal{W}(f),\mathcal{W}(f)_i , \hdots, {}_1 \mathcal{W}(f),\mathcal{W}(f)_1 \right]
\end{align}
Now the left diagram of (\ref{diag:Z-coalg}) tells us that: 
\begin{align}\label{eq:W-comp}
f = {}_n\mathcal{W}(f)^\dagger \circ \mathcal{W}(f)_n \circ \hdots \circ {}_i\mathcal{W}(f)^\dagger \circ \mathcal{W}(f)_i \circ \hdots \circ {}_1 \mathcal{W}(f) \circ \mathcal{W}(f)_1
\end{align}
On the other hand, the right diagram of (\ref{diag:Z-coalg}) first tells us that:
\begin{gather*}
\left[ \mathcal{W}\left( {}_n\mathcal{W}(f) \right), \mathcal{W}\left(\mathcal{W}(f)_n \right) ,\hdots , \mathcal{W}\left({}_i\mathcal{W}(f) \right), \mathcal{W}\left(\mathcal{W}(f)_i\right), \hdots, \mathcal{W}\left({}_1 \mathcal{W}(f) \right), \mathcal{W}\left(\mathcal{W}(f)_1\right) \right] \\
= \left[ [\mathsf{id},{}_n\mathcal{W}(f)], [\mathsf{id},\mathcal{W}(f)_n], \hdots, [\mathsf{id},{}_i \mathcal{W}(f)], [\mathsf{id},\mathcal{W}(f)_i], \hdots, [\mathsf{id},{}_1\mathcal{W}(f)], [\mathsf{id},\mathcal{W}(f)_1]\right]
\end{gather*}
Now since $\left( {}_n\mathcal{W}(f), \mathcal{W}(f)_n ,\hdots, {}_1 \mathcal{W}(f),\mathcal{W}(f)_1 \right)$ was a normal list, it follows that:
\[\left( [\mathsf{id},{}_n\mathcal{W}(f)], [\mathsf{id},\mathcal{W}(f)_n], \hdots [\mathsf{id},{}_1\mathcal{W}(f)], [\mathsf{id},\mathcal{W}(f)_1]\right)\] 
is also a normal list. Note however that $( \mathcal{W}\left( {}_n\mathcal{W}(f) \right), \mathcal{W}\left(\mathcal{W}(f)_n \right) ,\hdots, \mathcal{W}\left({}_1 \mathcal{W}(f) \right), \mathcal{W}\left(\mathcal{W}(f)_1\right)$ is a list of the same length. Thus by Prop \ref{prop:normal} it follows that for all $1 \leq i \leq n$ we have that: 
\begin{align}\label{eq:WW-zig}
 \mathcal{W}\left({}_i\mathcal{W}(f) \right) = [\mathsf{id},{}_i\mathcal{W}(f)] && \mathcal{W}\left(\mathcal{W}(f)_i\right) = [\mathsf{id},\mathcal{W}(f)_i]
\end{align}

By well-known facts about adjunctions, every free dagger category is a $\mathsf{Z}$-coalgebra. 

\begin{lemma}\label{lemma:free-coalgebra} Let $(\mathbb{D}, \dagger, \mathcal{N})$ be a  free dagger category over a category $\mathbb{B}$. Define the dagger functor $\mathcal{W}_\mathcal{N}: (\mathbb{D}, \dagger) \to  (\mathsf{Z}[\mathbb{D}], \dagger)$ as the unique dagger functor which makes the following diagram commute: 
\begin{align*} \xymatrixcolsep{5pc}\xymatrix{ \mathbb{D}  \ar@{-->}[r]^-{\exists! ~ \mathcal{W}_\mathcal{N}}  & \mathsf{Z}[\mathbb{D}]  \\
 \mathbb{B}  \ar[u]^-{\mathcal{N}} \ar[r]_-{\mathcal{N}} & \mathbb{D}\ar[u]_-{\mathcal{N}_\mathbb{D}} }
\end{align*}
Then $(\mathbb{D}, \dagger, \mathcal{W}_\mathcal{N})$ is a $\mathsf{Z}$-coalgebra.
\end{lemma}

\begin{corollary}\label{cor:free-canon-coalgebra} Let $\mathbb{B}$ be a category and let $\mathcal{W}_\mathbb{B}:  (\mathsf{Z}[\mathbb{B}], \dagger) \to (\mathsf{Z}\left[\mathsf{Z}[\mathbb{B}]\right], \dagger)$ be the dagger functor defined as $\mathcal{W}_\mathbb{B} := \mathcal{W}_{\mathcal{N}_\mathbb{B}}$. Explicitly, on objects $\mathcal{W}_\mathbb{B}(A)=A$ while on maps:
\begin{align}
\mathcal{W}_\mathbb{B}\left([{}_n f, f_n, \hdots, {}_i f, f_i, \hdots, {}_1f, f_1] \right) = \left[ [\mathsf{id},{}_nf], [\mathsf{id},f_n], \hdots, [\mathsf{id},{}_i f], [\mathsf{id},f_i], \hdots, [\mathsf{id},{}_1f], [\mathsf{id},f_1]\right]
\end{align}
Then $(\mathsf{Z}[\mathbb{B}], \dagger, \mathcal{W}_\mathbb{B})$ is a $\mathsf{Z}$-coalgebra. 
\end{corollary}

Therefore, every zigzag dagger category is a $\mathsf{Z}$-coalgebra, whose $\mathsf{Z}$-coalgebra structure can be given in terms of the zigzag decomposition. 

\begin{proposition}\label{prop:zigzag-zalg} Let $(\mathbb{D}, \dagger,\mathsf{Zig})$ be a zigzag dagger category. Then define the functor $\mathcal{W}_\mathsf{Zig}: \mathbb{D} \to \mathsf{Z}[\mathbb{D}]$ on objects as $\mathcal{W}_\mathsf{Zig}(A) = A$ and on maps as:
\begin{align}\label{eq:Wzig-f}
\mathcal{W}_\mathsf{Zig}(f) = \left[ {}_n\mathcal{Z}[f]^\dagger , \mathcal{Z}[f]_n, \hdots, {}_i\mathcal{Z}[f]^\dagger, \mathcal{Z}[f]_i, \hdots, {}_1\mathcal{Z}[f]^\dagger, \mathcal{Z}[f]_1 \right]
\end{align}
Then $(\mathbb{D}, \dagger, \mathcal{W}_\mathsf{Zig})$ is a $\mathsf{Z}$-coalgebra. 
\end{proposition}
\begin{proof} By Prop \ref{prop:zigzag-to-free}, we know that $(\mathbb{D}, \dagger,\mathcal{N}_\mathsf{Zig})$ is a free dagger category of $\mathbb{D}_\mathsf{Zig}$. Then applying Lemma \ref{lemma:free-coalgebra}, we get that $(\mathbb{D}, \dagger, \mathcal{W}_{\mathcal{N}_\mathsf{Zig}})$ is a $\mathsf{Z}$-coalgebra. Expanding $\mathcal{W}_{\mathcal{N}_\mathsf{Zig}}$ out explicitly using Cor \ref{cor:free-explicit}, it is straightforward to check that $\mathcal{W}_{\mathcal{N}_\mathsf{Zig}} = \mathcal{W}_\mathsf{Zig}$. Thus $(\mathbb{D}, \dagger, \mathcal{W}_\mathsf{Zig})$ is a $\mathsf{Z}$-coalgebra as desired. 
\end{proof}

We will now show that every $\mathsf{Z}$-coalgebra is a free dagger category. To do so, we will show that every $\mathsf{Z}$-coalgebra is a zigzag dagger category. 

\begin{proposition}\label{prop:zalg-zigzag} If $(\mathbb{D},\dagger, \mathcal{W})$ is a $\mathsf{Z}$-coalgebra, then $(\mathbb{D},\dagger,\mathsf{Zig}_\mathcal{W})$ is a zigzag dagger category where:
\begin{align}\label{def:ZigW}
\mathsf{Zig}_\mathcal{W} = \lbrace f \in \mathsf{Maps}(\mathbb{D}) \vert~ \mathcal{W}(f) = [\mathsf{id},f] \rbrace
\end{align}
Therefore, $(\mathbb{D},\dagger)$ is a free dagger category. 
\end{proposition}
\begin{proof} Notice that another description of $\mathsf{Zig}_\mathcal{W}$ is saying $f \in \mathsf{Zig}_\mathcal{W}$ if and only if $\mathcal{W}(f)= \mathcal{N}_\mathbb{D}(f)$. Since $\mathcal{W}$ and $\mathcal{N}_\mathbb{D}$ are both functors, it clear that $\mathsf{Zig}_\mathcal{W}$ is closed under composition and identity maps. Now consider an arbitrary map $f: A \to B$ in $\mathbb{D}$, and consider its image $\mathcal{W}(f): A \to B$ in $\mathsf{Z}[\mathbb{D}]$. Write its normal as in (\ref{eq:W-normal}). Then consider the following even length list: 
\begin{align}\label{def:ZigW-decomp}
\mathcal{Z}_\mathcal{W}[f] = \left( {}_n\mathcal{W}(f)^\dagger , \mathcal{W}(f)_n , \hdots , {}_i\mathcal{W}(f)^\dagger , \mathcal{W}(f)_i , \hdots , {}_1 \mathcal{W}(f),  \mathcal{W}(f)_1 \right) 
\end{align}
Now (\ref{eq:WW-zig}) tells us that each $\mathcal{W}(f)_i$ is indeed a zig and also that each ${}_i\mathcal{W}(f)$ is also a zig, hence each ${}_i\mathcal{W}(f)^\dagger$ are zags. Next (\ref{eq:W-comp}) tells us that the composition of the list $\mathcal{Z}_\mathcal{W}[f]$ is equal to $f$ as desired. For uniqueness, suppose there was another zigzag decomposition $({}_m f, f_m, \hdots, {}_j f, f_j, \hdots, {}_1f, f_1)$ of $f$. Then we compute: 
\begin{gather*}
\left[ {}_n\mathcal{W}(f), \mathcal{W}(f)_n ,\hdots, {}_1 \mathcal{W}(f),\mathcal{W}(f)_1 \right] \overset{\text{(\ref{eq:W-normal})}}{=} \mathcal{W}(f) \overset{\text{(\ref{eq:zigzag-decomp})}}{=} \mathcal{W}\left( {}_m f \circ f_m \circ \hdots \circ {}_1f \circ f_1 \right) \\
\overset{\text{(\ref{def:functor})}}{=} \mathcal{W}({}_m f) \circ \mathcal{W}(f_m) \circ \hdots \circ \mathcal{W}({}_1 f) \circ \mathcal{W}(f_1) \overset{\text{(\ref{eq:dagger})}}{=} \mathcal{W}({}_m f^{\dagger\dagger}) \circ \mathcal{W}(f_m) \circ \hdots \circ \mathcal{W}({}_1 f^{\dagger\dagger}) \circ \mathcal{W}(f_1) \\\overset{\text{(\ref{def:dagfun})}}{=} \mathcal{W}({}_m f^\dagger)^\dagger \circ \mathcal{W}(f_m) \circ \hdots \circ \mathcal{W}({}_1 f^\dagger)^\dagger \circ \mathcal{W}(f_1)  \overset{\text{(\ref{def:ZigW})}}{=} [\mathsf{id},{}_mf^\dagger]^\dagger \circ [\mathsf{id},f_m] \circ \hdots \circ [\mathsf{id},{}_1f^\dagger]^\dagger \circ [\mathsf{id},f_1] \\ \overset{\text{(\ref{def:dag-free})}}{=} [{}_mf^\dagger,\mathsf{id}] \circ [\mathsf{id},f_m] \circ \hdots \circ [{}_1f^\dagger, \mathsf{id}]\circ [\mathsf{id},f_1] \overset{\overset{\text{(\ref{def:comp-free})}}{\text{(\ref{reduce-1})+(\ref{reduce-2})}}}{=} \left[{}_mf^\dagger, f_m, \hdots, {}_1f^\dagger, f_1 \right]
\end{gather*}
So we have that $\left[ {}_n\mathcal{W}(f), \mathcal{W}(f)_n ,\hdots, {}_1 \mathcal{W}(f),\mathcal{W}(f)_1 \right]= \left[{}_mf^\dagger, f_m, \hdots, {}_1f^\dagger, f_1 \right]$. However, note that since $({}_m f, f_m, \hdots, {}_1f, f_1)$ is the zigzag decomposition of $f$ in $\mathbb{D}$, it follows that $({}_m f^\dagger, f_m, \hdots, {}_1f^\dagger, f_1)$ is a normal list. Hence by Prop \ref{prop:normal}, we then get that $n=m$ and that $f_i = \mathcal{W}(f)_i$ and ${}_if^\dagger= {}_i\mathcal{W}(f)$ for all $1 \leq i \leq n$, where by (\ref{eq:dagger}) the latter also tells us that ${}_if= {}_i\mathcal{W}(f)^\dagger$. Thus we get that zigzag decompositions are indeed unique. So we conclude that $(\mathbb{D},\dagger,\mathsf{Zig}_\mathcal{W})$ is a zigzag dagger category as desired. 
\end{proof}

It turns out that the constructions of Prop \ref{prop:zigzag-zalg} and Prop \ref{prop:zalg-zigzag} are inverses of each other.

\begin{theorem}\label{thm:coag=zigzag} For a dagger category $(\mathbb{D}, \dagger)$ there is a bijective correspondence between zigzag structure and $\mathsf{Z}$-coalgebra structure. Explicitly,
\begin{enumerate}[{\em (i)}]
\item If $\mathsf{Zig}$ is a zigzag structure for $(\mathbb{D}, \dagger)$, then $\mathsf{Zig}_{\mathcal{W}_\mathsf{Zig}} = \mathsf{Zig}$.
\item If $\mathcal{W}$ is a $\mathsf{Z}$-coalgebra structure for $(\mathbb{D}, \dagger)$, then $\mathcal{W}_{\mathsf{Zig}_\mathcal{W}} = \mathcal{W}$. 
\end{enumerate}
\end{theorem}
\begin{proof} Starting with a zigzag structure $\mathsf{Zig}$. Suppose that $f \in \mathsf{Zig}_{\mathcal{W}_\mathsf{Zig}}$, so $\mathcal{W}_\mathsf{Zig}(f) = [\mathsf{id},f]$. However by definition this means $\left[ {}_n\mathcal{Z}[f]^\dagger, \mathcal{Z}[f]_{n}, \hdots, {}_1\mathcal{Z}[f]^\dagger, \mathcal{Z}[f]_1 \right] = [\mathsf{id},f]$. But since $({}_n\mathcal{Z}[f], \mathcal{Z}[f]_{n}, \hdots, {}_1\mathcal{Z}[f], \mathcal{Z}[f]_1)$ is the zigzag decomposition of $f$, this means that  $\left[ {}_n\mathcal{Z}[f]^\dagger, \mathcal{Z}[f]_{n}, \hdots, {}_1\mathcal{Z}[f]^\dagger, \mathcal{Z}[f]_1 \right]$ is in normal form. Thus by Prop \ref{prop:normal}, we get that $n=1$ and ${}_1\mathcal{Z}[f]=\mathsf{id}$ and $\mathcal{Z}[f]_1=f$. Since by definition of a zigzag decomposition, $\mathcal{Z}[f]_1$ is a zig, we get that $f$ is a zig, so $f \in \mathsf{Zig}$. Conversely, if $f \in \mathsf{Zig}$, by Lemma \ref{lemma:zigzag-zigzag-decomp}.(\ref{lemma:zigzag-zigzag-decomp.1}), $(\mathsf{id},f)$ is the zigzag decomposition of $f$. As such, this gives us that $\mathcal{W}_\mathsf{Zig}(f) = [\mathsf{id},f]$ and so $f \in \mathsf{Zig}_{\mathcal{W}_\mathsf{Zig}}$. Thus we conclude that $\mathsf{Zig}_{\mathcal{W}_\mathsf{Zig}} = \mathsf{Zig}$ as desired. 

On the other hand, starting instead with a $\mathsf{Z}$-coalgebra structure $\mathcal{W}$. Then we compute: 
\begin{gather*}
\mathcal{W}_{\mathsf{Zig}_\mathcal{W}}(f)  \overset{\text{(\ref{eq:Wzig-f}) + (\ref{def:ZigW-decomp}}}{=} \left[ {}_n\mathcal{W}(f)^{\dagger\dagger} , \mathcal{W}(f)_{n}, \hdots, {}_i\mathcal{W}(f)^{\dagger\dagger}, \mathcal{W}(f)_{i}, \hdots, {}_1\mathcal{W}(f)^{\dagger\dagger}, \mathcal{W}(f)_1 \right] \\
\overset{\text{(\ref{eq:dagger})}}{=} \left[ {}_n\mathcal{W}(f), \mathcal{W}(f)_n ,\hdots ,{}_i\mathcal{W}(f),\mathcal{W}(f)_i , \hdots, {}_1 \mathcal{W}(f),\mathcal{W}(f)_1 \right] \overset{\text{(\ref{eq:W-normal})}}{=} \mathcal{W}(f) 
\end{gather*}
Thus from here we can conclude that $\mathcal{W}_{\mathsf{Zig}_\mathcal{W}} = \mathcal{W}$ as desired. 
\end{proof}

Applying the above constructions to a free dagger category with a specified base category recaptures the canonical $\mathsf{Z}$-coalgebra structure and zigzag structure. 

\begin{corollary}\label{cor:Wzig=} Let $(\mathbb{D}, \dagger, \mathcal{N})$ be a free dagger category over a category $\mathbb{B}$. Then $\mathcal{W}_\mathcal{N} = \mathcal{W}_{\mathsf{Zig}_\mathcal{N}}$ and $\mathsf{Zig}_\mathcal{N} = \mathsf{Zig}_{\mathcal{W}_\mathcal{N}}$.
\end{corollary}
\begin{proof} Observe that for $\mathcal{N}(f) \in \mathsf{Zig}_\mathcal{N}$, by definition we have that $\mathcal{W}_{\mathsf{Zig}_\mathcal{N}}\left( \mathcal{N}(f) \right) = \left[ \mathsf{id}, \mathcal{N}(f) \right] = \mathcal{N}_\mathbb{D}(\mathcal{N}(f))$. Thus we have that $\mathcal{W}_{\mathsf{Zig}_\mathcal{N}} \circ \mathcal{N} = \mathcal{N}_\mathbb{D} \circ \mathcal{N}$. Since $\mathcal{W}_{\mathsf{Zig}_\mathcal{N}}$ is a dagger functor, by the universal property of free dagger categories (Def \ref{def:dag-free}), it follows that $\mathcal{W}_\mathcal{N} = \mathcal{W}_{\mathsf{Zig}_\mathcal{N}}$ as desired. Then by Thm \ref{thm:coag=zigzag}, we get that $\mathsf{Zig}_{\mathcal{W}_\mathcal{N}} =  \mathsf{Zig}_{\mathcal{W}_{\mathsf{Zig}_\mathcal{N}}} = \mathsf{Zig}_\mathcal{N}$, as desired. 
\end{proof}

\begin{corollary} Let $\mathbb{B}$ be a category. Then for $(\mathsf{Z}[\mathbb{B}], \dagger)$, we have that $\mathcal{W}_\mathbb{B} = \mathcal{W}_{\mathsf{Zig}_\mathbb{B}}$.
\end{corollary}
\begin{proof} This follows from Cor \ref{cor:free-canon-coalgebra} and then applying Cor \ref{cor:Wzig=} to $(\mathsf{Z}[\mathbb{B}], \dagger)$. 
\end{proof}

Therefore, bringing all of this together, we obtain our main result about free dagger categories. 

\begin{theorem}\label{thm:full-free} For a dagger category, the following are equivalent: 
\begin{enumerate}[{\em (i)}]
\item It is free;
\item It is a zigzag dagger category; 
\item It is a $\mathsf{Z}$-coalgebra. 
\end{enumerate}
\end{theorem}

We conclude this section with the observation that Thm \ref{thm:coag=zigzag} can easily be extended to an isomorphism between the (large) category of zigzag dagger categories and the category of $\mathsf{Z}$-coalgebras, where the latter is often called the coEilenberg-Moore category. 

\begin{definition}\label{def:zigzagfun} For zigzag dagger categories $(\mathbb{D}_1, \dagger, \mathsf{Zig}_1)$ and $(\mathbb{D}_2, \dagger, \mathsf{Zig}_2)$, a \textbf{zigzag dagger functor} ${\mathcal{F}: (\mathbb{D}_1, \dagger,\mathsf{Zig}_1) \to (\mathbb{D}_2, \dagger,\mathsf{Zig}_2)}$ is a dagger functor ${\mathcal{F}: (\mathbb{D}_1, \dagger) \to (\mathbb{D}_2, \dagger)}$ which maps zig maps to zig maps, that is, if $f \in \mathsf{Zig}_1$ then $\mathcal{F}(f) \in \mathsf{Zig}_2$. Let $\mathsf{ZZDAG}$ be the (large) category of zigzag dagger categories and zigzag dagger functors between them. 
\end{definition}

\begin{definition} For $\mathsf{Z}$-coalgebras $(\mathbb{D}_1, \dagger, \mathcal{W}_1)$ and $(\mathbb{D}_2, \dagger, \mathcal{W}_2)$, a \textbf{$\mathsf{Z}$-coalgebra functor} $\mathcal{F}: (\mathbb{D}_1, \dagger,\mathcal{W}_1) \to (\mathbb{D}_2, \dagger,\mathcal{W}_2)$ is a dagger functor ${\mathcal{F}: (\mathbb{D}_1, \dagger) \to (\mathbb{D}_2, \dagger)}$ such that the following diagram commutes: 
\begin{equation}\begin{gathered}\label{diag:Z-coalg-fun}   \xymatrixcolsep{5pc}\xymatrix{  \mathbb{D}_1  \ar[r]^-{\mathcal{F}}  \ar[d]_-{\mathcal{W}_1} & \mathbb{D}_2 \ar[d]^-{\mathcal{W}_2} \\ 
\mathsf{Z}[\mathbb{D}_1] \ar[r]_-{\mathsf{Z}[\mathcal{F}]}  &  \mathsf{Z}[\mathbb{D}_2]
 }
\end{gathered}\end{equation}
Let $\mathsf{Z}\text{-}\mathsf{COALG}$ be the (large) category of $\mathsf{Z}$-coalgebras and $\mathsf{Z}$-algebra functors between them. 
\end{definition}

\begin{lemma} Let ${\mathcal{F}: (\mathbb{D}_1, \dagger,\mathsf{Zig}_1) \to (\mathbb{D}_2, \dagger,\mathsf{Zig}_2)}$ be a zigzag dagger functor. Then $\mathcal{F}: (\mathbb{D}_1, \dagger,\mathcal{W}_{\mathsf{Zig}_1}) \to (\mathbb{D}_2, \dagger,\mathcal{W}_{\mathsf{Zig}_2})$ is a $\mathsf{Z}$-coalgebra functor. 
\end{lemma}
\begin{proof} It is straightforward to see that since a zigzag dagger functor preserves zigs, it preserves zags, and hence also preserves zigzag decomposition. Explicitly, the zigzag decomposition of $\mathcal{F}(f)$ is $\mathcal{Z}[\mathcal{F}(f)] = (\mathcal{F}\left({}_n\mathcal{Z}[f] \right), \mathcal{F}\left(\mathcal{Z}[f]_n\right), \hdots, \mathcal{F}\left({}_1\mathcal{Z}[f] \right), \mathcal{F}\left( \mathcal{Z}[f]_1 \right))$. So we may write that for all $1 \leq i \leq n$, we have that: 
\begin{align}\label{eq:ZF=FZ}
\mathcal{Z}\left[\mathcal{F}(f) \right]_i = \mathcal{F}\left(\mathcal{Z}[f]_i\right) && {}_i\mathcal{Z}\left[\mathcal{F}(f) \right] = \mathcal{F}\left({}_i\mathcal{Z}[f]\right) 
\end{align}
Then we compute: 
\begin{gather*}
\mathsf{Z}[\mathcal{F}]\left( \mathcal{W}_{\mathsf{Zig}_1}(f) \right) \overset{\text{(\ref{eq:Wzig-f})}}{=} \mathsf{Z}[\mathcal{F}]\left( \left[ {}_n\mathcal{Z}[f]^\dagger , \mathcal{Z}[f]_n, \hdots, {}_1\mathcal{Z}[f]^\dagger, \mathcal{Z}[f]_1 \right] \right) \\\overset{\text{(\ref{def:ZF})}}{=} \left[ \mathcal{F}\left( {}_n\mathcal{Z}[f]^\dagger \right) , \mathcal{F}\left(\mathcal{Z}[f]_n\right), \hdots, \mathcal{F}\left({}_1\mathcal{Z}[f]^\dagger\right), \mathcal{F}\left(\mathcal{Z}[f]_1 \right) \right] \\
\overset{\text{(\ref{def:dagfun})}}{=} \left[ \mathcal{F}\left( {}_n\mathcal{Z}[f] \right)^\dagger , \mathcal{F}\left(\mathcal{Z}[f]_n\right), \hdots, \mathcal{F}\left({}_1\mathcal{Z}[f]\right)^\dagger, \mathcal{F}\left(\mathcal{Z}[f]_1 \right) \right] \\\overset{\text{(\ref{eq:ZF=FZ})}}{=}\left[ {}_n\mathcal{Z}\left[\mathcal{F}(f) \right]^\dagger , \mathcal{Z}\left[\mathcal{F}(f) \right]_n, \hdots, {}_1\mathcal{Z}\left[\mathcal{F}(f) \right]^\dagger , \mathcal{Z}\left[\mathcal{F}(f) \right]_1 \right] \overset{\text{(\ref{eq:Wzig-f})}}{=} \mathcal{W}_{\mathsf{Zig}_2}\left( \mathcal{F}(f) \right)
\end{gather*}
Thus it follows that $\mathsf{Z}[\mathcal{F}] \circ \mathcal{W}_{\mathsf{Zig}_1} =  \mathcal{W}_{\mathsf{Zig}_2} \circ \mathcal{F}$. So ${\mathcal{F}: (\mathbb{D}_1, \dagger,\mathcal{W}_{\mathsf{Zig}_1}) \to (\mathbb{D}_2, \dagger,\mathcal{W}_{\mathsf{Zig}_2})}$ is a $\mathsf{Z}$-coalgebra functor. 
\end{proof}

\begin{lemma} Let ${\mathcal{F}: (\mathbb{D}_1, \dagger,\mathcal{W}_1) \to (\mathbb{D}_2, \dagger,\mathcal{W}_2)}$ be a $\mathsf{Z}$-coalgebra functor. Then $\mathcal{F}: (\mathbb{D}_1, \dagger,\mathsf{Zig}_{\mathcal{W}_1}) \to (\mathbb{D}_2, \dagger,\mathsf{Zig}_{\mathcal{W}_2})$ is a zigzag dagger functor.
\end{lemma}
\begin{proof} Let $f \in \mathsf{Zig}_{\mathcal{W}_1}$. Then we compute: 
\begin{gather*}
\mathcal{W}_2\left( \mathcal{F}(f) \right) \overset{\text{(\ref{diag:Z-coalg-fun})}}{=} \mathsf{Z}[\mathcal{F}]\left( \mathcal{W}_{\mathsf{Zig}_1}(f) \right) \overset{\text{(\ref{def:ZigW})}}{=}  \mathsf{Z}[\mathcal{F}]\left( [\mathsf{id}, f] \right) \overset{\text{(\ref{def:ZF})}}{=} [\mathcal{F}(\mathsf{id}), \mathcal{F}(f)] \overset{\text{(\ref{def:functor})}}{=} [\mathsf{id}, \mathcal{F}(f)]
\end{gather*}
So $\mathcal{F}(f) \in \mathsf{Zig}_{\mathcal{W}_2}$. Thus ${\mathcal{F}: (\mathbb{D}_1, \dagger,\mathsf{Zig}_{\mathcal{W}_1}) \to (\mathbb{D}_2, \dagger,\mathsf{Zig}_{\mathcal{W}_2})}$ is a zigzag dagger functor.
\end{proof}

By combining the above lemmas with Thm \ref{thm:coag=zigzag}, we conclude that: 

\begin{theorem}\label{thm:zzdag=coalg} $\mathsf{ZZDAG} \simeq \mathsf{Z}\text{-}\mathsf{COALG}$. 
\end{theorem}

We can in fact say a bit more about $\mathsf{Z}\text{-}\mathsf{COALG}$. Indeed, the forgetful functor $\mathsf{U}: \mathsf{DAG} \to \mathsf{CAT}$ is monadic \cite[Thm 2.1.11]{karvonen2019way}, which means that the algebras of the induced monad on $\mathsf{CAT}$ are precisely dagger categories, or in other words, the Eilenberg-Moore category of this induced monad is equivalent to $\mathsf{DAG}$. It turns out that $\mathsf{Z}: \mathsf{CAT} \to \mathsf{U}$ is comonadic, which means that $\mathsf{Z}\text{-}\mathsf{COALG}$ is equivalent to $\mathsf{CAT}$. In other words, the free dagger category monad on $\mathsf{CAT}$ is of effective descent type \cite[Sec 2]{mesablishvili2006monads}.

\begin{theorem} $\mathsf{Z}\text{-}\mathsf{COALG} \simeq \mathsf{CAT}$. Hence $\mathsf{Z}: \mathsf{CAT} \to \mathsf{U}$ is comonadic. 
\end{theorem}
\begin{proof} Abusing notation, it straightforward to workout that the comparison functor sends a category $\mathbb{B}$ to the $\mathsf{Z}$-coalgebra $(\mathsf{Z}[\mathbb{B}], \dagger, \mathcal{W}_\mathbb{B})$, as defined in Cor \ref{cor:free-canon-coalgebra}, while it sends a functor $\mathcal{F}: \mathbb{B}_1 \to \mathbb{B}_2$ to the $\mathsf{Z}$-coalgebra functor $\mathsf{Z}[\mathcal{F}]: (\mathsf{Z}[\mathbb{B}_1], \dagger, \mathcal{W}_{\mathbb{B}_1}) \to (\mathsf{Z}[\mathbb{B}_2], \dagger, \mathcal{W}_{\mathbb{B}_2})$ where $\mathsf{Z}[\mathcal{F}]$ is defined as in (\ref{def:ZF}), so $\mathsf{Z}[\mathcal{F}] = (\mathcal{N}_{\mathbb{B}_2} \circ \mathcal{F})^\sharp$. $\mathsf{Z}[\mathcal{F}]$ is indeed a $\mathsf{Z}$-coalgebra functor due to well-known results about adjunctions. We need to show that this comparison functor is equivalence, which we will do so by showing that it is full, faithful, and essentially surjective. 

For fullness consider a $\mathsf{Z}$-coalgebra functor $\mathcal{G}: (\mathsf{Z}[\mathbb{B}_1], \dagger, \mathcal{W}_{\mathbb{B}_1}) \to (\mathsf{Z}[\mathbb{B}_2], \dagger, \mathcal{W}_{\mathbb{B}_2})$. By the universal property of free dagger categories, we must have that $\mathcal{G} = (\mathcal{G} \circ \mathbb{N}_{\mathbb{B}_1})^\sharp$, or in other words, $\mathcal{G} = \mathsf{Z}\left[ \mathcal{G} \circ \mathbb{N}_{\mathbb{B}_1} \right]$. For faithfulness, consider functors $\mathcal{F}_1: \mathbb{B}_1 \to \mathbb{B}_2$ and $\mathcal{F}_2: \mathbb{B}_1 \to \mathbb{B}_2$ such that $\mathcal{Z}[\mathcal{F}_1] = \mathcal{Z}[\mathcal{F}_2]$. This first implies that on objects, $\mathcal{F}_1(A) = \mathcal{F}_2(A)$. While on maps note that this gives us that $[\mathsf{id}, \mathcal{F}_1(f)] = \mathcal{Z}[\mathcal{F}_1](f) = \mathcal{Z}[\mathcal{F}_2](f) = [\mathsf{id}, \mathcal{F}_1(f)]$. However this of course implies that $\mathcal{F}_1(f)=\mathcal{F}_2(f)$. Thus we get that $\mathcal{F}_1=\mathcal{F}_2$. Lastly essential surjectivity follows from the fact that by Prop \ref{prop:zalg-zigzag}, every $\mathsf{Z}$-coalgebra is a free dagger category thus by Cor \ref{cor:free-iso.2} dagger isomorphic to a canonical free dagger category, and it easy to check that this dagger isomorphism is in fact an isomorphism of $\mathsf{Z}$-coalgebras. 
\end{proof}

\bibliographystyle{plain}      
\bibliography{references}   

\end{document}